%% file: cyclicityone.tex
\documentclass[a4paper, twosided, 11pt]{amsart}
\pdfoutput=1

\input{defaultpreamble}

\title{On cyclicity one elliptic islands of the standard map}
\author{Jacopo De Simoi}
\address{Jacopo De Simoi\\
  Dipartimento di Matematica\\
  II Universit\`{a} di Roma (Tor Vergata)\\
  Via della Ricerca Scientifica, 00133 Roma, Italy.}
\email{{\tt desimoi@mat.uniroma2.it}}
\urladdr{\href{http://www.mat.uniroma2.it/~desimoi/redirect.php?refer=sm1}{http://www.mat.uniroma2.it/~desimoi}}

\begin{document}
\maketitle
\input{introduction}
\input{definitions}
\input{distribution}
\input{proof}
\input{appendix}

\bibliographystyle{abbrv}
\bibliography{sm}
\end{document}

%% file: defaultpreamble.tex
\usepackage[english]{babel}
\usepackage{amssymb}
\usepackage{amsfonts}
\usepackage{amsmath}
\usepackage{amsthm}
\usepackage{amscd}
\usepackage{mathrsfs}
\usepackage{eucal}
\usepackage{graphicx}
\usepackage{verbatim}
\usepackage{ushort}
\usepackage{enumerate}
\usepackage{amsbsy}
\usepackage{color}

\newcommand{\naturals}{\mathbb{N}}
\newcommand{\integers}{\mathbb{Z}}
\newcommand{\reals}{\mathbb{R}}
\newcommand{\preals}{\mathbb{R^+}}

\newcommand{\proj}{\mathbb{P}}

\newcommand{\torus}{\mathbb{T}}

\newcommand{\continuous}[1]{\mathcal{C}^{#1}}

\newcommand{\fa}{\forall\,}
\newcommand{\ex}{\exists\,}
\newcommand{\eps}{\varepsilon}
\newcommand{\defeq}{\doteqdot}
\newcommand{\deh}{\textup{d}}
\newcommand{\dpar}[2]{\frac{\partial #1}{\partial #2}}
\newcommand{\de}[2]{\frac{\textup{d} #1}{\textup{d} #2}}
\newcommand{\st}{\textup{ s.t. }}
\newcommand{\bigo}[1]{\mathcal{O}\left(#1\right)}
\newcommand{\id}{\textup{Id}}

\newcommand{\const}{\textup{Const}}
\newcommand{\prob}{\mathbb{P}}
\newcommand{\intr}[1]{\textup{int}\,#1}
\newcommand{\cl}[1]{\textup{cl}\,#1}
\newcommand{\diam}{\textup{diam}\,}
\newcommand{\dist}{\textup{dis{}t}}

\newtheorem{thm}{Theorem}[section]
\newtheorem{lem}[thm]{Lemma}
\newtheorem{prp}[thm]{Proposition}
\newtheorem{cor}[thm]{Corollary}
\theoremstyle{remark}
\newtheorem{rmk}[thm]{Remark}
\theoremstyle{definition}
\newtheorem{mydef}[thm]{Definition}

\newtheorem*{qst}{Question}
\newtheorem*{conjecture}{Conjecture}
\numberwithin{equation}{section}

%% file: introduction.tex
\input{standardmap.p}
\begin{abstract}
  We study abundance of a special class of elliptic islands for the standard family of
  area preserving diffeomorphism for large parameter values, i.e.\ far from the KAM regime.
  Outside a bounded set of parameter values, we prove that the measure of the set of
  parameter values for which an infinite number of such elliptic islands coexist is zero.
  On the other hand we construct a positive Hausdorff dimension set of arbitrarily large
  parameter values for which the associated standard map admits infinitely many elliptic
  islands whose centers accumulate on a locally maximal hyperbolic set.
\end{abstract}
\section{Introduction}
It is largely expected that symplectic dynamical systems generically present so-called
\emph{coexistence} of elliptic and stochastic behavior. What is usually meant by
coexistence is that the phase space can be divided in two invariant components of positive
volume, such that on the first component we have zero Lyapunov exponents and the second
one we have non-zero Lyapunov exponents. Pesin theory \cite{Pesin77} would then allow to
conclude that the metric entropy of the system is positive.

A well-understood nontrivial class of symplectic dynamical systems is given by completely
integrable systems; the phase space of such systems is foliated by invariant tori on which
the dynamics is conjugated to a rotation. This picture is however very fragile with
respect to perturbations; arbitrarily small perturbations of a completely integrable
system can indeed break some of these tori and chaotic behavior will then necessarily
arise (see for instance \cite{GL}). On the one hand, KAM theory implies that most
invariant tori of the completely integrable system persist also as invariant tori for the
perturbed map, albeit this is guaranteed only on the complement of a set of small but
positive measure. On the other hand, current techniques do not allow to prove that
hyperbolicity takes place in a positive volume set in the phase space; even arbitrarily
close to completely integrable systems we therefore expect to find some form of
coexistence, although a rigorous proof of this fact is far from being established.

On the other end of the spectrum, far away from completely integrable systems, one would
expect to have large ergodic components of positive volume; at the same time, for surface
diffeomorphisms, a conservative version of Newhouse phenomenon (described in
\cite{Duarte08}) implies that, close to a system presenting homoclinic tangencies, we have
abundance (from both the topological and parametric sense) of systems with elliptic
islands accumulating on hyperbolic sets. We can then expect also in this case the
coexistence picture to be prevalent.  Despite an impressive amount of numerical results
that confirm this intuition, it is amazingly difficult to prove that a given dynamical systems
presents coexistence; a continuous (but not $\continuous{1}$) example was given in
\cite{Wojtk}, whereas smooth (but not analytic) examples were constructed in \cite{Przy}
and \cite{Liverani}.

The Standard Family is a very natural and perhaps the most famous example of twist
mapping; it is given by a one-parameter family of area-preserving analytic twist
diffeomorphisms of the two dimensional torus and it is, in a sense, the natural candidate
to attempt a proof of the statements that we mentioned.  The Standard Family was first
numerically studied as a dynamical system in the late 60's by B. V.  Chirikov, in order to
study confinement of charged particles in mirror magnetic traps, and independently by
J. B. Taylor.  The Standard Family describes the dynamics of a mechanical system known as
the ``kicked rotor'' but it can be recognized in a vast collection of physical models; to
name a few, it describes ground states of the Frenkel-Kontorova model (see
\cite{frenkel,aubry}), it models dynamics of particles in accelerators
(\cite{Chirikov,Izraelev}) and dynamics of balls bouncing on a periodically oscillating
platform (\cite{Pust}) as well as cometary motions (\cite{Petrosky}).

The parameter value $\la$, sometimes referred to as the \emph{coupling constant}, is
usually normalized in the following way: if $\la=0$, then the Standard Map is completely
integrable; if $\la$ is \emph{small}, then we are in the KAM realm, whereas if $\la$ is
\emph{large} we are in the ``chaotic'' realm.  For $\la\not=0$ it is known that the map is
non-integrable (\cite{Fontich}), has horseshoes (\cite{Gelfreich, Goroff85}) and has
positive topological entropy (\cite{Knill96}).  However, despite a considerable amount of
work and effort it is still a wide open problem to prove whether or not there exists even
a single value of $\la$ for which the Standard Map is ergodic; moreover, the following in
principle less challenging questions are still far from being answered.
\begin{qst}[Sinai \cite{Sinai}]
  Is the metric entropy of the Standard Map positive for some values of $\la$? For $\la$
  belonging to some positive measure set? For all nonzero values of $\la$?
\end{qst}
In this paper we obtain some related results for large values of the parameter
$\la$. First, we recall some previous results: in \cite{Duarte94} it was proved that for
large enough parameters (i.e.\ far enough from the KAM regime) there exists a residual set
of parameter values for which the standard map $\sm$ has infinitely many elliptic islands
accumulating on a locally maximal hyperbolic set which fills the torus as
$\la\to\infty$. Although elliptic islands constitute an immediate obstruction to
ergodicity, their presence clearly does not prevent the positivity of the metric entropy of
the system, it does, however, imply that there exist regions of positive area where the metric
entropy is zero. We recall the following related conjecture:
\begin{conjecture}[Carleson \cite{Carleson}]
  The density of the set of parameters $\la$ for which at least one elliptic island exists
  tends to $0$ as $\la\to\infty$.
\end{conjecture}
More recently, in \cite{Anton10}, it was shown that for a residual set of large enough
parameters there exists a full Hausdorff dimension transitive invariant set where the map
has nonzero Lyapunov exponents.

In this paper we will consider a special type of elliptic islands: following
\cite{Kaloshin} we assign to all periodic points a quantity called \emph{cyclicity}; this
quantity counts how many points of the orbit visit a special region of the phase space
that is usually called \emph{critical set}. We perform a careful geometrical study of the
dynamics which leads to a firm understanding of the structure of the set of (large)
parameters for which elliptic islands of cyclicity one are present. This leads to our Main
Theorem, which can be loosely formulated as follows: we prove that for almost all
sufficiently large parameter values, the standard map has only finitely many elliptic
islands of cyclicity one; as a byproduct we show that Carleson conjecture holds for
cyclicity one elliptic islands. On the other hand we also prove that the set of parameter
values for which the standard map admits infinitely many cyclicity one elliptic islands
accumulating on a large locally maximal hyperbolic set is dense among sufficiently large
parameters and has Hausdorff dimension larger than $1/4$.

The geometric construction that is presented in this work seems quite promising for
extending our results to higher cyclicity; taking inspiration again from \cite{Kaloshin},
we expect results similar to our Main Theorem to hold true for orbits of cyclicity either
bounded or appropriately small with respect to the period; we expect to be possible to
prove such results without the ``unimaginable'' complication of understanding the
geometrical features of multiple passages through the critical region that seems to be
required by most previous attempts to deal with this problem.

This article is organized as follows: in section~\ref{s_definition} we recall the
definition of the Standard Family, we define the critical set and the notion of cyclicity,
in order to be able to state our Main Theorem. In section~\ref{s_distribution} we go
through the lengthy construction of the Markov structure on the phase space, which will
allow us to construct a locally maximal hyperbolic set and to classify cyclicity one
periodic points. We moreover study how the Markov structure varies by changing the
parameter, which will be fundamental to prove our results. Section~\ref{s_proof} is
devoted to the proof of the Main Theorem, which articulates in several stages. This
article also features a technical appendix, where we collect the proofs of three lemmata
that are necessary to obtain our results but are as well of independent interest.

{\footnotesize I thank with real pleasure Vadim Kaloshin, whom illustrated me the problem
  and sparked my curiosity on the subject; I would also like to express my gratitude to
  Bassam Fayad, who asked me the proverbial right question at the right time, and to
  Raphaël Krikorian and Carlangelo Liverani for their most welcome and helpful comments.
  I am also indebted to anonymous referees for providing me with accurate remarks and
  references which I was previously unfamiliar with.  It is also a privilege to thank
  the Fields Institute in Toronto, Canada for the excellent hospitality and working
  conditions provided in spring semester 2011.  This work has been partially supported by
  the Fondation Sciences Mathématiques de Paris and by the European Advanced Grant
  Macroscopic Laws and Dynamical Systems (MALADY) (ERC AdG 246953).
}


%% file: definitions.tex
\input{definitions.p}

\section{Definitions and statement of results}\label{s_definition}
In this section we will introduce the Chirikov-Taylor standard family and collect all definitions that are necessary to state our Main Theorem.

Let $\torus\defeq\reals/\integers$; we write the standard family $\sm:\torus^2\to\torus^2$, $\la\in\reals$ as follows:
\begin{equation}\label{eq_defSM}
  \sm:(x,y)\mapsto(y,-x+2y+\la \dot\phi(y))\mod\integers^2,
\end{equation}
where we fix for definiteness $\phi(\cdot)\defeq (2\pi)^{-1}\sin(2\pi \cdot)$. The
standard map $\sm$ is an area-preserving diffeomorphism; moreover $\sm$ is reversible with
$r:(x,y)\mapsto(y,x)$ as a reversor, i.e.:
\begin{align*}
  r^2&=\textrm{Id}& r\sm&=\sm^{-1}r.
\end{align*}
Finally notice that $(x,y)\mapsto(x+1/2,y+1/2)\mod \integers^2$ conjugates $\sm$ to
$\smaux{-\la}$, so that we will henceforth assume $\la\in\preals$.

An elementary inspection of the dynamics of $\sm$, for large values of the parameter $\la$,
leads to the natural definition of a \emph{critical set} $\crit(\la)$ (see below for
details but also e.g. \cite{Duarte94,Anton10,Goroff85}): geometrically, $\crit(\la)$ is
given by a $\la^{-1/2}$-thin vertical strip around each of the two critical points of
$\dot\phi(x)$. Dynamically, it is characterized by the property that any orbit of $\sm$
which does not intersect $\crit(\la)$ is hyperbolic. In particular:
\begin{prp}\label{p_hyperbolicSet}
  For any $\la$ sufficiently large there exists a locally maximal hyperbolic set
  $\phypset(\la)\subset\torus^2\setminus\crit(\la)$ satisfying the following properties:
  \begin{enumerate}[(a)]
  \item $\phypset(\la)$ is dynamically increasing i.e.\ for all $\la>\la^*$ there exists an
    injection $\iota(\la^*,\la):\phypset(\la^*)\to\phypset(\la)$; for any fixed
    $p\in\phypset(\la^{*})$ we have that $\iota(\la^*,\la)p$ is smooth\footnote{In fact
      the dependance is actually analytic, see e.g. \cite{dlLMM}} in $\la$;
  \item the set $\phypset(\la)$ is $\const\,\la^{-1/2}$-dense in $\torus^2$.
  \end{enumerate}
\end{prp}
The previous proposition corresponds to theorem A in \cite{Duarte94} and theorem 1 in
\cite{Anton10} and simply amounts to the construction and description of a hyperbolic set
well suited for our needs. It is then natural to give the following

\begin{mydef}[Cyclicity] \label{d_cyclicity} Fix $\la$; let $\Theta=\{p_0,\cdots,p_N =
  p_0\}$ be a periodic orbit for $\sm$ of least period $N$; define $s(\Theta)$ the
  \emph{cyclicity} of the orbit $\Theta$ as:
  \[
  s(\Theta)\defeq \card(\Theta\cap\crit(\la)).
  \]
  We define the \emph{cyclicity of a periodic point} as the cyclicity of its orbit and the
  \emph{cyclicity of an elliptic island} as the cyclicity of its center.
\end{mydef}
The dynamical properties of the critical set (see proposition~\ref{l_coneInvariance}
below) imply that all cyclicity $0$ periodic points are hyperbolic; thus if $\Theta$ is an
elliptic periodic orbit, necessarily $s(\Theta)\geq 1$. In this work we focus our efforts
on cyclicity one elliptic periodic orbits; our results are described in the following
\begin{mthm}
  \label{t_cyclicityOneMeasure}
  There exists $\la_0>1$ such that:
  \begin{enumerate}[(a)]
  \item the set of parameters $\la\in[\la_0,\infty)$ such that the standard map $\sm$ has
    infinitely many elliptic islands of cyclicity $1$ is a null set for Lebesgue measure;
  \item there exists a residual set $\fatset\subset[\la_0,\infty]$ of Hausdorff dimension
    at least $1/4$ such that if $\la\in\fatset$ the standard map $\sm$ has infinitely many
    elliptic islands of cyclicity $1$ whose centers accumulate on $\phypset(\la)$.
  \end{enumerate}
\end{mthm}
Before we venture in giving the precise definition of the critical set and proving its
basic properties, let us add a few comments on our results. In \cite{TLY}, the authors
studied periodic attractors of cyclicity one (which they called \emph{simple periodic
  attractors}) in a standard example of dissipative diffeomorphism and prove a result
which is indeed similar in spirit and techniques to item (a) of our Main Theorem, with
elliptic islands replaced by sinks.  Further results in the dissipative case were later
obtained in \cite{Gonch93,Gonch97} for low cyclicity (in this case called \emph{loop
  count}), and in \cite{Kaloshin} for high cyclicity.  We hope that the main result of our
paper could be the first step on a similar path in the realm of conservative
diffeomorphisms.  Notice moreover that item (b) of our Main Theorem is indeed analogous
to the main result in \cite{Duarte94}; our approach, though, is very different: in
particular we give a constructive proof and explicitly show where elliptic islands of
cyclicity one appear in both phase space and parameter space. This allows us to improve
Duarte's result by not only obtaining a lower bound on the Hausdorff dimension of the
parameter set, but also proving that his theorem holds true even considering solely
cyclicity one elliptic islands.

In order to study hyperbolic properties of the dynamics it is convenient to consider slope
fields rather than vector fields; we thus introduce a few basic notions and write some
useful formulae for future reference. For fixed $\la\in\preals$ define a
$\la$\emph{-adapted slope field} as a function
$\slope{}:\torus^2\to\reals\proj=\reals\cup\{\infty\}$ such that $\la\slope{}$ is the
slope of a line field; by convention if $\slope{}(p)=\infty$, then the line field has a
vertical tangent in $p$. Let $\slope{}$ be a $\la$-adapted slope field; given a $\torus^2$
diffeomorphism $g:(x,y)\mapsto(g_x,g_y)$, we obtain the push-forward $g_*\slope{}$ as the
$\la$-adapted slope field given by:
\begin{equation}\label{eq_pushForwardSlope}
  [g_*\slope{}](p) \defeq \frac{1}{\la}\frac{\partial_xg_y+\la\partial_yg_y\slope{}}{\partial_xg_x+\la\partial_yg_x\slope{}}(g^{-1}p)
\end{equation}
with the convention that if $[\partial_xg_x+\la\partial_yg_x\slope{}](g^{-1}p)=0$, we let $[g_*\slope{}](p)=\infty$
. A $\la$-adapted slope field $\slope{}$ is said to be $\continuous{r}$-smooth  if for all $p\in\torus^2$ either $\slope{}$ is $\continuous{r}$-smooth in $p$ or $r_*\slope{}$ is $\continuous{r}$-smooth in $rp$.
Let $\slope{}$ be a $\continuous{1}$ $\la$-adapted slope field and $\psi$ a smooth function on $A\subset\torus^2$; then for $p\in A$ denote by $\parSlope{\slope{}}\psi\,(p)$ the directional derivative of the map $\psi$ at point $p$ in the direction given by $\slope{}(p)$, normalized as follows:
\[
\parSlope{\slope{}}\psi\,(p) \defeq \begin{cases}
  \partial_y\psi{}(p) & \textrm{if } \slope{}(p)=\infty\\
  \partial_x\psi{}(p)+\la\slope{}(p)\cdot\partial_y\psi{}(p) &\textrm{otherwise}.
\end{cases}
\]
We will quite often consider the expression $\parSlope{\slope{}}\slope{}$; this object is pushed forward by a $\continuous{2}$-diffeomorphism $g:\torus^2\to\torus^2$ by means of the following definition:
\begin{equation*}
  g_*\parSlope{\slope{}}{\slope{}} \defeq \parSlope{g_*\slope{}}[g_*\slope{}]
\end{equation*}
Here follow the expressions describing the action of $r$ on a $\la$-adapted slope field  $\slope{}$ and on a directional derivative $\parSlope{\slope{}}\slope{}$:
\begin{subequations}\label{eq_explicitReversor}
  \begin{align}
    r_*\slope{}\,(p) &= \frac{1}{\la^{2}\slope{}}(rp)\\
    r_*\parSlope{\slope{}}\slope{}\,(p) &= \frac{\parSlope{\slope{}}\slope{}}{\la^{3}\slope{}^{3}}(rp)
  \end{align}
\end{subequations}
Let $\slope{0}$ be the constant $\infty$ slope field, then for $k\in\integers$ define the \emph{$k$-th reference $\la$-adapted slope field} $\slope{k}$ by pushing forward $\slope{0}$ by $\sm^k$, i.e.:
\begin{equation}\label{eq_definitionReferenceSequence}
  \slope{k}\defeq{\sm^k}_*\slope{0}.
\end{equation}
Using \eqref{eq_explicitReversor} and \eqref{eq_definitionReferenceSequence} it is easy to prove the following relation:
\[
\slope{k}(p)=\frac{1}{\la^2\slope{-(k+1)}}(rp)
\]
By means of elementary computations we find:
\begin{subequations}\label{e_oneDefinitions}
  \begin{align}
    \slope{1}(x,y) &= \ddot\phi(x) + \smash{\frac{2}{\la}} \\
    \parSlope{\slope{1}}\slope{1}(x,y) &= \dddot\phi(x) \\
    \slope{-1}(x,y) &= 0\\
    \parSlope{\slope{-1}}\slope{-1}(x,y) &= 0.
  \end{align}
\end{subequations}
We write below, for future reference, a few useful formulae which follow directly from the
definitions and from equations \eqref{eq_explicitReversor}; recall that a superscript $*$
denotes the pull-back i.e.\ the push-forward by the inverse map:
\begin{subequations}
  \label{eq_pushForwards}
  \begin{align}
    {\sm}_*\slope{}\,(p) &= \slope{1}(p) - \frac{1}{\la^2 \slope{}}(\sm^{-1}p)\label{eq_pfh}\\
    {\sm}_*\parSlope{\slope{}}\slope{}\,(p) &= \parSlope{\slope{1}}\slope{1}(p) - \frac{\parSlope{\slope{}}\slope{}}{\la^3 \slope{}^3}(\sm^{-1}p)\\
    {{\sm}^*}\slope{}\,(p) &= \frac{1}{\la^2 (\slope{1}-\slope{})}(\sm p)\label{eq_pbh}\\
    {{\sm}^*}\parSlope{\slope{}}\slope{}\,(p) &= \frac{\parSlope{\slope{}}\slope{}}{\la^3 (\slope{1}-\slope{})^3}(\sm p)
  \end{align}
\end{subequations}
Notice that, denoting $(x',y')=\sm(x,y)$, we have:
\begin{align}\label{eq_expandingRelations}
  \parSlope{\slope{}}{x'}\,(x,y) &= \la\slope{}(x,y) & \parSlope{\slope{}}{x}\,(x',y') &= \la(\slope{1}(x',y')-\slope{}(x',y')),
\end{align}
which in particular implies that integral curves of $\slope{1}$ are forward-expanded in the $x$ direction by a factor $\la\slope{1}$, and that integral curves of $\slope{-1}$ are backward-expanded in the $x$ direction by the same factor $\la\slope{1}$.

%
We are now ready to introduce the \emph{critical set} $\crit$; we expect appropriate cone
conditions and distortion bounds (see proposition~\ref{l_coneInvariance}) to hold on the
complement of $\crit$ and its definition actually depends on our exact requirements. Our
critical set will not satisfy the Markov property,
i.e.\ $\sm\partial\crit\not\subset\partial\crit$; to compensate for this issue, it is
standard practice to define an \emph{inner} critical set $\icrit$ and to employ either
$\crit$ or $\icrit$ according to the situation. This essentially allows to ``blur'' the
boundary of the critical set and makes the lack of Markov property not important; 
let us introduce $\pix$ and $\piy$ the projections on the first and second coordinate
respectively.
\begin{mydef}
  Assume $\la>1$, let $\tness\in\preals$, define the \emph{critical set} $\crit$ and the \emph{inner critical set} $\icrit$ as follows:
  \begin{align*}
    \crit(\la;\tness)&\defeq\{p\in\torus^2\st|\la\slope{1}(p)|<\tness\cdot\la^{1/2}\},&
    \icrit(\la;\tness)&\defeq\crit(\la;\tness/10)
    \intertext{
      Define also the \emph{projected critical set} and the \emph{projected inner critical set}:}
    \pcrit(\la;\tness)&\defeq\{\xi\in\torus\st |\ddot\phi(\xi) +2/\la| < \tau\cdot\la^{-1/2}\},& \picrit(\la;\tness)&\defeq \pcrit(\la;\tness/10).
  \end{align*}
\end{mydef}
We fix once and for all $\tness=80\pi$ and drop it from the notation; furthermore, we will always assume $\la$ to be large enough to ensure that
\begin{equation}\label{e_icritUniformEstimate}
  |\la\slope{1}(p)|<\frac{1}{20}\la\ \text{if }p\in\crit(\la);
\end{equation}
in particular $\pcrit$ has two connected components which we denote by $\cpcrit{+}$ and $\cpcrit{-}$ according to the sign of $\dddot\phi$ on each of them; the same applies to $\picrit$ so that we can define $\cpicrit\pm$. Notice that by \eqref{e_oneDefinitions} we have that $\crit = \pix^{-1}(\pcrit)$ and $\icrit = \pix^{-1}(\picrit)$. In general, given $\pcrit_*\subset\pcrit$ we call $\crit|_{\pcrit_*} = \pix^{-1}(\pcrit_*)$ the \emph{restriction of the critical set on $\pcrit_*$}; we then define $\ccrit\pm = \crit|_{\cpcrit\pm}$ and  $\cicrit\pm = \crit|_{\cpicrit\pm}$.
Our choice of $\phi$ easily implies that:
\begin{equation}\label{e_minimumDistance}
  \dist(\xi,\cl\picrit(\la))>\la^{-1/2}\text{ for any }\xi\not\in\pcrit(\la);
\end{equation}
in fact by \eqref{e_icritUniformEstimate} and definition of $\phi$ we obtain that $|\dddot\phi(\xi)|>2\pi^2$ for $\xi\in\pcrit(\la)$; consequently:
\[
\dist(\xi,\cl\picrit(\la))>\left(1-\frac{1}{10}\right)\frac{\tness\la^{-1/2}}{2\pi^2}>\la^{-1/2}.
\]
We now state properties of the dynamics on the complement of the critical set; for any fixed $\la$ define the following cone fields:
\begin{align*}
  \cone{\unstable}&\defeq{\sm}_*\{\slope{}\st |\slope{}|>2\la^{-1/2}\}&
  \cone{\stable}&\defeq\{\slope{}\st |\slope{}|<2\la^{-1/2}\}.
\end{align*}
As usual, we say that a curve is unstable (respectively stable) if the slopes at all points belong to the unstable (respectively stable) cone.
\begin{prp} \label{l_coneInvariance}
  For large enough $\la$, the cone fields $\cone{\unstable}$ and $\cone{\stable}$ are respectively forward and backward invariant outside $\icrit$; namely, if $p\not\in\icrit$, $\slope{}(p)\in\cone{\unstable}(p)$ and $\slope{}'(p)\in\cone{\stable}(p)$ we have:
  \begin{align*}
    {\sm}_*\slope{}(\sm p)&\in\cone{\unstable}(\sm p) &{\sm}^*\slope{}'(\sm^{-1} p)&\in\cone{\stable}(\sm^{-1} p);
  \end{align*}
  in particular $\bigcap_{k\in\integers}\sm^{k}(\torus^2\setminus\icrit(\la))$ is a hyperbolic set for $\sm$.
  Moreover if $(x',y')=\sm(x,y)$, then if $\slope{}\in\cone{\unstable}$ and $(x,y)\not\in\icrit$ we have
  \begin{subequations}\label{eq_expansionRate}
    \begin{equation}
      |\parSlope{\slope{}}x'(x,y)|>2\la^{1/2};
    \end{equation}
    on the other hand, if $\slope{}'\in\cone{\stable}$ and $(x',y')\not\in\icrit$ we have
    \begin{equation}
      |\parSlope{\slope{}'}x(x',y')|>2\la^{1/2}.
    \end{equation}
  \end{subequations}
  Finally, the following \emph{bounded distortion condition} holds outside $\icrit$:
  \begin{equation}\label{e_smallDistortion}
    \left|\frac{\la\parSlope{\slope{1}}{\slope{1}}(p)}{(\la\slope{1}(p))^2}\right|
    \leq\frac{1}{16}\text{ if }p\not\in\icrit.
  \end{equation}
\end{prp}
\begin{proof}
  Notice that by \eqref{eq_pfh} we have $\cone{\unstable}=\{\slope{}\st|\slope{}-\slope{1}|<1/2\cdot\la^{-3/2}\}.$ To prove forward invariance it is sufficient to prove that if $p\not\in\icrit$, then:
  \[
  \cone{\unstable}(p)\subset\{\slope{}\st |\slope{}|>2\la^{-1/2}\},
  \]
  but since $p\not\in\icrit$, then $|\slope{1}(p)|\geq 8\pi\cdot\la^{-1/2}$, hence if $\la$ is large enough we can conclude.

  Likewise, by \eqref{eq_pbh} we have that if $p\not\in\icrit$ and $\slope{}\in\cone{\stable}$, then $|{\sm}^*\slope{}(\sm^{-1}p)|<(1/4\pi)\cdot\la^{-3/2}$, hence if $\la$ is large enough we can again conclude.  Estimates \eqref{eq_expansionRate} immediately follow from \eqref{eq_expandingRelations} and \eqref{e_smallDistortion} is elementary given the definition of $\icrit$ and expressions \eqref{e_oneDefinitions}.
\end{proof}
We conclude this section by collecting a few useful estimates
\begin{mydef}
  Let $k>1$ and define respectively the \emph{$k$-forward regular set} and \emph{$k$-backward regular set} as the following open sets:
  \begin{align*}
    \fnicePoints_k(\la) &=\torus^2\setminus \bigcup_{j=1}^{k-1} \sm^{-j} \cl\icrit(\la)&
    \bnicePoints_k(\la) &=\torus^2\setminus \bigcup_{j=1}^{k-1} \sm^{j} \cl\icrit(\la).
  \end{align*}
\end{mydef}

\begin{prp} \label{l_estimateReferenceSequence}
  Fix $\la$; let $k>1$ and take $p\in\cl\fnicePoints_k(\la)$ and $q\in\cl\bnicePoints_k(\la)$; then for $0<j<k$ the following estimates hold:
  \begin{subequations}\label{eq_estimateHH}
    \begin{align}
      |\slope{k}(q) - \slope{j}(q)| &\leq \smash{\prod_{l=1}^{j-1}}(\la\slope{k-l}(\sm^{-l}q))^{-2}\cdot\notag\\
      &\qquad\qquad \cdot\frac{1}{\la^2|\slope{k-j}(\sm^{-j}q)|}\cdot(1+\bigo{\la^{-\alpha}});\label{eq_estimateHkHj}\\
      |\slope{-k}(p) - \slope{-j}(p)| &\leq \smash{\prod_{l=1}^{j-1}}(\la\slope{-k+(l-1)}(\sm^{l-1}p))^2\cdot\notag\\&\qquad\qquad \cdot|\slope{-k+(j-1)}(\sm^{j-1}p)|\cdot(1+\bigo{\la^{-\alpha}}).\label{eq_estimateH-kH-j}
    \end{align}
  \end{subequations}
  Moreover:
  \begin{subequations}\label{eq_estimatedH}
    \begin{align}
      |\parSlope{\slope{k}}\slope{k}(q)-\parSlope{\slope{1}}\slope{1}(q)|& <2|\la\slope{k-1}(\sm^{-1}q)|^{-3}    \label{eq_estimatedHk}\\
      \label{eq_estimatedH-k}
      |\parSlope{\slope{-k}}\slope{-k}(p)-\parSlope{\slope{-1}}\slope{-1}(p)|&<2|\la\slope{-k}(p)|^{-3}
    \end{align}
  \end{subequations}
  \begin{subequations}
    \begin{align}
      |\parSlope{\slope{0}}{\slope{k}}(q)|&<2|\la\slope{k-1}(\sm^{-1}q)|^{-2}\label{eq_verticalEstimateFordHk}\\
      |\parSlope{\slope{0}}{\slope{-k}}(p)|&<2|\la\slope{-k}(p)|^{2}\label{eq_verticalEstimateFordH-k}.
    \end{align}
  \end{subequations}
  In particular we can write:
  \begin{align}
    |\slope{k}(q)-\slope{j}(q)|&=\bigo{\la^{-j\smash{-1/2}}};
    \tag{\ref{eq_estimateHkHj}'}
    \label{eq_forwardEstimateForHk}\\
    |\slope{-k}(p)-\slope{-j}(p)|&=\bigo{\la^{-j\smash{-1/2}}};
    \tag{\ref{eq_estimateH-kH-j}'}
    \label{eq_backwardEstimateForHk}\\
    |\parSlope{\slope{k}}\slope{k}(q)-\parSlope{\slope{1}}\slope{1}(q)|&=\bigo{\la^{-3\smash{/2}}};
    \tag{\ref{eq_estimatedHk}'}
    \label{eq_forwardEstimateFordHk}\\
    |\parSlope{\slope{-k}}\slope{-k}(p)-\parSlope{\slope{-1}}\slope{-1}(p)|&=\bigo{\la^{-3\smash{/2}}};
    \tag{\ref{eq_estimatedH-k}'}
    \label{eq_backwardEstimateFordHk}\\
    |\parSlope{\slope{0}}\slope{k}(q)|&=\bigo{\la^{-1}};
    \tag{\ref{eq_verticalEstimateFordHk}'}
    \label{eq_forwardVerticalEstimateFordHk}\\
    |\parSlope{\slope{0}}\slope{-k}(p)|&=\bigo{\la^{-1}}.
    \tag{\ref{eq_verticalEstimateFordH-k}'}
    \label{eq_forwardVerticalEstimateFordH-k}
  \end{align}
\end{prp}
\begin{proof}
  The proof follows from the definitions using expressions \eqref{eq_pushForwards}.
\end{proof}


%% file: distribution.tex
\input{definitions.p}
\input{domains.p}
\input{distribution.p}
\section{Dynamically adapted covering}\label{s_distribution}
In this section we define a special covering of the phase space which carries relevant dynamical information; the construction is similar to the construction of a Markov partition, and in fact will produce, as a byproduct, locally maximal hyperbolic sets endowed with a Markov partition.
This section is divided into four parts, in the first part we define, for each fixed $\la$, a covering of the complement of the inner critical set; once we are done with the definitions,  we dynamically refine the covering in order to define hyperbolic sets in the second part and critical domains in the third part. Finally in the fourth part we study the how these objects behave by varying the parameter value.


\subsection{Markov structure}
In this section we describe the geometrical construction which yields the Markov structure; the construction is rather simple in itself, although it needs quite cumbersome notation to be properly defined. The reader will hopefully find the pictures helpful to more easily follow the definitions.

The complement of the inner critical set is given by two connected components,
\[
\torus^2\setminus\icrit(\la) = \cstr+(\la)\sqcup\cstr-(\la);
\]
let $\cpstr\pm(\la)\defeq\pix\,\cstr\pm(\la)$. The superscript index $+$ or a $-$ is chosen according to the sign of $\slope{1}$ on each component;  on the other hand recall that the superscript in $\crit^{\pm}$ had been chosen according to the sign of $\dddot\phi$.
It is convenient to introduce a unified notation for the sets $\crit$, $\cstr{}$ and their projections; we do so in the following way
\begin{align*}
  \gstr{\cs}{\pm}(\la)&=\cl\ccrit\pm(\la)  &\gpstr{\cs}{\pm}(\la)&=\cl\cpcrit\pm(\la)\\
  \gstr{\ds}{\pm}(\la)&=\cstr\pm(\la)  &\gpstr{\ds}{\pm}(\la)&=\cpstr\pm(\la)
\end{align*}
Fix the following points on $\torus$:
\begin{align*}
  \gxi\cs-&=0&\gxi\ds-&=\frac{1}{4}&\gxi\cs+&=\frac{1}{2}&\gxi\ds+&=\frac{3}{4},
\end{align*}
along with the corresponding vertical lines $\gvl\is\pm=\pix^{-1}(\gxi\is\pm)$, for $\is\in\{\cs,\ds\}$. It follows then from the definitions that $\gxi\is\pm\in\gpstr\is\pm(\la)$ (see figure~\ref{f_critical}).
\begin{figure}[!ht]
  \begin{center}
    \def\svgwidth{12cm}
    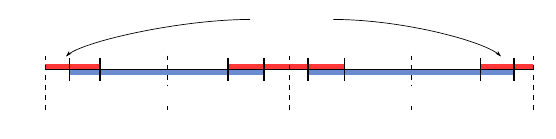
  \end{center}
  \caption{Sketch of the geometry of the sets $\gpstr{z}{s}$; the corresponding sets $\gstr{z}{s}$ are given by vertical strips above $\gpstr{z}{s}$. Notice that for large values of $\la$ the sets $\gpstr{c}{s}$ are much smaller than $\gpstr{d}{s}$; for sake of clarity the picture is not to scale.}
  \label{f_critical}
\end{figure}
For $\s,\si,\sii \in\{+,-\}$ and $\is,\js\in\{\cs,\ds\}$ define the following sets (see figures~\ref{f_squares} and~\ref{f_rectangles} for an illustration):
\begin{align*}
  \gsqu{\is}{\js}{\s}{\si}(\la) &=\gstr\is\s(\la)\cap\sm^{-1}\gstr\js\si(\la)\\
  \grect\is\sii\s\js\si(\la)&=\sm\gsqu\is\ds\sii\s(\la)\cap\gsqu\ds\js\s\si(\la).
\end{align*}
\begin{figure}[!ht]
  \footnotesize
  \begin{center}
    \begin{minipage}{5cm}
      \def\svgwidth{4.5cm}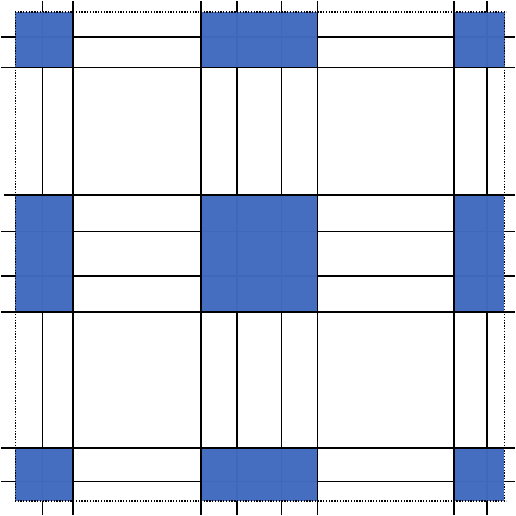\\[0.35cm]
      \def\svgwidth{4.5cm}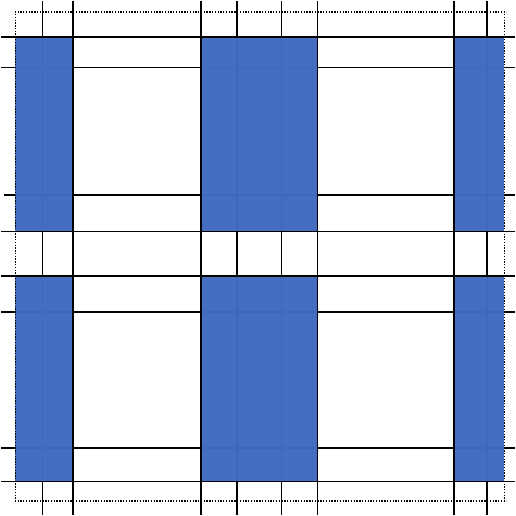
    \end{minipage}
    \begin{minipage}{5cm}
      \def\svgwidth{4.5cm}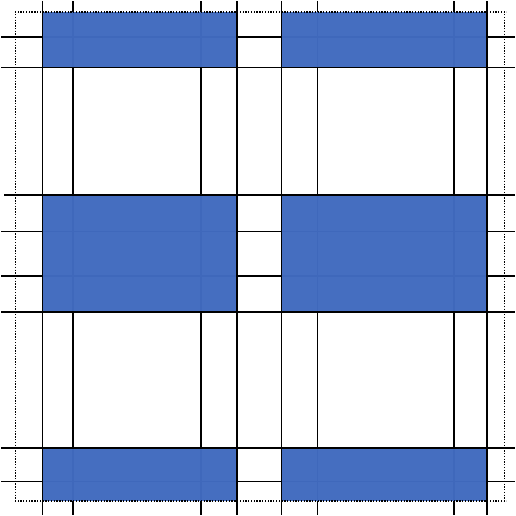\\[0.35cm]
      \def\svgwidth{4.5cm}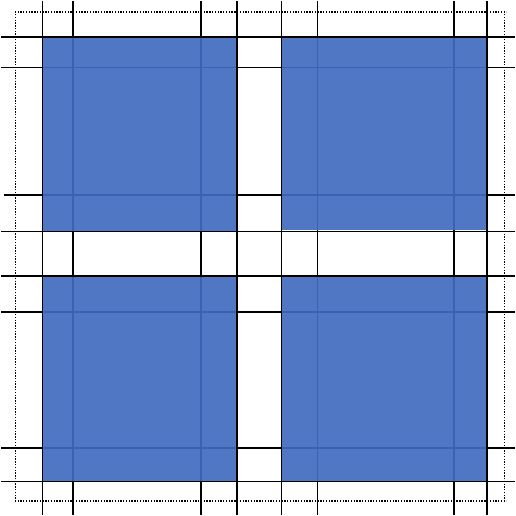
    \end{minipage}
  \end{center}
  \caption{Structure of the sets $\gsqu{\is}{\js}{\s}{\si}$; each of the four picture represents $\torus^2$ and sketches the geometry of the associated set.}
  \label{f_squares}
\end{figure}
\begin{figure}[!ht]
  \footnotesize
  \begin{center}
    \begin{minipage}{6.0cm}
      \def\svgwidth{5.5cm}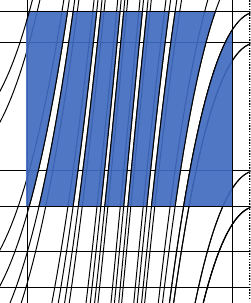\\[0.35cm]
      \def\svgwidth{5.5cm}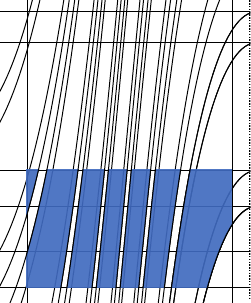
    \end{minipage}
    \begin{minipage}{6.0cm}
      \def\svgwidth{5.5cm}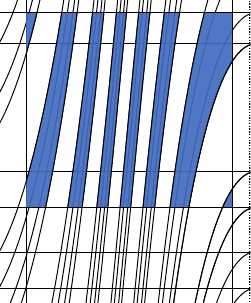\\[0.35cm]
      \def\svgwidth{5.5cm}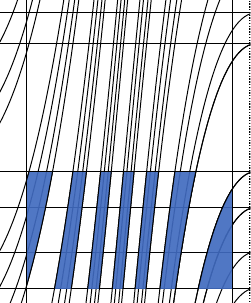
    \end{minipage}
  \end{center}
  \caption{Structure of the sets $\grect{\is}{\sii}{\s}{\js}{\si}$; in the pictures we set for definiteness $\s=\si=+$, the other combinations of signs give similar geometrical structures. Each of the four pictures represents a portion of $\torus^2$ (a neighborhood of the top-right quadrant in figure~\ref{f_squares}) and sketches the geometry of the associated set. Again for sake of clarity we have chosen $\la$ not too large.}
  \label{f_rectangles}
\end{figure}
By construction, the boundary of each set $\rect$ is given by pieces of integral curves of $\slope{-1},\slope{0}$ and $\slope{1}$; we call \emph{stable boundary} and denote by $\partial_\stable\rect$ the subset of $\partial\rect$ given by integral curves of $\slope{-1}$; likewise the \emph{unstable boundary}, denoted by $\partial_\unstable\rect$, refers to the subset of $\partial\rect$ given by integral curves of $\slope{1}$; in particular:
\begin{align*}
  \partial_\stable\grect\is\sii\s\js\si(\la) &= \partial\grect\is\sii\s\js\si(\la) \cap \partial\smi\gstr\js\si(\la)\\
  \partial_\unstable\grect\is\sii\s\js\si(\la) &= \partial\grect\is\sii\s\js\si(\la) \cap \partial\sm\gstr\is\sii(\la).
\end{align*}
Each set $\rect(\la)$ is the union of compact connected components $\rect^*$: we define $\partial_\stable$ and $\partial_\unstable$ on connected components $\rect^*$ by restriction:
\begin{align*}
  \partial_\stable\rect^*&=\partial_\stable\rect\cap\rect^*&\partial_\unstable\rect^*&=\partial_\unstable\rect\cap\rect^*.
\end{align*}

Let $\rect^*$ be a connected component of $\grect\is\sii\s\js\si$; it is convenient to define the \emph{forward} and \emph{backward bases} of $\rect^*$ as follows:
\begin{align*}
  \bbase{\rect^*}&=\gpstr\is\sii(\la)
  & \fbase{\rect^*}&=\gpstr\js\si(\la);
\end{align*}
notice that by definition we have $\dpix{-1}\rect^*\subset\bbase{\rect^*}$ and $\dpix{1}\rect^*\subset\fbase{\rect^*}$.
We introduce the map $\homeo(\la):\torus^2\to\torus\times\torus$ given by $\homeo=(\bbpi,\fbpi)$; $\homeo$ is differentiable and:
\[
\deh\homeo(p,\la) = \matrixtt{\la\slope{1}(p,\la)}{-1}{0}{1}.
\]
\begin{prp}\label{p_alternativeAdmissible}
  Fix $\la$ and let $\rect^*$ be a connected component of $\rect(\la)$; the following statements are equivalent:
  \begin{enumerate}[(a)]
  \item $\partial\rect^*\subset\partial_\stable\rect^*\cup\partial_\unstable\rect^*$;
  \item $\partial\rect^* \cap \partial\icrit \subset \partial_\stable\rect^* \cap \partial_\unstable\rect^*$    
  \item $\homeo(\la)$ is a diffeomorphism between $\rect^*$ and $\bbase{\rect^*}\times\fbase{\rect^*}$.
  \end{enumerate}
  If any of the previous properties holds, then $\rect^*$ is said to be an \emph{admissible component} or a \emph{rectangle}. Moreover, a rectangle $\rect^*$ is said to be \emph{proper} if $\pix\rect^*\subset\torus\setminus\cl\pcrit(\la)$.
\end{prp}

\begin{proof}
  Notice that by definition of $\rect$ we have that:
  \[
  \partial\rect^* \subset \partial\icrit \cup \partial_\stable\rect^* \cup \partial_\unstable\rect^*,
  \]
  hence \emph{(b)} implies \emph{(a)}.
  On the other hand, since $\slope{1}$ and $\slope{-1}$ never attain the value $\infty$ it is also clear that \emph{(a)} implies \emph{(b)}

  We now prove that \emph{(a)} implies \emph{(c)}: we know that the map $\homeo(\la)$ is a local diffeomorphism on $\rect^*$ since $\slope{1}$ is nonzero in some neighborhood of $\rect^*$;
  therefore, it suffices to prove that $\homeo\rect^*=\bbase{\rect^*}\times\fbase{\rect^*}$, since then $\homeo\rect^*$ would be simply connected, which implies injectivity of $\homeo$. For ease of notation let $\bbasei{}=\bbase{\rect^*}$ and $\fbasei{}=\fbase{\rect^*}$; consider $q\in\rect^*$ and fix any $\xi\in\bbasei{}$, $\xi'\in\fbasei{}$; denote by $\Gamma_q$ the leaf of the $\slope{1}$-foliation of $\rect^*$ containing $q$; \emph{(a)} implies that $\Gamma_q$ joins the two opposite sides of $\partial_\stable\rect^*$. Hence, by proposition~\ref{l_estimateReferenceSequence} we have that $\sm\Gamma_q$ is the graph of a function over $\fbasei{}$ and, as such, it will intersect the vertical line $\pix^{-1}(\xi')$ in a unique point that we denote by $\sm r$.  Let now $\Gamma'_r$ be the leaf of the $\slope{-1}$-foliation of $\rect^*$ containing $r$; again $\Gamma'_r$ must join the two opposite sides of $\partial_\unstable\rect^*$, therefore $\sm^{-1}\Gamma'_r$ is the graph of a function over $\bbasei{}$ which intersects the vertical line $\pix^{-1}(\xi)$ in a unique point that we denote by $\sm^{-1}p$; then, by construction, $p\in\rect^*$ and $\homeo(p)=(\xi,\xi')$. Since $\xi$ and $\xi'$ were arbitrary, $\homeo$ is surjective.

  Assume \emph{(c)}, then we have $\partial\rect^* = \homeo^{-1}(\partial(\bbase{\rect^*}\times\fbase{\rect^*}))$, which is \emph{(a)}.
\end{proof}

We now proceed to describe a canonical indexing procedure for admissible components; for $\s,\si,\sii\in\{+,-\}$ and $\is,\js\in\{\cs,\ds\}$ define the following sets:
\[
\gpoint\is\sii\s\js\si(\la) = \sm\gvl\is\sii\cap\gstr\ds\s(\la)\cap\smi\gvl\js\si
\]
Proposition~\ref{l_estimateReferenceSequence} implies that $\slope{1}$ and $\slope{-1}$ are always transverse in $\gstr\ds\s(\la)$; consequently, each set $\point$ is given by a disjoint union of points; we assign to each of such points an index $k$ as follows:
\[
\gpoint\is{\sii}{\s}\js{\si}\ni(x,y) \mapsto k = \s(\la\dot\phi(x)+2(x-\gxi\ds\s))+\gxi\js\si-\gxi\is\sii;
\]
it is immediate, given the definition of $\point$ and $\sm$, to check that $k\in\integers$. We denote each of these points by $\gppointC\is\sii\s\js\si k$;  by proposition~\ref{p_alternativeAdmissible}, every admissible component $\rect^*$ has to contain exactly one point in the appropriate set $\point$; we assign to each admissible component the index of the corresponding point: each rectangle will then be uniquely identified by the notation $\grectC\is\sii\s\js\si k(\la)$. Before proceeding any further it is convenient to define a shorthand notation:
\begin{mydef}
  Define the \emph{basic alphabet} $\alphabet$ to be the set of all sextuples $\letter$ of the form
  \[
  \letter=[\sixind\is\sii\s\js\si k]
  \]
  for $\s,\si,\sii\in\{+,-\}$, $\is,\js\in\{\cs,\ds\}$ and $k\in\integers$. Also define the \emph{restricted alphabets} $\ralphabet{\is}{\js}$ as the subset of $\alphabet$ given by all symbols with prescribed $\is$ and $\js$.
\end{mydef}
For $\letter\in\alphabet$ we introduce the natural notation $\recta\letter(\la)=\grectC\is{\sii}{\s}\js{\si}{k}(\la)$ and $\ppointa\letter(\la) = \gppointC\is{\sii}{\s}\js{\si}{k}(\la)$; similarly we let $\bbasea{\letter}(\la)=\bbase{\recta{\letter}(\la)}$ and $\fbasea{\letter}(\la)=\fbase{\recta{\letter}(\la)}$. For $\la\in\reals^+$ we say that a symbol $\letter$ is $\la$-admissible if $\recta{\letter}(\la)$ is admissible and define the \emph{$\la$-admissible alphabet} $\alphabet(\la)$ as follows:
\[
\alphabet(\la)=\{\letter\in\alphabet\st \letter\textrm{ is $\la$-admissible}\}
\]
and similarly for the restricted alphabets. Likewise we say that a symbol $\letter$ is $\la$-proper if $\recta{\letter}(\la)$ is proper.

We now introduce a few additional definitions: we say that a rectangle $\recta\letter$ is a \emph{$\la$-bulk rectangle} (and that the symbol $\letter$ is a \emph{$\la$-bulk symbol}) if the following holds:
\[
\fa p\in\recta\letter\ \la|\slope{1}(p)|>\frac{1}{10}\la;
\]
notice that this is a non-trivial definition since we assume \eqref{e_icritUniformEstimate} to hold. Restricting ourselves to bulk rectangles in the constructive part of our proof will be extremely useful since bulk rectangles satisfy nearly optimal expansion and distortion conditions, which will ease the process of establishing good lower bounds. On the other hand, in order to obtain good upper estimates, it is convenient to classify rectangles according to the following definition: for $\alpha\in(1/2,1)$ we say that a rectangle is $\alpha$-critical if it is not a bulk rectangle and:
\[
\fa p\in\rect^*\ \la|\slope{1}(p)|<4\la^\alpha.
\]
The following proposition shows that, for large enough $\la$, most rectangles are bulk rectangles and those which we can only obtain poor expansion and distortion estimates are not too many.
\begin{prp}\label{p_rectCardinalityBound}
  For large enough $\la$ any $\rect$ allows at most $4\la$ admissible components. Moreover the number of $\alpha$-critical rectangles is bounded by $\const\,\la^{2\alpha-1}$.
\end{prp}
\begin{proof}
  We noticed already that each admissible component $\recta\letter$ contains a point $\ppointa\letter$; we can then bound the total number of admissible components with the number of points in each point set $\gpoint\is\sii\s\js\si(\la)$, which can in turn be bounded, by definition, by the number of times that the image of a piece of integral curve of $\slope{1}$ contained in $\gstr{\ds}{\s}$ wraps around the cylinder along the first coordinate. But this is simple to obtain, since the maximum horizontal expansion rate on an integral curve of $\slope{1}$ is by definition bounded by $\sup_{p\in\torus^2}\la|\slope{1}(p)|=2\pi\la+2<7\la$ for $\la$ large enough. Since $\diam\gpstr{\ds}{\s}<1/2$ we can then conclude.

  The very same argument gives a bound on the number of $\alpha$-critical rectangles: in fact let
  \[
  \picrit(\la;\alpha)=\{\xi\in\torus\st|\ddot\phi(\xi)+2/\la|<\min(4\la^{\alpha-1},1/10)\}.\]
  Then clearly we have $|\picrit(\la;\alpha)|<\const\,\la^{\alpha-1}$; moreover by definition each $\alpha$-critical component $\rect^*$ is such that $\pix\rect^*\subset\picrit(\la;\alpha)$; the maximum expansion rate on such rectangles along integral curves of $\slope{1}$ is by definition $4\la^\alpha$, from which we can conclude since the number of $\alpha$-critical rectangles is then bounded by $4|\picrit(\la;\alpha)|\la^\alpha=\const\,\la^{2\alpha-1}$.
\end{proof}
We conclude this subsection by collecting a few simple geometrical estimates about admissible components; these estimates are the quantitative counterpart of figure~\ref{f_rectangles}.

\begin{prp}[Geometry of components]\label{p_componentGeometry}
  Let $\rect^*$ be a connected component of some $\rect$; then
  \begin{equation}\label{e_weakDiameterRect}
    \diam\pix\rect^*<\frac{1}{4}\la^{-1/2};
  \end{equation}
  moreover, if $\rect^*$ is a rectangle, denoted by $\recta{\letter}$, let $\xi_\letter=\pix\ppointa{\letter}$ and $\deltaXi_\letter=3/4(\la\slope{1}(\ppointa{\letter}))^{-1}$; then:
  \begin{equation}\label{e_strongDiameterRect}
    \pix \recta{\letter} \subset [\xi_\letter-\deltaXi_{\letter},\xi_{\letter}+\deltaXi_{\letter}]
  \end{equation}
\end{prp}
\begin{proof}
  Consider the differential $\deh[\pix\homeo^{-1}](\homeo(p))=(\la\slope{1}(p,\la))^{-1}(1\ 1)$; by definition of inner critical set we know that $|\la\slope{1}(p,\la)|>{4}\la^{1/2}$ for $p\in\rect^*$, so that:
  \[
  \diam(\pix\rect^*)<(\diam{\fbase{\rect^*}}+\diam{\bbase{\rect^*}})\frac{1}{4}\la^{-1/2}<\frac{1}{4}\la^{-1/2};
  \]
  from which we obtain \eqref{e_weakDiameterRect}. Likewise, we can obtain \eqref{e_strongDiameterRect} from analogous arguments simply using the definition of $\ppointa{\letter}$, estimates \eqref{e_oneDefinitions} and distortion estimate \eqref{e_smallDistortion}.
\end{proof}
Proposition~\ref{p_componentGeometry} immediately implies that $\fa p\in\recta{\letter}$ we have
\[\slope{1}(p)=\slope{1}(\ppointa{\letter})\onepbigo{\la^{-1}|\slope{1}(\ppointa\letter)|^{-2}};\]
in particular we obtain for a bulk rectangle $\recta{\letter}$ the following enhanced estimate:
\begin{equation}\label{e_bulkSlopeControl}
  \fa p\in\recta{\letter}\ \slope{1}(p)=\slope{1}(\ppointa{\letter})\onepbigo{\la^{-1}};
\end{equation}
A straightforward check, given our definition of the indexing procedure, proves the following relations which hold whenever the corresponding rectangles are admissible
\begin{subequations}\label{e_orderRelationsXiC}
  \begin{align}
    \xiC\is-\s\js\si j &<  \xiC\is+\s\js\si j &\xiC\is{+}{\s}\js{\si}{j} &<  \xiC\is{-}{\s}\js{\si}{j+1}\\
    \xiC\is + \s\js + j &= \xiC\is- \s \js - j &  \xiC\is+ \s \js - j &=  \xiC\is-\s\js+{j-1},
  \end{align}
\end{subequations}
moreover if the corresponding symbols are bulk rectangles we obtain the following estimates:
\begin{subequations}\label{e_deltaRelationsXiC}
  \begin{align}
    \xiC\is\sii\s\js\si {j+1}-\xiC\is\sii\s\js\si j &=|\la\slope{1}(\gppointC\is\sii\s\js\si j)|^{-1}\onepbigo{\la^{-1}}\\
    \xiC\is{+}{\s}\js{\si}{j}-\xiC\is{-}{\s}\js{\si}{j}&=\frac{1}{2}|\la\slope{1}(\gppointC\is{+}{\s}\js{\si}{j})|^{-1}\onepbigo{\la^{-1}}
  \end{align}
\end{subequations}
furthermore, we have that:
\begin{subequations}\label{e_absoluteRelationsXiC}
  \begin{align}
    \xiC\is{\si}{\s}\js{\sii}{0} &= \xi_\s + \bigo{\la^{-1}}\\
    \xiC\is{\si}{+}\js{\sii}{j} -\xi_+ &= \xiC\is{\si}{-}\js{\sii}{j} -\xi_- + \bigo{\la^{-1}}
  \end{align}
\end{subequations}
which immediately follow from the definition of $\phi$ and from our choice of indexing procedure. In particular:
\begin{equation}\label{e_zeroXiC}
  \xiC\ds{\si}{\s}\ds{\si}{0} = \xi_\s.
\end{equation}
The following proposition is simply a collection of useful estimates whose proof is trivial given relations \eqref{e_orderRelationsXiC}, \eqref{e_deltaRelationsXiC} and \eqref{e_absoluteRelationsXiC}.
\begin{prp}\label{p_symmetricLetters}
  Fix $\la$ and $\letter\in\alphabet(\la)$ a bulk symbol; let $\letter'\in\alphabet(\la)$ a bulk symbol obtained from $\letter$ by means of any of the following operations: $\sii\mapsto -\sii$, $\si\mapsto -\si$, $k\mapsto -k$ and $k\mapsto k+1$; then:
  \[
  \la\slope{1}(\ppointa{\letter'})=\la\slope{1}(\ppointa{\letter})\onepbigo{\la^{-1}};
  \]
  on the other hand, if $\letter'$ has instead been obtained from $\letter$ via $\s\mapsto -\s$:
  \[
  \la\slope{1}(\ppointa{\letter'})=-\la\slope{1}(\ppointa{\letter})\onepbigo{\la^{-1}}.
  \]
\end{prp}
\subsection{Refinement of the structure}
As the reader might expect, we will use symbols in the alphabet set to encode chunks of orbits of the map $\sm$; if a given orbit never visits the critical set $\crit$, it is possible to encode it with an infinite sequence of symbols belonging to $\alphabetReg$. Conversely, we only require to encode orbits that eventually visit $\crit$ as long as they stay outside of it; this implies that we will need to deal with a number of different types of encodings, which will be described in detail below.
We first now define rules according to which the dynamics composes words out of our symbols.
\begin{prp}\label{p_compatibilityIsGood}
  Let $\letter\in\alphabetHead(\la)$ or $\alphabetReg(\la)$; let $\altLetter\in\alphabetTail(\la)$ or $\alphabetReg(\la)$. Then we have:
  \[
  \sm\recta{\letter}\cap\recta{\altLetter}\not = \emptyset
  \]
  if and only if the following compatibility condition holds:
  \begin{align}\label{e_compatibilityCondition}
    \s_\letter &=\sii_{\altLetter}& \si_\letter &=\s_{\altLetter}
  \end{align}
\end{prp}
\begin{proof}
  Consider any curve $\Gamma_\letter\subset\recta{\letter}$ joining the two opposite sides of $\partial_\stable\recta{\letter}$; take any other curve $\Gamma_\altLetter\subset\recta{\altLetter}$ joining the two opposite sides of $\partial_\unstable\recta{\altLetter}$. Then, by definition, $\sm^{-1}\Gamma_\altLetter\subset\gsqu\ds\ds{\sii_\altLetter}{\s_\altLetter}$ and it joins the two opposite vertical sides of $\gsqu\ds\ds{\sii_\altLetter}{\s_\altLetter}$; on the other hand $\Gamma_\letter\subset\gsqu\ds\ds{\s_\letter}{\si_\letter}$ and joins the two opposite horizontal sides of $\gsqu\ds\ds{\s_\letter}{\si_\letter}$. Consequently we can conclude that $\Gamma_\letter\cap\sm^{-1}\Gamma_\altLetter\not = \emptyset$ if and only if $\s_\letter = \sii_\altLetter$ and $\si_\letter = \s_\altLetter$, that is \eqref{e_compatibilityCondition}.
\end{proof}
\begin{prp}[Markov property]\label{l_markovProperty}\label{l_existenceInvariantManifold}
  Let $\letter\in\alphabetReg(\la)$ or $\alphabetHead(\la)$ and $\altLetter\in\alphabetReg(\la)$ or $\alphabetTail(\la)$ satisfying condition \eqref{e_compatibilityCondition}; let $\Gamma_\unstable^\letter$ be an unstable curve joining the two opposite sides of $\partial_\stable\recta{\letter}$. Then we have that $\Gamma_\unstable^\altLetter=\sm\Gamma_\unstable^\letter\cap\recta{\altLetter}$ is a connected unstable curve that joins the two opposite sides of $\partial_\stable\recta{\altLetter}$; moreover we have $\intr\Gamma_\unstable^\altLetter\cap\partial_\unstable\recta{\altLetter}=\emptyset$.
  Likewise let $\Gamma_\stable^\altLetter$ be a stable curve joining the two opposite sides of $\partial_\unstable\recta{\altLetter}$. Then we have that $\Gamma_\stable^{\letter}=\sm^{-1}\Gamma_\stable^\altLetter\cap\recta{\letter}$ is a connected stable curve that joins the two opposite sides of $\partial_\unstable\recta{\letter}$; moreover we have $\intr\Gamma_\stable^{\letter}\cap\partial_\stable\recta{\letter}=\emptyset$.
\end{prp}
\begin{proof}
  Consider any stable curve in $\recta{\altLetter}$ joining the two opposite sides of $\partial_\unstable\recta{\altLetter}$; then its preimage will intersect $\Gamma_\unstable^\letter$ in an unique point, since outside $\icrit$ the two cones $\cone\stable$ and $\cone\unstable$ are invariant and disjoint. Now foliate $\recta{\altLetter}$ with horizontal curves; then $\sm\Gamma_\unstable^\letter$ will intersect each leaf only once; since $\cone\unstable$ is invariant we have that $\Gamma_\unstable^\altLetter$ is itself an unstable curve. Finally assume that there exists $p\in\Gamma_\unstable^\altLetter\cap\partial_\unstable\recta{\altLetter}$; then $\sm^{-1}p\in\partial\icrit$ by definition, but since $\letter$ is admissible, this implies that $\sm^{-1}p\in\sm^{-1}\partial\icrit$; hence, since $\altLetter$ is admissible we have $p\in\partial_\stable\recta{\altLetter}$. Since an unstable curve cannot have a horizontal tangent, it must be $p\in\partial\Gamma_\unstable^\altLetter$.

  The argument for the stable curve is completely symmetrical and it is omitted.
\end{proof}
We now give the definition of \emph{words}, i.e.\ ordered strings of symbols which we will
use to identify specific portions of the phase space. We need to distinguish between
closed, left and right closed and open words. Closed words will be used to track finite
chunks of orbits whose first and last point lies inside $\crit$ with all other points
lying outside; left-closed words will instead track chunks of orbits which start in
$\crit$ such that all other points are outside, whereas right-closed words will track
those which finish inside $\crit$ with all other points outside; finally, open words will
track chunks of orbits which do not intersect $\crit$ at all.
\begin{mydef}[Words]
  A \emph{closed word} is the datum of two signs $\si,\sii$ and a possibly empty finite ordered string of symbols belonging to $\alphabet$ constructed according to the following rules:
  \begin{itemize}
  \item if the closed word is composed of just one symbol, than this symbol belongs to $\alphabetDeg$;
  \item if the closed word is composed of more than one symbol, then the following conditions hold:
    \begin{itemize}
    \item the first symbol belongs to $\alphabetHead$;
    \item the last symbol belongs to $\alphabetTail$;
    \item all other symbols (if any) belong to $\alphabetReg$;
    \end{itemize}
  \end{itemize}
  Closed words will be denoted by enclosing the finite string between the two signs $\si$ and $\sii$ respectively before the first symbol and after the last symbol, e.g. $\si\letter_1\cdots\letter_{l}\sii$; we moreover require that the compatibility condition holds for all symbols, i.e.\ the symbol $\altLetter$ can follow the symbol $\letter$ only provided \eqref{e_compatibilityCondition} holds; this condition is naturally required to hold for the boundary signs as well.

  A \emph{left-closed word} is the datum of a sign $\si$ and a possibly empty, finite or forward-infinite ordered string of symbols belonging to $\alphabet$ constructed according to the following rules:
  \begin{itemize}
  \item the first symbol (if any) belongs to $\alphabetHead$;
  \item all other symbols (if any) belong to $\alphabetReg$;
  \end{itemize}
  Left-closed words will be denoted by writing the sign $\si$ on the left of the string e.g. $\si\letter_1\cdots\letter_{l}$; we again require the compatibility condition to hold for all symbols, including the boundary sign.

  Likewise a \emph{right-closed word} is the datum of a sign $\sii$ and a possibly empty, finite or backward-infinite ordered string of symbols belonging to $\alphabet$ constructed according to the following rules:
  \begin{itemize}
  \item the last symbol (if any) belongs to $\alphabetTail$;
  \item all other symbols (if any) belong to $\alphabetReg$;
  \end{itemize}
  Right-closed words will be denoted by writing the sign $\sii$ on the right of the string e.g. $\letter_1\cdots\letter_{l}\sii$; we again require the compatibility condition to hold for all symbols, including the boundary sign.

  Finally an \emph{open word} is a finite, one-sided-infinite or two-sided-infinite ordered string of symbols belonging exclusively to $\alphabetReg$ and subject to the compatibility condition.
\end{mydef}
A word is said to be \emph{$\la$-admissible} if every symbol is $\la$-admissible; likewise a word is said to be \emph{$\la$-proper} if every symbol is $\la$-proper. The set of all closed words and closed $\la$-admissible words of length $l$ is denoted respectively by $\cwords{l}$ and $\cwords{l}(\la)$.

It is natural to establish a correspondence between right and left closed words and cylinder sets of closed words; let $r>0$:
\begin{itemize}
\item given a left-closed word $\si\letter_1\cdots\letter_r$ of length $r$, we define the associated \emph{forward cylinder set of rank $r$} as the set of all closed words with prescribed first $r$ symbols:
  \begin{align*}
    \fcyl=\smash{\bigcup_{l>r}}\fcyl_l&=\{\itin\in\cwords{l}\st\itin=\si\letter_1\cdots\letter_r\altLetter_{r+1}\cdots\altLetter_{l}\sii\\&\textrm{ with } \altLetter_j\in\alphabetReg \textrm{ for } r<j< l\textrm{ and } \altLetter_l\in\alphabetTail\}
  \end{align*}
\item given a right-closed word $\letter_{l-r+1}\cdots\letter_l\sii$ of length $r$, we define the associated \emph{backward cylinder set of rank $r$} as the set of all closed words with prescribed last $r$ symbols:
  \begin{align*}
    \bcyl=\smash{\bigcup_{l>r}}\bcyl_l&=\{\itin\in\cwords{l}\st\itin=\si\altLetter_1\cdots\altLetter_{l-r}\letter_{l-r+1}\cdots\letter_{l}\sii\\&\textrm{ with } \altLetter_1\in\alphabetHead\textrm{ and } \altLetter_j\in\alphabetReg\textrm{ for } 1 < j < l-r+1\}
  \end{align*}
\item we define a \emph{bicylinder set of rank $r$} as the intersection of a forward cylinder set with  a backward cylinder set of same rank $r$. Bicylinders will be denoted simply by $\Bcyl$.
\end{itemize}
For fixed $\la$ we will say that a forward cylinder (resp.\ backward cylinder, bicylinder) is a \emph{$\la$-bulk} forward cylinder (resp. \emph{$\la$-bulk} backward cylinder, \emph{$\la$-bulk} bicylinder) if any symbol $\letter$ appearing in their corresponding words is a $\la$-bulk rectangle.

We now proceed to define four sequences of subsets of the phase space; if a point $p$ belongs to a set in one of these sequences, then we can assign to $p$ a word of the appropriate class. We begin by defining the sets: $\frectUnion_0(\la)\defeq\brectUnion_0(\la) \defeq \bigcup_{\s,\si\in\{+,-\}}\gsqu\ds\ds{\s}{\si}(\la)$ and for $N>1$ let:
\begin{align*}
  \frectUnion_N(\la) &\defeq \bigcap_{0\leq l < N}\sm^{-l}\recta{\alphabetReg(\la)}(\la) &
  \brectUnion_N(\la) &\defeq \bigcap_{-N\leq l < 0}\sm^{-l}\recta{\alphabetReg(\la)}(\la)
\end{align*}
Then define the symmetrical set:
\[
\rectUnion_N(\la) \defeq \frectUnion_{N}(\la) \cap \brectUnion_{N}(\la).
\]
Furthermore, let
\begin{align*}
  \fcritiUnion_0(\la)&\defeq\bigcup_{\s,\si\in\{+,-\}}\squH{\s}{\si}(\la)&
  \bcritiUnion_0(\la)&\defeq\sm\bigcup_{\s,\si\in\{+,-\}}\squT{\s}{\si}(\la);\\
  \fcritiUnion_1(\la)&\defeq\sm^{-1}\recta{\alphabetHead(\la)}(\la)&
  \bcritiUnion_1(\la)&\defeq\sm\recta{\alphabetTail(\la)}(\la);
\end{align*}
then for $N>1$ let:
\begin{align*}
  \fcritiUnion_N(\la)&\defeq\sm^{-1}\recta{\alphabetHead(\la)}(\la)\cap\bigcap_{1< l \leq N}\sm^{-l}\recta{\alphabetReg(\la)}(\la)\\
  \bcritiUnion_N(\la)&\defeq\sm\recta{\alphabetTail(\la)}(\la)\cap \bigcap_{1 < l \leq N}\sm^{l}\recta{\alphabetReg(\la)}(\la)
\end{align*}
Finally:
\begin{align*}
  \wall_0(\la)&\defeq\bigcup_{\s,\si\in\{+,-\}}\squD{\s}{\si}(\la)\\
  \wall_1(\la)&\defeq\sm^{-1}\recta{\alphabetDeg(\la)}(\la)\\
  \wall_2(\la)&\defeq\sm^{-1}\recta{\alphabetHead(\la)}(\la)\cap\sm^{-2}\recta{\alphabetTail(\la)}(\la)
  \intertext{and for $N>2$:}
  \wall_N(\la)&\defeq\sm^{-1}\recta{\alphabetHead(\la)}(\la)\cap\bigcap_{2\leq l<N}\sm^{-l}\recta{\alphabetReg(\la)}(\la) \cap\sm^{-N} \recta{\alphabetTail(\la)}(\la)
\end{align*}
By definition $\frectUnion_N(\la)\cap\icrit(\la)=\emptyset$ and similarly $\brectUnion_N(\la)\cap\icrit(\la)=\emptyset$, whereas $\fcritiUnion_N(\la)\subset\crit(\la)$, $\bcritiUnion_N(\la)\subset\crit(\la)$ and $\wall_N(\la)\subset\crit(\la)$. Notice moreover that $\frectUnion_N(\la)\supset\frectUnion_{N+1}(\la)$, $\brectUnion_N(\la)\supset\brectUnion_{N+1}(\la)$ and likewise $\fcritiUnion_N(\la)\supset\fcritiUnion_{N+1}(\la)$, $\bcritiUnion_N(\la)\supset\bcritiUnion_{N+1}(\la)$.
\begin{figure}[!ht]
  \begin{center}
    \def\svgwidth{5.0cm}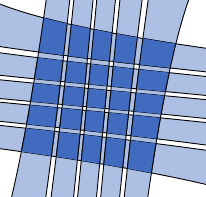\\[0.35cm]
  \end{center}
  \caption{Sketch of the portion of the set $\rectUnion_1$ contained in $\gsqu\ds\ds++$; the curved vertical strips form $\frectUnion_1$, whereas the curved horizontal strips form $\brectUnion_1$. The darker region describes the set $\rectUnion_1$.
  }
\end{figure}
\begin{figure}[!ht]
  \footnotesize
  \begin{center}
    \begin{minipage}{6.0cm}
      \def\svgwidth{5.0cm}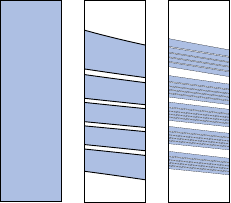\\[0.35cm]
    \end{minipage}
    \begin{minipage}{6.0cm}
      \def\svgwidth{5.0cm}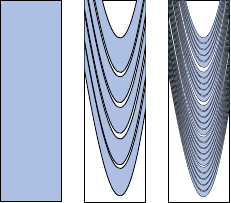\\[0.35cm]
    \end{minipage}
  \end{center}
  \caption{Sketch of a portion of the sets $\fcritiUnion$ and $\bcritiUnion$, for simplicity restricted to an appropriate choice of signs. The picture on the right is further simplified: in fact the parabolic strips have in reality very high curvature and a single strip should wrap around the torus several times.}
  \label{f_CylinderDomains}
\end{figure}
\begin{mydef}[Coding]
  Let $r$ be a non-negative integer; then:
  \begin{itemize}
  \item given $p\in\wall_r(\la)$ we define $\itin_r(p)$ the \emph{closed itinerary of rank $r$ of $p$} as the closed word $\si\letter_1\cdots\letter_{r}\sii$ such that: $p\in\gstr\cs\si$, $\sm^{r+1}p\in\gstr\cs\sii$ and for $1\leq l \leq r$ we have $\sm^lp\in\recta{\letter_l}$;
  \item given $p\in\fcritiUnion_r(\la)$ we define $\fcyl_r(p)$ the \emph{forward cylinder of rank $r$ associated to $p$} as the forward cylinder associated to $\si\letter_1\cdots\letter_{r}$ such that $p\in\gstr\cs\si$ and for $1 \leq l \leq r$ we have $\sm^lp\in\recta{\letter_l}(\la)$;
  \item given $q\in\bcritiUnion_r(\la)$ we define $\bcyl_r(q)$ the \emph{backward cylinder of rank $r$ associated to $q$} as the backward cylinder associated to $\letter_{-r}\cdots\letter_{-1}\sii$ such that $p\in\gstr\cs\sii$ and for $-r\leq l \leq -1$ we have $\sm^{l}p\in\recta{\letter_{l}}$; 
  \item given $p\in\rectUnion_r(\la)$, we define $\itin_r(p)$ the \emph{open itinerary of rank $r$ of $p$} as the open word $\letter_{-r}\cdots\letter_{r-1}$ such that  for $-r \leq l < r$ we have $\sm^lp\in\recta{\letter_l}(\la)$;
  \item given $p\in\brectUnion_r(\la)$, define $\bitin_r(p)$ the \emph{open backward itinerary of rank $r$ of $p$} as the open word $\letter_{-r}\cdots\letter_{-1}$ such that for $-r\leq l <0$ we have $\sm^{l}p\in\recta{\letter_{l}}(\la)$;
  \item given $p\in\frectUnion_r(\la)$, define $\fitin_r(p)$ the \emph{open forward itinerary of rank $r$ of $p$} as the open word $\letter_0\cdots\letter_{r-1}$ such that for $0\leq l < r$ we have $\sm^lp\in\recta{\letter_l}(\la)$;

  \end{itemize}
\end{mydef}
\subsection{Hyperbolic set and critical domains}
In this subsection we construct a locally maximal hyperbolic set and we introduce the notion of critical domain, which is of fundamental importance in the proof of the Main Theorem. Until next subsection we assume $\la$ to be fixed and drop it from the notation where it will not be source of confusion.
\begin{mydef}
  Let $p\in\rectUnion_N$ and $\itin_N(p)=\letter_{-N}\cdots\letter_{N-1}$; for $-N \leq l < N$ define $\Gamma_u^l(p)$ as the leaf of the $\slope{N+l+1}$-foliation of $\recta{\letter_l}$ containing $\sm^lp$. Similarly, define $\Gamma_s^l(p)$ as the leaf of the $\slope{N-l}$-foliation of $\recta{\letter_l}$ containing $\sm^lp$.
\end{mydef}
Then, proposition~\ref{l_markovProperty} immediately implies the following
\begin{cor}\label{c_regularityOfGammaus}
  Let $p\in\rectUnion_N$, then for $-N \leq l < N$ we have $\Gamma_\unstable^l(p)\subset\cl\bnicePoints_{N+l+1}$ and $\Gamma_\stable^l(p)\subset\cl\fnicePoints_{N-l}$. In particular $\Gamma_\unstable^l(p)$ is an unstable curve and $\Gamma_\stable^l(p)$ is a stable curve.
\end{cor}
\begin{prp}\label{l_exponentialCantorEstimate}
  Let $p,q\in\rectUnion_N$ such that $\itin_N(p)=\itin_N(q)$. Then the following estimate holds:
  \[
  \dist(p,q)<\const\,\la^{-N/2}
  \]
\end{prp}
\begin{proof}
  Corollary~\ref{c_regularityOfGammaus} implies that $\Gamma_\unstable^0(p)\cap\Gamma^0_\stable(q)$ is given by a unique point $r\in\recta{\letter_0}$. We claim that $r\in\rectUnion_N$ and that it has the same itinerary as both $p$ and $q$. In fact proposition~\ref{l_existenceInvariantManifold} gives
  \begin{align*}
    \sm^jr\in\sm^j\Gamma^0_\stable(q)\subset&\Gamma^j_\stable(q)\subset\recta{\letter_j},\\
    \sm^{-j}r\in\sm^{-j}\Gamma^0_\unstable(p)\subset&\Gamma^{-j}_\stable(p)\subset\recta{\letter_{-j}};
  \end{align*}
  the claim is thus proved. We can consequently bound the distance by the triangle inequality $\dist(p,q)<\dist(p,r) + \dist(r,q)$; using expansion estimates \eqref{eq_expansionRate} and bounds given by proposition~\ref{l_estimateReferenceSequence} we can then conclude.
\end{proof}
\begin{mydef}\label{c_definitionWus}
  Let $p\in\rectUnion_\infty$, then define the sets
  \begin{align*}
    \Wu(p)&=\{q\in\brectUnion_\infty\st\fa n\in\naturals\ \bitin_n(q) = \bitin_n(p)\}\\
    \Ws(p)&=\{q\in\frectUnion_\infty\st\fa n\in\naturals\ \fitin_n(q) = \fitin_n(p)\}.
  \end{align*}
  by proposition~\ref{l_exponentialCantorEstimate} we observe that $\Wu(p)$ and $\Ws(p)$ are respectively a local unstable manifold and a local stable manifold of $p$.
\end{mydef}
\begin{prp}\label{l_locallyMaximal}
  The invariant set $\rectUnion_\infty$ is a locally maximal hyperbolic set
\end{prp}
\begin{proof}
  By standard arguments (see e.g. \cite{Katok}) it is sufficient to prove that $\rectUnion_\infty$ has local product structure; let $p_0,q_0\in\rectUnion_\infty$ be close enough so that $p_0,q_0\in\recta{\letter_0}$ for some $a_0$; by proposition~\ref{p_compatibilityIsGood}, we can find a point $r_0\in\rectUnion_\infty$ such that the backward itinerary of $r_0$ agrees with the backward itinerary of $p_0$ and its forward itinerary agrees with the forward itinerary of $q_0$. By corollary~\ref{c_definitionWus} we have that $r_0\in\Wu(p_0)\cap\Ws(q_0)$, hence $\rectUnion_\infty$ has local product structure.
\end{proof}
Define $\phypset(\la)\subset\hypset_\infty(\la)$ as follows:
\[
\phypset(\la)=\{p\in\hypset_\infty(\la)\st\fa n\in\naturals\ \itin_n(p)\textrm{ is $\la$-proper}\}
\]
By construction $\phypset(\la)$ is itself a locally maximal hyperbolic set for $\sm$; we now prove that it satisfies property (b) in proposition~\ref{p_hyperbolicSet}; the proof of property (a) will be given later in subsection~\ref{ss_moving}.
\begin{proof}
  Proposition~\ref{l_exponentialCantorEstimate} implies that a $\const\,\la^{-1/2}$-neighborhood of $\phypset(\la)$ covers the set $\phypset_1(\la)$ where:
  \[
  \phypset_1(\la)=\{p\in\hypset_1(\la)\st \itin_1(p)\textrm{ is $\la$-proper}\}.
  \]
  It is therefore sufficient to show a $\const\,\la^{-1/2}$-neighborhood of $\phypset_1(\la)$ covers $\torus^2$. For arbitrary fixed $\eps$, let $\tcrit(\la)$ be a $(\frac{1}{4}+\eps)\la^{-1/2}$-neighborhood of $\crit(\la)$; then proposition~\ref{p_componentGeometry} gives that if $p\not\in\tcrit(\la)$ belongs to some rectangle $\recta\letter$ then $\letter$ is proper, since $\recta\letter$ cannot intersect $\crit(\la)$. Hence, by definition of $\hypset_1(\la)$ we obtain:
  \[ \phypset_1(\la)\supset\torus^2\setminus(\sm^{-2}\tcrit(\la)\cup\sm^{-1}\tcrit(\la)\cup\tcrit(\la)\cup\sm^{1}\tcrit(\la))=\hat\hypset_1(\la)
  \]
  We claim that a $\const\,\la^{-1/2}$-neighborhood of $\hat\hypset_1(\la)$ covers $\torus^2$; in fact, by inspection, the intersection of $\tcrit(\la)$ and $\sm^{-2}\tcrit(\la)$ with any horizontal line is given by two intervals of diameter $\bigo{\la^{-1/2}}$ whose centers are approximately $1/2$ apart. The same is true for the intersection of $\sm\tcrit(\la)$ and $\smi\tcrit(\la)$ with any vertical line. Hence if $q\not\in\hat\hypset_1(\la)$, we can find a point $p\in\hat\hypset_1(\la)$ by moving with unit speed along horizontal and vertical lines for a time bounded by $\const\,\la^{-1/2}$, which concludes the proof.
\end{proof}

We now begin the description of so-called \emph{critical domains}; these sets will be used
to classify elliptic orbits of the Standard Map. First, for $k\in\integers$, introduce the
dynamical projections:
\begin{align*}
  \dpix{k}(p;\la)&\defeq\pix\sm^{k}(p)&\dpiy{k}(p;\la)&\defeq\piy\sm^{k}(p).
\end{align*}
\begin{mydef}
  Let $\Bcyl=\fcyl\cap\bcyl$ be a $\la$-admissible bicylinder of rank $r$; define the sets
  \begin{align*}
    \cdomain_{\fcyl}(\la)&=\{p\in\fcritiUnion_r(\la)\st\fcyl_r(p)=\fcyl\};\\
    \cdomain_{\bcyl}(\la)&=\{q\in\bcritiUnion_r(\la)\st\bcyl_r(q)=\bcyl\};
  \end{align*}
  Similarly let $\itin\in\cwords{r}(\la)$; the set
  \[
  \domain_\itin(\la)=\{p\in\wall_{r}(\la)\st \itin_{r}(p)=\itin\}
  \]
  is said to be a \emph{critical domain of rank $r$}.
  Notice that the Markov property implies that $\domain_\itin(\la)$ is non-empty provided that $\itin$ is $\la$-admissible, and similarly for $\cdomain_{\fcyl}$ and $\cdomain_{\bcyl}$.
\end{mydef}
\begin{figure}[!ht]
  \begin{center}
    \def\svgwidth{5cm}
    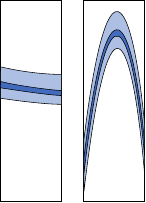
  \end{center}
  \caption{On the left we sketch $\cdomain_{\fcyl}$ (in light shade) containing a critical domain $\domain_\itin$ for $\itin\in\fcyl$ (in darker shade); on the right a sketch of $\cdomain_{\bcyl}$ containing the $(r+1)$-st image of a critical domain $\domain_{\itin'}$ of rank $r$ for $\itin'\in\bcyl$. Notice that $\cdomain_{\bcyl}$ in principle wraps around the torus several times on $\crit(\la)$, but for sake of clarity in the picture this is not shown.}
\end{figure}
Let $\fcyl$ be given by $\si\letter_1\cdots\letter_r$; then we denote by $\bbasei\fcyl=\bbasea{\letter_1}=\gpstr\cs\si$ and $\fbasei\fcyl=\fbasea{\letter_r}$; similarly if $\bcyl$ is given by $\letter_{-r}\cdots\letter_{-1}\sii$, we denote by $\bbasei\bcyl=\bbasea{\letter_{-r}}$ and $\fbasei\bcyl=\fbasea{\letter_{-1}}=\gpstr\cs\sii$. Finally let $\itin=\si\letter_1\cdots\letter_{r}\sii\in\cwords{r}(\la)$; we denote by $\bbasei\itin=\bbasea{\letter_1}=\gpstr\cs\si$ and $\fbasei\itin=\fbasea{\letter_r}=\gpstr\cs\sii$.

We introduce the maps $\fhomeo_r(\la):\torus^2\to\torus$ and $\bhomeo_r(\la):\torus^2\to\torus$ given by:
\begin{align*}
  \fhomeo_r &=(\dpix{0},\dpix{r+1})&    \bhomeo_r &=(\dpix{-r-1},\dpix{0});
\end{align*}
Notice that $\fhomeo_r$ and $\bhomeo_r$ are differentiable and
\begin{align*}
  \deh\fhomeo_r(p)&=\matrixtt{1}{0}{-\stepex{0r+1}(p)\la\slope{-r-1}(p)}{\stepex{0r+1}(p)}\\
  \deh\bhomeo_r(q)&=\matrixtt{\stepex{0r+1}(\sm^{-r-1}q)\la\slope{r+1}(q)}{-\stepex{0r+1}(\sm^{-r-1}q)}{1}{0}
\end{align*}
where $\stepex{0r+1}(p)=\la\slope{1}(\sm p)\cdots\la\slope{r}(\sm^r p)$.
Notice that by definition $\homeo(\la)=\fhomeo_2(\la)\circ\sm^{-1}=\bhomeo_2(\la)\circ\sm$.
\begin{prp}\label{p_admissibleDomain}
  Let $\itin$ be a $\la$-admissible closed word of length $r$ and $\Bcyl=\fcyl\cap\bcyl$ be a $\la$-admissible bicylinder of rank $r$; then:
  \begin{itemize}
  \item $\fhomeo_r$ is a diffeomorphism between $\domain_\itin$ and $\bbasei\itin\times\fbasei\itin$;
  \item $\fhomeo_r$ is a diffeomorphism between $\cdomain_{\fcyl}$ and $\bbasei\fcyl\times\fbasei\fcyl$;
  \item $\bhomeo_r$ is a diffeomorphism between $\cdomain_{\bcyl}$ and $\bbasei\bcyl\times\fbasei\bcyl$.
  \end{itemize}
\end{prp}
\begin{proof}
  The proof is a generalization of the proof of the analogous property for proposition~\ref{p_alternativeAdmissible}. Moreover the proof is basically the same for all three items, with trivial adaptations. We explicitly give the proof of the first item only.

  For simplicity of notation let $\domain=\domain_\itin$, $\bbasei{}=\bbasei{\itin}$ and $\fbasei{}=\fbasei{\itin}$; assume for definiteness that $\itin=\si\letter_1\cdots\letter_r\sii$. By simple inspection $\fhomeo_r$ is a local diffeomorphism in a neighborhood of $\domain$; it suffices then to prove that $\fhomeo_r\domain = \bbasei{}\times\fbasei{}$.  Let $\xi\in\bbasei{}$ and $\xi'\in\fbasei{}$ and fix a point $q\in\domain$; for $0<j\leq r$ let $\Gamma_j$ denote the leaf of the $\slope{j}$-foliation of $\recta{\letter_j}$ passing through $\sm^j q$. By the Markov property we have that $\Gamma_{j+1}\subset\Gamma_j$ and that $\Gamma_r$ is an unstable curve joining the two opposite sides of $\partial_\stable\recta{\letter_r}$. Hence $\sm\Gamma_r$ is a graph of a function over $\fbasei{}$ and, as such, it will intersect any given vertical line $\pix^{-1}(\xi')$ in a unique point that we denote by $q'$; let now $\Gamma'_j$ be the leaf of the $\slope{j-r-1}$-foliation of $\recta{\letter_j}$ containing $\sm^{j-r-1}q'$. Then by the Markov property again we have that $\Gamma'_j\subset\sm^{-1}\Gamma'_{j+1}$ and that $\Gamma'_1$ is a stable curve joining the two opposite sides of $\partial_\unstable\recta{\letter_1}$. Therefore $\sm^{-1}\Gamma'_1$ is the graph of a function over $\bbasei{}$ and it will intersect any vertical line $\pix^{-1}(\xi)$ in a unique point denoted $p$ such that $\fhomeo_r(p)=(\xi,\xi')$. Since $\xi$ and $\xi'$ were arbitrary, $\fhomeo_r$ is a diffeomorphism.
\end{proof}
We introduce the useful notations:
\begin{align*}
  \pd{\itin}(\xi,\xi';\la)&=\left[\fhomeo_r|_{\domain_\itin(\la)}\right]^{-1}(\xi,\xi')&
  \qd{\itin}(\xi,\xi';\la)&=\sm^{r+1}\pd{\itin}(\xi,\xi';\la)\\
  \pd{\fcyl}(\xi,\xi';\la)&=\left[\fhomeo_r|_{\cdomain_{\fcyl}(\la)}\right]^{-1}(\xi,\xi')&
  \qd{\fcyl}(\xi,\xi';\la)&=\sm^{r+1}\pd{\fcyl}(\xi,\xi';\la)\\
  \pd{\bcyl}(\xi,\xi';\la)&=\sm^{-r-1}\qd{\bcyl}(\xi,\xi';\la)&
  \qd{\bcyl}(\xi,\xi';\la)&=\left[\bhomeo_r|_{\cdomain_{\bcyl}(\la)}\right]^{-1}(\xi,\xi')
\end{align*}
This construction allows us to assign a closed word to each periodic orbit of cyclicity one; this correspondence is not $1-1$: in fact several orbits may correspond to a given closed word; in the next section we will indeed prove that there is a $1-1$ correspondence between closed word and cyclicity one elliptic periodic orbits.
\begin{prp}
  Fix $\la$ and let $p\in\crit(\la)\cap\sm^{-N}\crit(\la)$ such that
  \[
  p\not\in\bigcup_{0<j<N}\sm^{-j}\crit(\la),\]
  then there exists $\itin\in\cwords{N-1}(\la)$ such that $p\in\domain_{\itin}(\la)$.
\end{prp}
\begin{proof}
  Let $p_j=\sm^j p$; by definition we have that, for $0<j<N$ the point $p_j$ belongs to
  some set $\rect_j$; let us denote by $\rect^*_j$ the connected component containing
  $p_j$ we need to prove that $\rect^*_j$ is admissible.  By proposition~\ref{p_componentGeometry} we have that $\pix\rect^*_j$ is contained in a $\frac{1}{4}\la^{-1/2}$-neighborhood of $\pix{}p_l$. But by \eqref{e_minimumDistance} we have that $\dist (\pix p_j,\cl\picrit(\la))>\la^{-1/2}$; which implies that $\rect^*_j\cap\cl\icrit(\la)=\emptyset$. Hence $\rect^*_j$ is admissible for all $j$ and $p\in\wall_{N-1}(\la)$ from which the statement follows.
\end{proof}
\begin{cor} \label{c_s1ppAdmissible}
  Let $p\in\crit$ be a cyclicity 1 periodic point of least period $N$; then $p$ belongs to some domain $\domain_\itin$ of rank $N-1$. Hence necessarily $p=\pd\itin(\xi,\xi;\la)=\qd\itin(\xi,\xi;\la)$ for $\xi=\pix p$.
\end{cor}
\subsection{Parameter dependence of the Markov structure}\label{ss_moving}
In this subsection we describe how the Markov structure changes when the parameter $\la$ varies. The evolution with the parameter $\la$ turns out to be rather simple to understand and describe; this happens to be decisive for the proof of our Main Theorem, which strongly relies on this quantitative description.
We first provide some definitions: fix $q\in\torus^2$ and a parameter value $\la$; for $k\in\integers$ let $q_k(q,\la)=(x_k(q,\la),y_k(q,\la))=\sm^kq$ and define $A_k,\nu_k,\hat A_k$ as follows:
\begin{align*}
  A_k(q,\la)&=\deh \sm(q_k), &\nu_{k}(q)&=\left(\begin{array}{c}0\\\dot\phi(x_{k+1})\end{array}\right),
\end{align*}
\begin{equation*}
  \hat A_{k}(q,\la) = \left(\begin{array}{cc}A_{k}(q,\la)&\nu_{k}(q)\\0&1\end{array}\right).
\end{equation*}
Let $A^{(0)}\equiv\Id_{2\times2}$ and $\hat A^{(0)}\equiv\Id_{3\times3}$; define for $k>0$:
\begin{align*}
  A^{(k)} &= A_{k-1} A^{(k-1)}
  &A^{-(k)} &= \left(A^{(k)}\right)^{-1}\\
  \hat A^{(k)} &= \hat A_{k-1}A^{(k-1)}
  &\hat A^{-(k)} &= \left(\hat A^{(k)}\right)^{-1}
\end{align*}
Notice that $\hat A^{(k)}$ and $\hat A^{-(k)}$ are of the form:
\begin{align*}
  \hat A^{(k)}&=\left(\begin{array}{cc}A^{(k)}&\nu^{(k)}\\0&1\end{array}\right)&  \hat A^{-(k)} &=\left(\begin{array}{cc}A^{-(k)}&\nu^{-(k)}\\0&1\end{array}\right), 
\end{align*}
where we define:
\begin{align*}
  \nu^{(k)}&=\sum_{j=1}^kA^{(k)}A^{-(j)}\nu_{j-1}, &  \nu^{-(k)} &= -\sum_{j=1}^kA^{-(j)}\nu_{j-1}.
\end{align*}
By definition the differentials $\deh x_k$ and $\deh y_k$ satisfy the following linear relation:
\begin{equation}\label{eq_iterationDifferential}
  \left(\begin{array}{c} \deh x_{k+1}\\ \deh y_{k+1}\\ \deh\la\end{array}\right) = \hat A_{k}
  \left(\begin{array}{c} \deh x_{k}  \\ \deh y_{k}  \\ \deh\la\end{array}\right).
\end{equation}
Finally we define, for $0\leq j,k\leq N$ the real functions $\stepex{jk}(q,\la)$:
\[
\stepex{jk} = \left(\begin{array}{cc}1&0\end{array}\right) A^{(k)}A^{-(j)} \left(\begin{array}{c} 0\\1\end{array}\right);
\]
notice that, since $A^{(\cdot)}$ is a $2\times 2$ area preserving matrix we have $\stepex{jk} = -\stepex{kj}$, hence $\stepex{jj}\equiv 0$; by simple inspection we obtain:
\begin{equation}\label{e_definitionAlpha}
  \stepex{jk}(q,\la) =\begin{cases}
    \la\slope{1}(q_{j+1})\cdots\la\slope{k-(j+1)}(q_{k-1}) &\textrm{if }0\leq j<k-1;\\
    1 &\textrm{if }j=k-1,
  \end{cases}
\end{equation}
which matches the notation for $\stepex{}$ we introduced in the previous subsection.
Notice that by definition of $\dpix{k}(p;\la)$ we have, for $k>0$:
\begin{align*}
  \partial_p\dpix{k}(p,\la)&=(-\stepex{1k}(p,\la)\quad \stepex{0k}(p,\la));\\
  \partial_\la\dpix{k}(p,\la)&=\sum_{i=1}^{k}\stepex{ik}(p,\la)\dot\phi(\dpix{i}(p,\la));\\
  \partial_p\dpix{-k}(p,\la)&=(\stepex{0k+1}(\sm^{-k}p,\la)\quad -\stepex{0k}(\sm^{-k}p,\la));\\
  \partial_\la\dpix{-k}(p,\la)&=\sum_{i=1}^{k}\stepex{0i}(\sm^{-k}p,\la)\dot\phi(\dpix{i-k}(p,\la));
\end{align*}
The following lemma is crucial: it gives essential estimates that control how the dynamics evolves when varying the parameter. 
\begin{lem}   \label{l_continuationXiXi'}
  Fix $k\in\naturals$, then for each $\la^*\in\reals^+$ and $\la\geq\la^*$ let $\rho=(\rho^\forw,\rho^\backw)$ be a smooth family of smooth maps $\rho(\la^*,\la):\torus\times\torus\to\torus\times\torus$ such that:
  \begin{align*}
    \rho(\la^*,\la^*)&=\id &\left\|\dpar{\rho^\forw}{\la}\right\|, \, \left\|\dpar{\rho^\backw}{\la}\right\|&<\frac{1}{4}.
  \end{align*}
  Then there exists a unique smooth family of smooth maps
  \[
  \tpflow{k}{\rho}{\la^*}{\la}{}:\fnicePoints_k(\la^*)\to\fnicePoints_k(\la)
  \]
  such that the following diagram commutes:
  \begin{equation}\label{e_commutingRelation}
    \begin{CD}
      \fnicePoints_k(\la^*) @>{\tpflow{k}{\rho}{\la^*}{\la}{}}>>\fnicePoints_k(\la)\\
      @VV\fhomeo_{k}(\la^*) V  @VV\fhomeo_{k}(\la) V\\
      \torus\times\torus @>\rho>> \torus\times\torus
    \end{CD}
  \end{equation}
  Furthermore,the following estimates hold:
  \begin{subequations}\label{eq_differentialEquations}
    \begin{align}
      \de{}{\la}\dpiy{0}\tpflow{k}{\rho}{\la^*}{\la}p&=\la\slope{-k}(p')\partial_\la\rho^\forw(\dpix{0}p')+ \notag\\&\quad+\stepex{0k}^{-1}(p')\partial_\la\rho^\backw(\dpix{k}p')+\bigo{\la^{-1/2}}\label{eq_differentialEquationForY0}\\
      \de{}{\la}\dpiy{k}\tpflow{k}{\rho}{\la^*}{\la}p&=\dot\phi(\dpix{k}p') - \stepex{0k}^{-1}(p')\partial_\la\rho^\forw(\dpix{0}p')\notag\\&\quad + \la\slope{k}(\sm^kp')\partial_\la\rho^\backw(\dpix{k}p')
      +\bigo{\la^{-1/2}}\label{eq_differentialEquationForYN}\\
      \intertext{and finally for $0<j<k$:}
      \left|\de{}{\la}\dpix{j}\tpflow{k}{\rho}{\la^*}{\la}p\right|&<|\la\slope{j}(\sm^jp')|^{-1}\onepbigo{\la^{-1/2}}\label{e_estimateXk}
    \end{align}
    where $p'=\tpflow{k}{\rho}{\la^*}{\la}p$.
  \end{subequations}
\end{lem}

\begin{proof}
  Fix $\la^*$ and introduce the notations:
  \begin{align*}
    Q(p,\la)&=\tpflow{k}{\rho}{\la^*}{\la}p & X_k=\dpix{k}Q&,\, Y_k=\dpiy{k}Q.
  \end{align*}
  Using \eqref{e_commutingRelation} we can write:
  \[
  \fhomeo_k(\la)\tpflow{k}{\rho}{\la^*}{\la}{} = \rho(\la^*,\la)\fhomeo_k(\la^*);
  \]
  differentiating the previous expression with respect to $\la$ yields:
  \[
  \matrixtt{1}{0}{-\stepex{1k}}{\stepex{0k}}\vectort{\partial_\la X_0}{\partial_\la Y_0} + \vectort{0}{\partial_\la\dpix{k}} = \vectort{\partial_\la\rho^\forw}{\partial_\la\rho^\backw}.
  \]
  Hence we immediately obtain the differential equations:
  \begin{subequations}\label{eq_xiXiDifferentialEquations}
    \begin{align}
      \partial_\la X_0(p,\la) &= \partial_\la\rho^\forw(p,\la)\\
      \stepex{0k}(p,\la)\partial_\la Y_0(p,\la) &= -\sum_{i=1}^k\stepex{ik}(p,\la)\dot\phi(X_i(p,\la)) +\notag\\&\quad +\stepex{1k}(p,\la)\partial_\la\rho^\forw(p,\la) +\partial_\la\rho^\backw(p,\la).\label{e_differentialEquationForSharp}
    \end{align}
  \end{subequations}

  Assume that $\la'\in[\la^*,\infty)$   $p'\in\fnicePoints_k(\la')$, then by \eqref{e_definitionAlpha} there exists an open neighborhood $\nbhd\ni(p',\la')$ such that $\stepex{}$ is smooth in $\nbhd$ and $|\stepex{ij}(p,\la)|\geq 1$ if $i\not = j$ for all $(p,\la)\in\nbhd$.
  Then equations \eqref{eq_xiXiDifferentialEquations} have a unique solution $Q(p,\la)$ for $(Q(p,\la),\la)\in\nbhd$ with smooth dependence on the initial condition $p$.

  We now claim that if for some $0<l<k$ we have $X_l(p,\la)\in\ptcrit(\la)\setminus\cl\picrit(\la)$, where $\ptcrit(\la)$ is a $\bigo{\la^{-1/2}}$-neighborhood of $\pcrit(\la)$, then $\partial_\la X_l$ points away from $\picrit(\la)$. For ease of exposition we now drop $p$ and $\la$ from the notations; then compute:
  \[
  \partial_\la X_l = \stepex{0l}\partial_\la Y_0 + \sum_{i=1}^{l}\stepex{il}\dot\phi(X_i) -\stepex{1l}\partial_\la\rho^\forw
  \]
  and using \eqref{eq_xiXiDifferentialEquations} and $\stepex{ll}=0$ we obtain:
  \begin{align}\label{e_completeDifferentialEquation}
    \partial_\la{X_l} &= \sum_{j=1}^{l-1}\left(\stepex{jl}-\frac{\stepex{jk}\stepex{0l}}{\stepex{0k}}\right)\dot\phi(X_j) - \frac{\stepex{lk}\stepex{0l}}{\stepex{0k}}\dot\phi(X_l) - \sum_{j=l+1}^{k-1}\frac{\stepex{jk}\stepex{0l}}{\stepex{0k}}\dot\phi(X_j)+\notag\\&\quad + \left(\frac{\stepex{1k}\stepex{0l}}{\stepex{0k}}-\stepex{1l}\right)\partial_\la\rho^\forw + \frac{\stepex{0l}}{\stepex{0k}}\partial_\la\rho^\backw.
  \end{align}
  Then, by \eqref{e_definitionAlpha}  and using \eqref{eq_estimateHkHj} we have:
  \begin{align*}
    \left|\frac{\stepex{jk}\stepex{0l}}{\stepex{0k}}\right| &<
    \left|\la\slope{l}(q_l)\cdots\la\slope{j}(q_{j})\right|^{-1}\onepbigo{\la^{-1/2}} & \textrm{for }l\leq j< k;\\
    \left|\frac{\stepex{jk}\stepex{0l}}{\stepex{0k}} - \stepex{jl}\right| &<
    \left|\la\slope{j}(q_j)\cdots\la\slope{l}(q_l)\right|^{-1}\onepbigo{\la^{-1/2}} & \textrm{for }0< j \leq l;
  \end{align*}
  Thus:
  \begin{equation}\label{eq_differentialEquationForl}
    \partial_\la{X_l} + \frac{\dot\phi(x_l)}{\la\slope{l}(q_l)} < (\la\slope{l}(q_l))^{-1} \left(\|\partial_\la\rho^\forw\| + \|\partial_\la\rho^\backw\| + \bigo{\la^{-1/2}}\right).
  \end{equation}
  If moreover $x_l\in\ptcrit\setminus\picrit$, we have $|\dot\phi(x_l)|\sim 1$, $\la\slope{l}(q_l)\sim \la^{1/2}$, therefore since $\partial_\la\rho^\forw$ and $\partial_\la\rho^\backw$ are small, $\partial_\la{X_l}$ has the same sign as $\frac{\dot\phi(x_l)}{\la\slope{l}(q_l)}$, and this immediately implies our claim.

  The claim, in turn, implies that $\tpflow{k}{\rho}{\la^*}{\la}$ is well defined; in fact for any $p\in\fnicePoints_{k}(\la^*)$ there exists a unique solution $Q(p,\la)$ of equations \eqref{eq_xiXiDifferentialEquations} with boundary condition $p$; this solution extends indefinitely for $\la > \la^*$ and it is such that $Q(p,\la)\in\fnicePoints_k(\la)$. Finally, \eqref{e_estimateXk} follows from \eqref{eq_differentialEquationForl}; estimate \eqref{eq_differentialEquationForY0} follows immediately from \eqref{e_differentialEquationForSharp} whereas estimate \eqref{eq_differentialEquationForYN} follows from the analogous computation obtained using $\bhomeo_k$ in place of $\fhomeo_k$.
\end{proof}
We also define the companion semiflow:
\[
\tqflow{k}{\rho}{\la^*}{\la} = \sm^k\circ\tpflow{k}{\rho}{\la^*}{\la}\circ\sms^{-k}
\]
Notice that by definition the following diagram also commute:
\[
\begin{CD}
  \bnicePoints_k(\la^*) @>{\tqflow{k}{\rho}{\la^*}{\la}{}}>>\bnicePoints_k(\la)\\
  @VV\bhomeo_k(\la^*) V  @VV\bhomeo_{k}(\la) V\\
  \torus\times\torus @>\rho>> \torus\times\torus
\end{CD}
\]
Moreover we define the semiflows associated to the identity map
\begin{align*}
  \pflow{k}{\la^*}{\la} &= \tpflow{k}{\Id}{\la^*}{\la}{} & \qflow{k}{\la^*}{\la} &= \tqflow{k}{\Id}{\la^*}{\la}{},
\end{align*}

along with the following semiflow that we call the \emph{$k$-inclusion map}:
\[
\lflow{k}{\la^*}\la{}=\sm^{k+1}\pflow{2k+1}{\la^*}\la{}\sms^{-(k+1)}.
\]
Notice that the above map preserves the set $\hypset_k$, i.e. $\lflow{k}{\la^*}{\la}{\hypset_k(\la^*)}\subset{\hypset_k(\la)}$. We are now able to give the
\begin{proof}[Proof of property (a) in proposition~\ref{p_hyperbolicSet}]
  Using \eqref{e_completeDifferentialEquation} it follows that $\lflowmap{k}$ has a smooth limit for $k\to\infty$ that we denote by $\lflowmap{}$. Clearly $\lflowmap{}$ preserves $\hypset_\infty$ and defines the natural inclusion $\hypset_\infty(\la^*)\to\hypset_\infty(\la)$. Using once more \eqref{eq_differentialEquationForl} we can prove that $\lflowmap{}$ preserves the set $\phypset$ as well, and this concludes the proof.
\end{proof}

As a first application of the previous lemma, we show that the admissibility condition is monotone with respect to the parameter i.e.\ if $\la_1\leq\la_2$ then $\alphabet(\la_1)\subset\alphabet(\la_2)$. This will allow us to make some useful simplifications in the proof, as it will be described at the beginning of the next section.
\begin{prp}\label{l_continuationRectangles}
  Let $\letter\in\alphabet(\la_1)$; then for all $\la_2\geq\la_1$, we have that $\letter\in\alphabet(\la_2)$.
\end{prp}
\begin{proof}
  For ease of notation let $\rect_i=\recta\letter(\la_i)$, $\bbasei{i}=\bbasea{\letter}(\la_i)$ and $\fbasei{i}=\fbasea{\letter}(\la_i)$; we need to prove that the map $\homeo(\la_2):\rect_2\to\bbasei{2}\times\fbasei{2}$ is a diffeomorphism. For $\la\geq\la_1$ let $\tilde\rho^\forw(\la_1,\la)$ be the family of affine orientation preserving diffeomorphisms mapping $\bbasei{1}$ to $\bbasea{\letter}(\la)$; similarly let $\tilde\rho^\backw(\la_1,\la)$ be the family of affine orientation preserving diffeomorphisms which map $\fbasei{1}$ to $\fbasea{\letter}(\la)$.  We then smoothly extend $\tilde\rho^\backw$ and $\tilde\rho^\forw$ to be maps of the torus onto itself by making them act as the identity outside of a neighborhood of $\fbasei{1}$ and $\bbasei{1}$ respectively. It is then easy to check that we can take the extension satisfy the following estimate:
  \[
  \|\partial_\la{\rho^\forw}\|, \|\partial_\la{\rho^\backw}\| \sim \la^{-3/2}<1/4
  \]
  Hence we can apply proposition~\ref{l_continuationXiXi'}, which implies that $\homeo(\la_2)$ is a diffeomorphism, since, as we noted before we have $\homeo(\la)=\fhomeo_2(\la)\circ\sm^{-1}$.
\end{proof}
\begin{cor}\label{c_forwardAdmissible}
  Let $\itin$ be a $\la_1$-admissible closed word; then $\itin$ is $\la_2$-admissible for all $\la_2\geq\la_1$.
\end{cor}
As a further application of lemma~\ref{l_continuationXiXi'} we obtain some extremely useful estimates concerning the variation of the slope fields along the semiflows $\pflow{k}{\la^*}{\la}$ and $\qflow{k}{\la^*}{\la}$, which will be of utmost importance to control the multiplier of a periodic point.
\begin{prp} \label{l_partialHLa}
  Fix $\la\geq\la^*$ and $k\in\naturals$; let $p\in\fnicePoints_k(\la)$ and $q\in\bnicePoints_k(\la)$; then for any $0<l\leq k$ we have:
  \begin{subequations}
    \begin{align}
      \de{}{\la}\,\slope{l}(\qflow{k}{\la^*}{\la}q;\la) = & -2\la^{-2} + \bigo{|\la\slope{j-1}(\sm\qflow{k}{\la^*}{\la}q;\la)|^{-3}};\label{eq_partialHkLa}\\
      \de{}{\la}\,\slope{-l}(\pflow{k}{\la^*}{\la}p;\la) & = \bigo{|\la\slope{-j}(\pflow{k}{\la^*}{\la}p;\la)|^{-3}}.\label{eq_partialH-kLa}
    \end{align}
  \end{subequations}
\end{prp}
\begin{proof}
  Let us denote $p_j=(x_j,y_j)=\sm^jp$; since $\slope{0}\equiv \infty$  assume conventionally that $\deh\slope{0}=0$ and for any $l>0$ and $j$ compute:
  \begin{align*}
    \deh\slope{l}(p_j) &= \dddot\phi(x_j)\deh x_j - \frac{2}{\la^2}\left(1+\frac{1}{\la\slope{l-1}(p_{j-1})}\right)\deh\la +\\&\quad +\frac{1}{\la^2\slope{l-1}^2(p_{j-1})}\deh\slope{l-1}(p_{j-1});\\
    \intertext{Similarly:}
    \deh\slope{-l}(p_j) &= 2\la^{-1}\slope{-l}(p_j)(1-\la\slope{-l}(p_j))+\\&\quad + (\la\slope{-l}(p_j))^2\left[ \dddot\phi(x_{j+1})\deh x_{j+1} - \deh\slope{-l+1}(p_{j+1})\right].
  \end{align*}
  By definition of $\pflowmap{k}$ and $\qflowmap{k}$ we have $\deh x_0=0$, $\deh x_k=0$; then \eqref{e_estimateXk} implies that for $0<j<k$:
  \[
  \deh x_j=\bigo{|\la\slope{-j}(\sm^j\pflow{k}{\la^*}{\la}p;\la)|^{-1}}\deh\la.
  \] Therefore, iterating the previous expressions we obtain \eqref{eq_partialHkLa} and \eqref{eq_partialH-kLa}.
\end{proof}
The previous proposition immediately implies the following bound for $p\in\fnicePoints_k(\la)$:
\begin{equation}\label{e_partialStepexLa}
  \partial_\la\stepex{0k}(p,\la)=\la^{-1}\stepex{0k}(p,\la)\onepbigo{\la^{-1/2}}
\end{equation}
Additionally we collect in the following proposition some useful estimates concerning the variation of the reference slopes in the coordinates defined by $\fhomeo$ and $\bhomeo$. 
\begin{prp}\label{l_partialHxi}
  Fix $\la$ and let $k\in\naturals$; let $p\in\fnicePoints_k(\la)$ and $q\in\bnicePoints_k(\la)$; let further $(\xi,\xi')=\fhomeo_k p$ and $(\eta,\eta')=\bhomeo_k q$ then for any $0<j\leq k$ we obtain:
  \begin{align*}
    \partial_{\xi}\,\slope{-j}([\fhomeo_k]^{-1}(\xi,\xi'))&=\parSlope{\slope{-k}}{\slope{-j}}(p);\\
    \partial_{\xi'}\,\slope{-j}([\fhomeo_k]^{-1}(\xi,\xi'))&=\stepex{0k}^{-1}(p)\cdot\parSlope{\slope{0}}{\slope{-k}}(p);\\
    \partial_{\eta}\,\slope{j}([\bhomeo_k]^{-1}(\eta,\eta')) &=\stepex{0k}^{-1}(\sm^{-j} q) \parSlope{\slope{0}}{\slope{k}}(q);\\
    \partial_{\eta'}\,\slope{j}([\bhomeo_k]^{-1}(\eta,\eta'))&=\parSlope{\slope{k}}{\slope{j}}(q).
    \intertext{and moreover}
    \partial_{\xi}\,\stepex{0j}([\fhomeo_k]^{-1}(\xi,\xi'))&\leq\frac{1}{16}\stepex{0j}(p)\onepbigo{\la^{-1/2}}\\
    \partial_{\xi'}\,\stepex{0j}([\fhomeo_k]^{-1}(\xi,\xi'))&\leq\frac{1}{16}\stepex{0j}(p)\onepbigo{\la^{-1/2}}
  \end{align*}
\end{prp}
\begin{proof}
  The proof simply follows by the expressions for the differentials $\deh\fhomeo_k$ and $\deh\bhomeo_k$ and by the distortion estimate \eqref{e_smallDistortion}.
\end{proof}
Fix $\la$; then for $\itin$, $\fcyl$ and $\bcyl$ respectively a $\la$-admissible closed word, forward cylinder and backward cylinder of rank $r$, define the following quantities bounding the total expansion rates along the $x$ direction:
\begin{align*}
  \ndLL{\itin}(\la)&=\inf_{p\in\domain_\itin(\la)}|\stepex{0r+1}(p;\la)|& \sdLL{\itin}(\la)&=\sup_{p\in\domain_\itin(\la)}|\stepex{0r+1}(p;\la)|;\\
  \ncLL{{\fcyl}}(\la)&=\inf_{p\in\cdomain_{\fcyl}(\la)}|\stepex{0r+1}(p;\la)| &\scLL{{\fcyl}}(\la)&=\sup_{p\in\cdomain_{\fcyl}(\la)}|\stepex{0r+1}(p;\la)|;\\
  \ncLL{{\bcyl}}(\la)&=\inf_{p\in\cdomain_{\bcyl}(\la)}|\stepex{0r+1}(\sm^{-r-1}p;\la)|&\scLL{{\bcyl}}(\la)&=\sup_{p\in\cdomain_{\bcyl}(\la)}|\stepex{0r+1}(\sm^{-r-1}p;\la)|.
\end{align*}
By \eqref{e_definitionAlpha}, proposition~\ref{l_estimateReferenceSequence} and the definition of critical set we have:
\[
(2\pi\la)^r \geq \ndLL{\itin}(\la),\,\ncLL{\fcyl}(\la),\,\ncLL{\bcyl}(\la) \geq (4\la)^{r/2}
\]
Let $\Bcyl=\fcyl\cap\bcyl$ be $\la$-admissible bicylinder, then define
\begin{align*}
  \ncLL{\Bcyl}(\la) &= \left(\ncLL{\fcyl}(\la)^{-1}+\ncLL{\bcyl}(\la)^{-1}\right)^{-1};\\
  \scLL{\Bcyl}(\la) &= \left(\scLL{\fcyl}(\la)^{-1}+\scLL{\bcyl}(\la)^{-1}\right)^{-1}.
\end{align*}
Notice finally that by proposition~\ref{l_partialHxi} we have that
\[
\sdLL\itin=\ndLL\itin\onepbigo{\la^{-1/2}}
\]
and for $\la$-bulk cylinders $\fcyl$ and $\bcyl$:
\begin{align}\label{e_bulkEstimates}
  \sdLL\fcyl&=\ndLL\fcyl\onepbigo{\la^{-1}}&
  \sdLL\bcyl&=\ndLL\bcyl\onepbigo{\la^{-1}}
\end{align}
We can moreover prove that, given $\la$-bulk cylinders $\fcyl=\si\letter_1\cdots\letter_r$ and $\bcyl=\letter_{-r}\cdots\letter_{-1}\sii$ then:
\begin{align}\label{e_bulkApproximation}
  \ndLL\fcyl &= \prod_{j=1}^r \la\left|\slope{1}(p_{\letter_j})\right|\onepbigo{\la^{-1}}&
  \ndLL\bcyl &= \prod_{j=-r}^{-1} \la\left|\slope{1}(p_{\letter_j})\right|\onepbigo{\la^{-1}}.
\end{align}


%% file: Critical.pdf_tex

\begingroup
  \makeatletter
  \providecommand\color[2][]{%
    \errmessage{(Inkscape) Color is used for the text in Inkscape, but the package 'color.sty' is not loaded}
    \renewcommand\color[2][]{}%
  }
  \providecommand\transparent[1]{%
    \errmessage{(Inkscape) Transparency is used (non-zero) for the text in Inkscape, but the package 'transparent.sty' is not loaded}
    \renewcommand\transparent[1]{}%
  }
  \providecommand\rotatebox[2]{#2}
  \ifx\svgwidth\undefined
    \setlength{\unitlength}{160pt}
  \else
    \setlength{\unitlength}{\svgwidth}
  \fi
  \global\let\svgwidth\undefined
  \makeatother
  \begin{picture}(1,0.25)%
    \put(0,0){\includegraphics[width=\unitlength]{Critical.pdf}}%
    \put(0.51763638,0.17040018){\color[rgb]{0,0,0}\makebox(0,0)[b]{\smash{$\gpstr{\cs}{+}$}}}%
    \put(0.51763638,0.21040013){\color[rgb]{0,0,0}\makebox(0,0)[b]{\smash{$\gpstr{\cs}{-}$}}}%
    \put(0.32263649,0.02040031){\color[rgb]{0,0,0}\makebox(0,0)[b]{\smash{$\gxi\ds-$}}}%
    \put(0.54105959,0.02040031){\color[rgb]{0,0,0}\makebox(0,0)[b]{\smash{$\gxi\cs+$}}}%
    \put(0.76074463,0.02040031){\color[rgb]{0,0,0}\makebox(0,0)[b]{\smash{$\gxi\ds+$}}}%
    \put(0.10168945,0.02040031){\color[rgb]{0,0,0}\makebox(0,0)[b]{\smash{$\gxi\cs-$}}}%
    \put(0.3026365,0.07040028){\color[rgb]{0,0,0}\makebox(0,0)[b]{\smash{$\gpstr{\ds}{-}$}}}%
    \put(0.74251635,0.07040028){\color[rgb]{0,0,0}\makebox(0,0)[b]{\smash{$\gpstr{\ds}{+}$}}}%
    \put(0.05500001,0.10968019){\color[rgb]{0,0,0}\makebox(0,0)[rb]{\smash{$\torus$}}}%
  \end{picture}%
\endgroup

%% file: squ1.pdf_tex

\begingroup
  \makeatletter
  \providecommand\color[2][]{%
    \errmessage{(Inkscape) Color is used for the text in Inkscape, but the package 'color.sty' is not loaded}
    \renewcommand\color[2][]{}%
  }
  \providecommand\transparent[1]{%
    \errmessage{(Inkscape) Transparency is used (non-zero) for the text in Inkscape, but the package 'transparent.sty' is not loaded}
    \renewcommand\transparent[1]{}%
  }
  \providecommand\rotatebox[2]{#2}
  \ifx\svgwidth\undefined
    \setlength{\unitlength}{148.50006104pt}
  \else
    \setlength{\unitlength}{\svgwidth}
  \fi
  \global\let\svgwidth\undefined
  \makeatother
  \begin{picture}(1,1.00000008)%
    \put(0,0){\includegraphics[width=\unitlength]{squ1.pdf}}%
    \put(0.08491354,0.90647322){\color[rgb]{0,0,0}\makebox(0,0)[b]{\smash{$--$}}}%
    \put(0.50236985,0.90647322){\color[rgb]{0,0,0}\makebox(0,0)[b]{\smash{$+-$}}}%
    \put(0.92932055,0.90647322){\color[rgb]{0,0,0}\makebox(0,0)[b]{\smash{$--$}}}%
    \put(0.08491354,0.4909852){\color[rgb]{0,0,0}\makebox(0,0)[b]{\smash{$-+$}}}%
    \put(0.50236985,0.4909852){\color[rgb]{0,0,0}\makebox(0,0)[b]{\smash{$++$}}}%
    \put(0.92932055,0.4909852){\color[rgb]{0,0,0}\makebox(0,0)[b]{\smash{$-+$}}}%
    \put(0.08491354,0.06327482){\color[rgb]{0,0,0}\makebox(0,0)[b]{\smash{$--$}}}%
    \put(0.50236985,0.06327482){\color[rgb]{0,0,0}\makebox(0,0)[b]{\smash{$+-$}}}%
    \put(0.92932055,0.06327482){\color[rgb]{0,0,0}\makebox(0,0)[b]{\smash{$--$}}}%
    \put(0.50236985,-0.06285391){\color[rgb]{0,0,0}\makebox(0,0)[b]{\smash{$\gsqu\cs\cs\cdot\cdot$}}}%
  \end{picture}%
\endgroup

%% file: squ2.pdf_tex

\begingroup
  \makeatletter
  \providecommand\color[2][]{%
    \errmessage{(Inkscape) Color is used for the text in Inkscape, but the package 'color.sty' is not loaded}
    \renewcommand\color[2][]{}%
  }
  \providecommand\transparent[1]{%
    \errmessage{(Inkscape) Transparency is used (non-zero) for the text in Inkscape, but the package 'transparent.sty' is not loaded}
    \renewcommand\transparent[1]{}%
  }
  \providecommand\rotatebox[2]{#2}
  \ifx\svgwidth\undefined
    \setlength{\unitlength}{148.50006104pt}
  \else
    \setlength{\unitlength}{\svgwidth}
  \fi
  \global\let\svgwidth\undefined
  \makeatother
  \begin{picture}(1,1.00000008)%
    \put(0,0){\includegraphics[width=\unitlength]{squ2.pdf}}%
    \put(0.50236985,-0.06285391){\color[rgb]{0,0,0}\makebox(0,0)[b]{\smash{$\gsqu\cs\ds\cdot\cdot$}}}%
    \put(0.50236985,0.24862332){\color[rgb]{0,0,0}\makebox(0,0)[b]{\smash{$+-$}}}%
    \put(0.50236985,0.72369443){\color[rgb]{0,0,0}\makebox(0,0)[b]{\smash{$++$}}}%
    \put(0.91995274,0.72369443){\color[rgb]{0,0,0}\makebox(0,0)[b]{\smash{$-+$}}}%
    \put(0.0858717,0.72369443){\color[rgb]{0,0,0}\makebox(0,0)[b]{\smash{$-+$}}}%
    \put(0.0858717,0.24862332){\color[rgb]{0,0,0}\makebox(0,0)[b]{\smash{$--$}}}%
    \put(0.9288007,0.24862332){\color[rgb]{0,0,0}\makebox(0,0)[b]{\smash{$--$}}}%
  \end{picture}%
\endgroup

%% file: squ3.pdf_tex

\begingroup
  \makeatletter
  \providecommand\color[2][]{%
    \errmessage{(Inkscape) Color is used for the text in Inkscape, but the package 'color.sty' is not loaded}
    \renewcommand\color[2][]{}%
  }
  \providecommand\transparent[1]{%
    \errmessage{(Inkscape) Transparency is used (non-zero) for the text in Inkscape, but the package 'transparent.sty' is not loaded}
    \renewcommand\transparent[1]{}%
  }
  \providecommand\rotatebox[2]{#2}
  \ifx\svgwidth\undefined
    \setlength{\unitlength}{148.50006104pt}
  \else
    \setlength{\unitlength}{\svgwidth}
  \fi
  \global\let\svgwidth\undefined
  \makeatother
  \begin{picture}(1,1.00000008)%
    \put(0,0){\includegraphics[width=\unitlength]{squ3.pdf}}%
    \put(0.50236985,-0.06285391){\color[rgb]{0,0,0}\makebox(0,0)[b]{\smash{$\gsqu\ds\cs\cdot\cdot$}}}%
    \put(0.74509727,0.49062001){\color[rgb]{0,0,0}\makebox(0,0)[b]{\smash{$++$}}}%
    \put(0.2712466,0.49062001){\color[rgb]{0,0,0}\makebox(0,0)[b]{\smash{$-+$}}}%
    \put(0.2712466,0.90677965){\color[rgb]{0,0,0}\makebox(0,0)[b]{\smash{$--$}}}%
    \put(0.2712466,0.06327482){\color[rgb]{0,0,0}\makebox(0,0)[b]{\smash{$--$}}}%
    \put(0.74509727,0.06327482){\color[rgb]{0,0,0}\makebox(0,0)[b]{\smash{$+-$}}}%
    \put(0.74509727,0.90677965){\color[rgb]{0,0,0}\makebox(0,0)[b]{\smash{$+-$}}}%
  \end{picture}%
\endgroup

%% file: squ4.pdf_tex

\begingroup
  \makeatletter
  \providecommand\color[2][]{%
    \errmessage{(Inkscape) Color is used for the text in Inkscape, but the package 'color.sty' is not loaded}
    \renewcommand\color[2][]{}%
  }
  \providecommand\transparent[1]{%
    \errmessage{(Inkscape) Transparency is used (non-zero) for the text in Inkscape, but the package 'transparent.sty' is not loaded}
    \renewcommand\transparent[1]{}%
  }
  \providecommand\rotatebox[2]{#2}
  \ifx\svgwidth\undefined
    \setlength{\unitlength}{148.50006104pt}
  \else
    \setlength{\unitlength}{\svgwidth}
  \fi
  \global\let\svgwidth\undefined
  \makeatother
  \begin{picture}(1,1.00000008)%
    \put(0,0){\includegraphics[width=\unitlength]{squ4.pdf}}%
    \put(0.50236985,-0.06285391){\color[rgb]{0,0,0}\makebox(0,0)[b]{\smash{$\gsqu\ds\ds\cdot\cdot$}}}%
    \put(0.74514993,0.72387298){\color[rgb]{0,0,0}\makebox(0,0)[b]{\smash{$++$}}}%
    \put(0.74514993,0.24917848){\color[rgb]{0,0,0}\makebox(0,0)[b]{\smash{$+-$}}}%
    \put(0.27116554,0.24917848){\color[rgb]{0,0,0}\makebox(0,0)[b]{\smash{$--$}}}%
    \put(0.27116554,0.72387298){\color[rgb]{0,0,0}\makebox(0,0)[b]{\smash{$-+$}}}%
  \end{picture}%
\endgroup

%% file: rectdd.pdf_tex

\begingroup
  \makeatletter
  \providecommand\color[2][]{%
    \errmessage{(Inkscape) Color is used for the text in Inkscape, but the package 'color.sty' is not loaded}
    \renewcommand\color[2][]{}%
  }
  \providecommand\transparent[1]{%
    \errmessage{(Inkscape) Transparency is used (non-zero) for the text in Inkscape, but the package 'transparent.sty' is not loaded}
    \renewcommand\transparent[1]{}%
  }
  \providecommand\rotatebox[2]{#2}
  \ifx\svgwidth\undefined
    \setlength{\unitlength}{72.12489014pt}
  \else
    \setlength{\unitlength}{\svgwidth}
  \fi
  \global\let\svgwidth\undefined
  \makeatother
  \begin{picture}(1,1.2078432)%
    \put(0,0){\includegraphics[width=\unitlength]{rectdd.pdf}}%
    \put(0.85490208,0.482353){\color[rgb]{0,0,0}\makebox(0,0)[b]{\smash{$+$}}}%
    \put(0.72254901,0.83529421){\color[rgb]{0,0,0}\makebox(0,0)[b]{\smash{$-$}}}%
    \put(0.5803923,0.67058835){\color[rgb]{0,0,0}\makebox(0,0)[b]{\smash{$+$}}}%
    \put(0.50196084,0.83529421){\color[rgb]{0,0,0}\makebox(0,0)[b]{\smash{$-$}}}%
    \put(0.38431384,0.67058835){\color[rgb]{0,0,0}\makebox(0,0)[b]{\smash{$+$}}}%
    \put(0.30196083,0.83529421){\color[rgb]{0,0,0}\makebox(0,0)[b]{\smash{$-$}}}%
    \put(0.18039226,0.99215691){\color[rgb]{0,0,0}\makebox(0,0)[b]{\smash{$+$}}}%
    \put(0.50002705,-0.07327964){\color[rgb]{0,0,0}\makebox(0,0)[b]{\smash{$\grect\ds\cdot+\ds+$}}}%
  \end{picture}%
\endgroup

%% file: rectdc.pdf_tex

\begingroup
  \makeatletter
  \providecommand\color[2][]{%
    \errmessage{(Inkscape) Color is used for the text in Inkscape, but the package 'color.sty' is not loaded}
    \renewcommand\color[2][]{}%
  }
  \providecommand\transparent[1]{%
    \errmessage{(Inkscape) Transparency is used (non-zero) for the text in Inkscape, but the package 'transparent.sty' is not loaded}
    \renewcommand\transparent[1]{}%
  }
  \providecommand\rotatebox[2]{#2}
  \ifx\svgwidth\undefined
    \setlength{\unitlength}{72.12489014pt}
  \else
    \setlength{\unitlength}{\svgwidth}
  \fi
  \global\let\svgwidth\undefined
  \makeatother
  \begin{picture}(1,1.20784311)%
    \put(0,0){\includegraphics[width=\unitlength]{rectdc.pdf}}%
    \put(0.79654685,0.21096058){\color[rgb]{0,0,0}\makebox(0,0)[b]{\smash{$+$}}}%
    \put(0.64532153,0.31114782){\color[rgb]{0,0,0}\makebox(0,0)[b]{\smash{$-$}}}%
    \put(0.54143372,0.21096058){\color[rgb]{0,0,0}\makebox(0,0)[b]{\smash{$+$}}}%
    \put(0.4595327,0.31114782){\color[rgb]{0,0,0}\makebox(0,0)[b]{\smash{$-$}}}%
    \put(0.34455306,0.21096058){\color[rgb]{0,0,0}\makebox(0,0)[b]{\smash{$+$}}}%
    \put(0.22383039,0.31114782){\color[rgb]{0,0,0}\makebox(0,0)[b]{\smash{$-$}}}%
    \put(0.11845755,0.43876207){\color[rgb]{0,0,0}\makebox(0,0)[b]{\smash{$+$}}}%
    \put(0.50002705,-0.0732796){\color[rgb]{0,0,0}\makebox(0,0)[b]{\smash{$\grect\ds\cdot+\cs+$}}}%
  \end{picture}%
\endgroup

%% file: rectcd.pdf_tex

\begingroup
  \makeatletter
  \providecommand\color[2][]{%
    \errmessage{(Inkscape) Color is used for the text in Inkscape, but the package 'color.sty' is not loaded}
    \renewcommand\color[2][]{}%
  }
  \providecommand\transparent[1]{%
    \errmessage{(Inkscape) Transparency is used (non-zero) for the text in Inkscape, but the package 'transparent.sty' is not loaded}
    \renewcommand\transparent[1]{}%
  }
  \providecommand\rotatebox[2]{#2}
  \ifx\svgwidth\undefined
    \setlength{\unitlength}{72.12488403pt}
  \else
    \setlength{\unitlength}{\svgwidth}
  \fi
  \global\let\svgwidth\undefined
  \makeatother
  \begin{picture}(1,1.20784321)%
    \put(0,0){\includegraphics[width=\unitlength]{rectcd.pdf}}%
    \put(0.84159581,0.87903074){\color[rgb]{0,0,0}\makebox(0,0)[b]{\smash{$+$}}}%
    \put(0.63639625,0.61559889){\color[rgb]{0,0,0}\makebox(0,0)[b]{\smash{$-$}}}%
    \put(0.55043424,0.87903077){\color[rgb]{0,0,0}\makebox(0,0)[b]{\smash{$+$}}}%
    \put(0.44506151,0.61559889){\color[rgb]{0,0,0}\makebox(0,0)[b]{\smash{$-$}}}%
    \put(0.37296419,0.87903077){\color[rgb]{0,0,0}\makebox(0,0)[b]{\smash{$+$}}}%
    \put(0.20935907,0.61559889){\color[rgb]{0,0,0}\makebox(0,0)[b]{\smash{$-$}}}%
    \put(0.16499162,1.04818184){\color[rgb]{0,0,0}\makebox(0,0)[b]{\smash{$+$}}}%
    \put(0.87209857,0.39376133){\color[rgb]{0,0,0}\makebox(0,0)[b]{\smash{$-$}}}%
    \put(0.50002709,-0.0732796){\color[rgb]{0,0,0}\makebox(0,0)[b]{\smash{$\grect\cs\cdot+\ds+$}}}%
  \end{picture}%
\endgroup

%% file: rectcc.pdf_tex

\begingroup
  \makeatletter
  \providecommand\color[2][]{%
    \errmessage{(Inkscape) Color is used for the text in Inkscape, but the package 'color.sty' is not loaded}
    \renewcommand\color[2][]{}%
  }
  \providecommand\transparent[1]{%
    \errmessage{(Inkscape) Transparency is used (non-zero) for the text in Inkscape, but the package 'transparent.sty' is not loaded}
    \renewcommand\transparent[1]{}%
  }
  \providecommand\rotatebox[2]{#2}
  \ifx\svgwidth\undefined
    \setlength{\unitlength}{72.12488403pt}
  \else
    \setlength{\unitlength}{\svgwidth}
  \fi
  \global\let\svgwidth\undefined
  \makeatother
  \begin{picture}(1,1.20784321)%
    \put(0,0){\includegraphics[width=\unitlength]{rectcc.pdf}}%
    \put(0.73622308,0.41594505){\color[rgb]{0,0,0}\makebox(0,0)[b]{\smash{$+$}}}%
    \put(0.48942896,0.41594511){\color[rgb]{0,0,0}\makebox(0,0)[b]{\smash{$+$}}}%
    \put(0.31195891,0.41594511){\color[rgb]{0,0,0}\makebox(0,0)[b]{\smash{$+$}}}%
    \put(0.89428229,0.21629145){\color[rgb]{0,0,0}\makebox(0,0)[b]{\smash{$-$}}}%
    \put(0.57539097,0.22183736){\color[rgb]{0,0,0}\makebox(0,0)[b]{\smash{$-$}}}%
    \put(0.38405623,0.22183736){\color[rgb]{0,0,0}\makebox(0,0)[b]{\smash{$-$}}}%
    \put(0.1483538,0.22183736){\color[rgb]{0,0,0}\makebox(0,0)[b]{\smash{$-$}}}%
    \put(0.50002709,-0.0732796){\color[rgb]{0,0,0}\makebox(0,0)[b]{\smash{$\grect\cs\cdot+\cs+$}}}%
  \end{picture}%
\endgroup

%% file: hyper1.pdf_tex

\begingroup
  \makeatletter
  \providecommand\color[2][]{%
    \errmessage{(Inkscape) Color is used for the text in Inkscape, but the package 'color.sty' is not loaded}
    \renewcommand\color[2][]{}%
  }
  \providecommand\transparent[1]{%
    \errmessage{(Inkscape) Transparency is used (non-zero) for the text in Inkscape, but the package 'transparent.sty' is not loaded}
    \renewcommand\transparent[1]{}%
  }
  \providecommand\rotatebox[2]{#2}
  \ifx\svgwidth\undefined
    \setlength{\unitlength}{59.2pt}
  \else
    \setlength{\unitlength}{\svgwidth}
  \fi
  \global\let\svgwidth\undefined
  \makeatother
  \begin{picture}(1,0.95608108)%
    \put(0,0){\includegraphics[width=\unitlength]{hyper1.pdf}}%
  \end{picture}%
\endgroup

%% file: ficriti.pdf_tex

\begingroup
  \makeatletter
  \providecommand\color[2][]{%
    \errmessage{(Inkscape) Color is used for the text in Inkscape, but the package 'color.sty' is not loaded}
    \renewcommand\color[2][]{}%
  }
  \providecommand\transparent[1]{%
    \errmessage{(Inkscape) Transparency is used (non-zero) for the text in Inkscape, but the package 'transparent.sty' is not loaded}
    \renewcommand\transparent[1]{}%
  }
  \providecommand\rotatebox[2]{#2}
  \ifx\svgwidth\undefined
    \setlength{\unitlength}{66.26624756pt}
  \else
    \setlength{\unitlength}{\svgwidth}
  \fi
  \global\let\svgwidth\undefined
  \makeatother
  \begin{picture}(1,0.87892047)%
    \put(0,0){\includegraphics[width=\unitlength]{ficriti.pdf}}%
    \put(0.13430787,-0.09356197){\color[rgb]{0,0,0}\makebox(0,0)[b]{\smash{$\fcritiUnion_0$}}}%
    \put(0.50063312,-0.09356197){\color[rgb]{0,0,0}\makebox(0,0)[b]{\smash{$\fcritiUnion_1$}}}%
    \put(0.86695836,-0.09356197){\color[rgb]{0,0,0}\makebox(0,0)[b]{\smash{$\fcritiUnion_2$}}}%
  \end{picture}%
\endgroup

%% file: bicriti.pdf_tex

\begingroup
  \makeatletter
  \providecommand\color[2][]{%
    \errmessage{(Inkscape) Color is used for the text in Inkscape, but the package 'color.sty' is not loaded}
    \renewcommand\color[2][]{}%
  }
  \providecommand\transparent[1]{%
    \errmessage{(Inkscape) Transparency is used (non-zero) for the text in Inkscape, but the package 'transparent.sty' is not loaded}
    \renewcommand\transparent[1]{}%
  }
  \providecommand\rotatebox[2]{#2}
  \ifx\svgwidth\undefined
    \setlength{\unitlength}{66.26624756pt}
  \else
    \setlength{\unitlength}{\svgwidth}
  \fi
  \global\let\svgwidth\undefined
  \makeatother
  \begin{picture}(1,0.87892047)%
    \put(0,0){\includegraphics[width=\unitlength]{bicriti.pdf}}%
    \put(0.13430787,-0.09356197){\color[rgb]{0,0,0}\makebox(0,0)[b]{\smash{$\bcritiUnion_0$}}}%
    \put(0.50063312,-0.09356197){\color[rgb]{0,0,0}\makebox(0,0)[b]{\smash{$\bcritiUnion_1$}}}%
    \put(0.86695836,-0.09356197){\color[rgb]{0,0,0}\makebox(0,0)[b]{\smash{$\bcritiUnion_2$}}}%
  \end{picture}%
\endgroup

%% file: domains.pdf_tex

\begingroup
  \makeatletter
  \providecommand\color[2][]{%
    \errmessage{(Inkscape) Color is used for the text in Inkscape, but the package 'color.sty' is not loaded}
    \renewcommand\color[2][]{}%
  }
  \providecommand\transparent[1]{%
    \errmessage{(Inkscape) Transparency is used (non-zero) for the text in Inkscape, but the package 'transparent.sty' is not loaded}
    \renewcommand\transparent[1]{}%
  }
  \providecommand\rotatebox[2]{#2}
  \ifx\svgwidth\undefined
    \setlength{\unitlength}{41.79624939pt}
  \else
    \setlength{\unitlength}{\svgwidth}
  \fi
  \global\let\svgwidth\undefined
  \makeatother
  \begin{picture}(1,1.39349253)%
    \put(0,0){\includegraphics[width=\unitlength]{domains.pdf}}%
    \put(0.21293957,-0.1004875){\color[rgb]{0,0,0}\makebox(0,0)[b]{\smash{$\domain_{\itin}\subset\cdomain_{\fcyl}$}}}%
    \put(0.8771538,-0.10048771){\color[rgb]{0,0,0}\makebox(0,0)[b]{\smash{$\sm^{r+1}\domain_{\itin'}\subset\cdomain_{\bcyl}$}}}%
  \end{picture}%
\endgroup

%% file: proof.tex
\input{proof.p}
\input{riemannlemma.p}
\section{Proof of main results}\label{s_proof}
In this section we give the proof of our main results; the first step is to cut the parameter line in overlapping intervals of length $\bigo{1}$ and then prove our results for $\la$ in each of these parameter intervals; in order to do so for, $n\in\naturals$ let
\begin{align*}
  \laa(n)&=\frac{1}{2}n-\frac{7}{16} &\lab(n)&=\frac{1}{2}n+\frac{7}{16};
\end{align*}
obviously, the union over all $n$ of the intervals $\J(n)=[\laa(n),\lab(n)]$ covers
$\reals^+$; we will prove our main theorem for $\la\in\J(n)$ for arbitrary, but large
enough, $n$. Fix $n$ and drop it from the notations; let $\J=\J(n)$ be the \emph{main
  parameter set} and denote the normalized Lebesgue measure on $\J$ with the symbol
$\prob$. We also define the \emph{main core parameter set}
  \(
  \hat\J=\ball{3/8}{n/2}\subset\J.
  \)
We now exploit the fact that admissibility is a monotonic property: let $\lwords{r}=\cwords r(\laa)$ and $\uwords{r}=\cwords{r}(\lab)$; notice that proposition~\ref{c_forwardAdmissible} implies:
\begin{equation}\label{e_ulboundForD}
  \lwords r\subset\cwords r(\la)\subset\uwords r\ \fa\la\in\J.
\end{equation}
If $\itin\in\uwords r$, then possibly $\itin$ is not admissible for some values of $\la\in\J$; thus we define the admissibility interval:
\[
\J^\admissible_\itin=\{\la\in\J\st\itin \textrm{ is $\la$-admissible}\}.
\]
On the other hand if $\itin\in\lwords r$ we have $\J^\admissible_\itin=\J$; we give similar definitions for cylinders:
\[
\J^\admissible_\Bcyl=\{\la\in\J\st\Bcyl\text{ is $\la$-admissible}\}.
\]
Moreover we define:
\begin{align*}
  \ndLL\itin&=\inf_{\la\in\J^\admissible_\itin}\ndLL\itin(\la)&  \sdLL\itin&=\sup_{\la\in\J^\admissible_\itin}\sdLL\itin(\la)\\
  \ndLL\Bcyl&=\inf_{\la\in\J^\admissible_\Bcyl}\ndLL\Bcyl(\la)&  \sdLL\Bcyl&=\sup_{\la\in\J^\admissible_\Bcyl}\sdLL\Bcyl(\la)
\end{align*}

For any $N$ and $\itin\in\uwords {N-1}$ define the set of parameters for which we have an elliptic periodic point belonging to a given critical domain:
\[
\E_1(N,\itin)=\{\la\in\J^\admissible_\itin \st \ex p\in\domain_\itin(\la) \textrm{ cyclicity 1 elliptic $N$-periodic point for }\sm \}.
\]
Furthermore, define:
\[
\E_1(N) =\bigcup_{\itin\in\uwords {N-1}}\E_1(N,\itin).
\]
Let $\la\in\J$ be such that there exists $p\in\crit(\la)$ a cyclicity 1 elliptic $N$-periodic point for $\sm$; by corollary~\ref{c_s1ppAdmissible} the point $p$ must belong to some domain, hence \eqref{e_ulboundForD} implies that $\la\in \E_1(N)$.
In the following subsections we prove two lemmata, which are the main technical ingredients of our main theorem; the first one gives an upper bound in measure in the parameter space for existence of a cyclicity one elliptic periodic point of fixed itinerary and it furthermore proves that we have $1-1$ correspondence between admissible closed word of rank $N-1$ and elliptic $N$-periodic points of cyclicity one. The second lemma is rather a lower bound; we show that in any small ball in parameter space we can find a smaller ball, with a proper lower bound on the diameter, for which there exists a cyclicity one elliptic periodic orbit of appropriate period which shadows for some small (logarithmic) fraction of its period an assigned hyperbolic orbit. This lies the foundation for the construction of a positive Hausdorff dimension set of parameters admitting infinitely many elliptic islands of cyclicity one.
\begin{lem}\label{l_estimateParameterSpace}
  Fix $N$ and let $\itin\in\uwords {N-1}$;  then:
  \begin{enumerate}[(a)]
  \item the following estimate holds:
    \begin{equation}\label{e_estimateParameterSpace}
      \prob (\E_1(N,\itin))\leq\const\cdot\laa^{-1}\ndLL\itin^{-2}
    \end{equation}
  \item there can be at most one elliptic $N$-periodic point of cyclicity 1 in each  $\domain_\itin(\la)$.
  \end{enumerate}
\end{lem}
\begin{lem}\label{l_fatDensityComplete}
  For any fixed $\J$ there exist two sequences
  \begin{align*}
    \eps_N &= C\laa^{-(N-\lfloor\log N\rfloor-5)} & \eps'_N &= C'\lab^{-4N}
  \end{align*}
  such that, for any ball $\mathcal{B}\subset\hat\J$ of diameter larger than $\eps_N$ and any open proper word $\itin$ of rank $r\leq\lfloor\log N\rfloor$, there exists a ball $\mathcal{B'}\subset\mathcal{B}$ of diameter larger than $\eps'_N$ such that if $\la\in\mathcal{B'}$ there exists a point $p\in\hypset_r(\la)$ that is the center of an elliptic island of period $N$ and $\itin_r(p)=\itin$.
\end{lem}

\begin{proof}[Proof of the main theorem]
  The proof of item (a) follows from this simple summation lemma; the proof is given in appendix~\ref{appendixRL}
  \begin{lem}\label{l_technicalSummation}
    Let $\indexset$ be a index set and $\{\delta_i\}_{i\in\indexset}$ and $\{\RiemFunc_i\}_{i\in\indexset}$ be positive real numbers; assume that there exist real numbers $\rslc_1,\rslc_2>0$, $0<\spread<1$ and $0<\expo\leq 1$ satisfying the following properties:
    \begin{itemize}
    \item $\rslc_1\cdot\spread<\RiemFunc_i< \rslc_1 \ \fa i\in\indexset$;
    \item for $\cursor\in(0,1)$, define the set $\indexset_\cursor=\{i\in\indexset\st \RiemFunc_i>\rslc_1\cdot\spread^{1-\cursor}\}$; then: 
      \[
      \sum_{i\in\indexset_\cursor} \delta_i < \rslc_2\cdot\spread^{\expo\cursor}.
      \]
    \end{itemize}
    Then $\fa \eps>0$ sufficiently small there exists $C_\eps\sim \rslc_1\rslc_2 \cdot \eps^{-1}$ such that: 
    \[
    \sum_i \RiemFunc_i\delta_i < C_\eps \cdot \spread^{\expo-\eps}.
    \]
  \end{lem}
  In fact each admissible itinerary $\itin=\si\letter_1\cdots\letter_{N-1}\sii$ is in one to one correspondence with a choice of $N-1$ admissible rectangles; thus proposition~\ref{p_rectCardinalityBound} immediately implies that the number of admissible domains of rank $r$ is bounded by $(4\la)^r$. By construction we know that
  \[
  \ndLL\itin >\const\prod_j\la|\slope{1}(p_{\letter_j})|.
  \]
  We want to apply lemma~\ref{l_technicalSummation} to the following expression:
  \[
  \prob(\E_1(N)) = \sum_{\itin\in\uwords r}\prob(\E_1(N,\itin));
  \]
  in fact for each $1 \leq j \leq N-1$ let $\indexset_j$ be the appropriate alphabet, define $\RiemFunc_\letter$ as $(\la|\slope{1}(p_{\letter}|)^{-2}$ and let $\delta_\letter=1$. Then by proposition~\ref{p_rectCardinalityBound}, the hypotheses of lemma~\ref{l_technicalSummation} hold with:
  \begin{align*}
    \rslc_1&=\const\,\la^{-1}&\rslc_2&=\const\,\la&\spread&=\const\,\la^{-1}&\expo&=1,
  \end{align*}
  hence, by lemma~\ref{l_estimateParameterSpace}a we have, for all small $\eps>0$:
  \begin{align}
    \sum_{\itin\in\uwords r}\prob(\E_1(N,\itin))&<\const\,\la^{-1}\sum_{\letter_1}\cdots\sum_{\letter_{N-1}}\RiemFunc_{\letter_1}\cdots\RiemFunc_{\letter_{N-1}};\notag\\
    &<C_\eps \la^{-1}\la^{-(N-1)(1-\eps)}<C_\eps\la^{-N(1-\eps)}\label{e_optimalENEstimate}
  \end{align}
  with $C_\eps<\const\,\eps^{-N}$. Furthermore, by lemma~\ref{l_estimateParameterSpace}b we have that there can be at most finitely many elliptic periodic points of bounded period and cyclicity 1, hence
  \[
  \{\la\in\J\st\sm \textrm{has infinitely many elliptic p.p.\ of cyclicity $1$}\} = \limsup_{N\to\infty} \E_1(N)
  \]
  then by the Borel-Cantelli lemma, \eqref{e_optimalENEstimate} implies item (a) of the Main Theorem.
  \begin{rmk}
    Notice that \eqref{e_optimalENEstimate} implies that the density of parameters for which we have at least one elliptic island of cyclicity $1$ goes to zero as $\la\to\infty$; this implies Carleson conjecture for elliptic islands of cyclicity one. It is interesting to note that in \cite{Carleson} it is conjectured that \eqref{e_estimateParameterSpace} should hold with $\ndLL\itin^{-3}$ on the right hand side; in our proof it appears that \eqref{e_estimateParameterSpace} is a sharp bound.
  \end{rmk}

  In order to prove item (b) we need the following general result; the proof is given in appendix~\ref{appendixHD}.
  \begin{lem}\label{l_technicalHD}
    For any $\la\in\J$ and $n\in\naturals$, denote by $P_n(\la)$ some property of the
    element $\la$ that depends on $n$. Assume that there exist two sequences
    $\eps_n>\eps'_n$ satisfying the fast decreasing property:
    \begin{align}\label{e_fastDecreasing}
     \eps_0&<\diam\hat\J&      \eps'_n&>\eps_{n+1}^{1/n},
    \end{align}
    such that for any ball $\mathcal{B}\subset\hat\J$ of diameter at least
    $\eps_N$ there exists a smaller ball $\mathcal{B}'\subset\mathcal{B}$ of diameter at
    least $\eps'_N$ such that property $P_N(\la)$ holds for all $\la\in\mathcal{B}'$. Then
    there exists a residual set $\fatset\subset\hat\J$ of Hausdorff dimension bounded below
    by:
    \begin{equation}\label{e_lowerBoundHD}
      \dim_H\fatset\geq\liminf_{n\to\infty}\frac{\log\eps'_n}{\log\eps_n}
    \end{equation}
    such that, for any $\la\in\fatset$, property $P_n(\la)$ holds for all $n$ larger than
    some $\bar n(\la)$.
  \end{lem}
  Lemma~\ref{l_fatDensityComplete} yields two sequences $\eps_N$ and $\eps'_N$: extract
  subsequences $\eps_{N_k}$ and $\eps'_{N_k}$ such that conditions
  \eqref{e_fastDecreasing} hold and define the properties $P_k$ as follows:
  \begin{align*}
    P_k(\la)= \ex p\in\hypset_{r_k}(\la)\st& \itin_{r_k}(p)=\itin^k,\\&p\ \text{is the center of a cyclicity 1 elliptic island of period}\ N_k
  \end{align*}
  where $r_k$ and $\itin^k$ are defined as follows.
  Let $\bar r_1=\lfloor \log N_1\rfloor$; and let $\itin^1,\cdots,\itin^{k_1}$ be an enumeration of all $\J$-proper open words of rank $\bar r_1$; we let $r_1 = \cdots = r_{k_1} = \bar r_1$.
  Let then $\bar r_2=\lfloor \log N_{k_1+1} \rfloor$ and let $\itin^{k_1+1},\cdots,\itin^{k_1+k_2}$ be an enumeration of all $\J$-proper open words of rank $\bar r_2$; we let $r_{k_1+1} = \cdots = r_{k_1+k_2} = \bar r_2$. By proceeding in this way we obtain infinite sequences $r_k$ and $\itin^k$.

  Let $\fatset$ be the set obtained by lemma~\ref{l_technicalHD}; we now show that for any $\la\in\fatset$ and any $p\in\hypset(\la)$ there is a sequence of center of elliptic islands of $\sm$ converging to $p$.
  In fact, if $p\in\hypset(\la)$, there exists by construction a subsequence of open words $\itin^{k_1},\itin^{k_2},\cdots$ such that $p\in\rectUnion_{\itin^{k_l}}(\la)$ for all $l\in\naturals$; but by lemma~\ref{l_technicalHD}, for each $l$ there exists $q_l$ the center of an elliptic island of period $N_{k_l}$ for $\sm$ such that $\itin_{r_{k_l}}(q_l)=\itin^{k_l}$. The proof then follows by proposition~\ref{l_exponentialCantorEstimate}.
\end{proof}
\subsection{Proof of lemmata~\ref{l_estimateParameterSpace} and~\ref{l_fatDensityComplete}}
A periodic orbit of cyclicity $1$ can intersect $\crit$ in a single point $p$, which therefore belongs to either $\crit^+$ or $\crit^-$; therefore the itinerary $\itin$ of $p$ necessarily satisfies the condition $\bbasei{\itin}=\fbasei{\itin}$; to fix ideas let $\basec=\cl\cpcrit-$ and take $\bbasei{\itin}=\fbasei{\itin}=\basec$. All closed words $\itin$ and cylinders $\bcyl$ and $\fcyl$ that we will consider are henceforth assumed to be such that $\bbasei{\fcyl}=\fbasei{\bcyl}=\basec$; the case $\basec=\cl\cpcrit+$ can be treated in a similar way and it will not be explicitly considered.

The proofs of lemmata~\ref{l_estimateParameterSpace} and~\ref{l_fatDensityComplete} follow from a fairly involved common construction that is split for convenience in several stages. We first introduce some notation and then outline the strategy behind the construction.
Given $\basec^\prime\subset\basec$ we can define the restrictions
\begin{align*}
  \left.\domain_\itin\right|_{\basec^\prime}&=\domain_\itin\cap (\fhomeo_r)^{-1}(\basec^\prime\times\basec^\prime);\\
  \left.\cdomain_{\fcyl}\right|_{\basec^\prime}&=\domain_{\fcyl}\cap (\fhomeo_r)^{-1}(\basec^\prime\times\fbasei{\fcyl});\\
  \left.\cdomain_{\bcyl}\right|_{\basec^\prime}&=\domain_{\bcyl}\cap (\bhomeo_r)^{-1}(\bbasei{\bcyl}\times\basec^\prime).
\end{align*}
Further, assume $p\in\crit(\la)$ is a cyclicity $1$ periodic point for $\sm$; then we can write the linearization of $\sm^N$ in $p$ as follows:
\[
\deh\sm^N(p)=
\matrixtt{0}{1}{-1}{\la\slope{N}(p)}
\matrixtt{\stepex{0N}^{-1}(p)}{0}{0}{\stepex{0N}(p)}
\matrixtt{1}{0}{-\la\slope{-N}(p)}{1}.
\]
We then conclude that the point $p$ is an elliptic periodic point if and only if the following inequality holds true:
\begin{equation}\label{eq_traceCondition}
  |\textup{Tr}(\deh\sm^N(p))|=|\stepex{0N}(p)\la(\slope{N}(p)-\slope{-N}(p))|< 2.
\end{equation}
We can then control the ellipticity of a periodic point $p$ by controlling the value of the reference slope fields on $p$; to this extent the estimates we obtained in the previous section will be most useful. For $\itin\in\uwords {N-1}$, $\la\in\J^\admissible_\itin$, define the following set:
\[
\basec^\elliptic_{\itin,\la}=\{\xi\in\basec(\la)\st \pd\itin(\xi,\xi;\la)=\qd\itin(\xi,\xi;\la)=p,\ \la|\slope N(p)-\slope{-N}(p)|<2\ndLL\itin^{-1}\}
\]
let furthermore
\begin{align*}
  \basec^\elliptic_{\itin} &= \bigcup_{\la\in\J^\admissible_\itin}\basec^\elliptic_{\itin,\la} &
  \basec^\elliptic_{\Bcyl} &= \bigcup_{\itin\in\Bcyl}\basec^\elliptic_{\itin}.
\end{align*}
Define:
\begin{align*}
  \J^\elliptic_\itin&=\{\la\in\J^\admissible_\itin\st\domain_\itin^{\elliptic}(\la)\cap\sm^{-N}\domain_\itin^\elliptic(\la)\not =\emptyset\}\\
  \J^\elliptic_\Bcyl&=\{\la\in\J^\admissible_\Bcyl\st\cdomain_{\fcyl}^\elliptic(\la)\cap\cdomain_{\bcyl}^\elliptic(\la)\not =\emptyset\}
\end{align*}
where $\domain_\itin^\elliptic(\la)=\domain_\itin(\la)|_{\basec^\elliptic_\itin}$, $\cdomain_{\fcyl}^\elliptic(\la)=\cdomain_{\fcyl}(\la)|_{\basec^\elliptic_{\fcyl}}$ and $\cdomain_{\bcyl}^\elliptic(\la)=\cdomain_{\bcyl}(\la)|_{\basec^\elliptic_{\bcyl}}$.
Notice that if $\Bcyl'\subset\Bcyl$ are two $\la$-admissible bicylinders, by definition we have that
\begin{align*}
  \basec^\elliptic_{\Bcyl'}&\subset  \basec^\elliptic_{\Bcyl} & \J^\elliptic_{\Bcyl'}&\subset  \J^\elliptic_{\Bcyl};
\end{align*}
By \eqref{eq_traceCondition} it is clear that if $p\in\domain_\itin$ is elliptic for $\sm$, then necessarily $\pix p\in\basei{\itin}^{\elliptic}$ and $\la\in\J_\itin^\elliptic$. By definition we then conclude that $\E(N,\itin) \subset\J_\itin^{\elliptic}$. In order to prove lemma~\ref{l_estimateParameterSpace} we will give an upper bound for the measure of the sets $\J_\itin^\elliptic$ for all $\itin$. Along with the upper bound, we will obtain, for $\itin$ belonging to a special class of words, a way to construct a suitably large interval of parameters that is indeed contained in $\E(N,\itin)$. At the same time we will be able to prove that the density of such intervals grows at an exponential rate in the period $N$. Additionally, the density grows exponentially even if we require that the corresponding periodic orbit shadows for some time a chunk of a given hyperbolic orbit. Combining these three observations we will obtain the proof of lemma~\ref{l_fatDensityComplete}.


\paragraph{Stage 1} We bootstrap the set $\basec$ in order to make it independent of $\la$, conveniently small and at the same time large enough to contain $\basec_\itin^\elliptic$ for all $\itin$.
Define
\[
\basec^* = \{\xi\in\basec(\laa)\st -4/\laa\leq \ddot\phi(\xi)\leq 0\};
\]
we then claim that for any $N$ and any $\itin\in\uwords{N-1}$ we have $\basec_\itin^\elliptic\subset\basec^*$. In fact if $\xi\in\basec_\itin^\elliptic$, then there exists $\la\in\J_\itin^\admissible$ such that $\pd\itin(\xi,\xi;\la) = \qd\itin(\xi,\xi;\la)=p$ and
\[
\la|\slope{N}(p)-\slope{-N}(p)|<2\ndLL\itin^{-1};
\]
if $N=1$, we have $\ndLL\itin=1$, which implies by definition that $\xi\in\basec^*$; otherwise notice that $\slope{N}$ and $\slope{-N}$ satisfy the following estimates:
\begin{subequations}\label{e_obviousEstimatesHn}
  \begin{align}
    |\slope{N}(p)-\slope{1}(p)|&=\la^{-2}|\slope{1}(\sm^{-1}p)|^{-1}\cdot\onepbigo{\laa^{-1/2}}\\
    |\slope{-N}(p)-\slope{-1}(p)|&=\la^{-2}|\slope{1}(\sm p)|^{-1}\cdot\onepbigo{\laa^{-1/2}}
  \end{align}
\end{subequations}
Hence:
\[
\la|\slope{1}(p)-\slope{-1}(p)| = \bigo{\la^{-1/2}} < 2
\]
which implies that $\xi\in\basec^*$. We obtain, by construction and by definition of $\phi$ that, for any $\xi\in\basec^*$:
\begin{subequations}
  \begin{align}
    \dot\phi(\xi)&= 1+\bigo{\laa^{-2}};\label{e_estimateOnDotPhi}\\
    \dddot\phi(\xi)&= - 4\pi^2\onepbigo{\laa^{-2}};\label{e_estimateOnDddotPhi}
  \end{align}
\end{subequations}
and the diameter estimate:
\begin{equation*}
  \diam\basec^*< \pi^{-2}\laa^{-1}\onepbigo{\laa^{-2}}.
\end{equation*}

\paragraph{Stage 2} Given a bicylinder $\Bcyl$, we further reduce the set $\basec^*$ to a smaller set $\basei{\Bcyl}$ still satisfying $\basec_\itin^\elliptic\subset\basei{\Bcyl}$ for all $\itin\in\Bcyl$. This will also allow us to define a set $\J_\Bcyl\subset\J$ containing all parameters $\la$ for which $\sm$ admits an elliptic periodic point of cyclicity $1$ with itinerary $\itin\in\Bcyl$.

The idea behind this reduction is simple to describe in words: recall that all closed words belonging to a bicylinder of rank $r$ have fixed first $r$ and last $r$ symbols; this implies that we control the first $r$ and last $r$ points of the orbit $\{\sm^jp\}$ for any $p$ belonging to the corresponding critical domain, which gives in turn good control on $\slope{r}$ and $\slope{-r}$ and thus, by proposition~\ref{l_estimateReferenceSequence}, on $\slope{N}$ and $\slope{-N}$ up to some small error. This allows to find an approximation of $\basec^\elliptic_{\itin}$ whose accuracy improves with the rank $r$ of the bicylinder.

We need to introduce further notation: given $\Bcyl=\fcyl\cap\bcyl$ a $\lab$-admissible bicylinder of rank $r$, we introduce the associated \emph{leaf functions} $\leafPrivate$, whose graphs foliate the corresponding domains $\cdomain_{\fcyl}$ or $\cdomain_{\bcyl}$, in the following way: for any  $\la\in\J_\Bcyl^\admissible$, $\zeta\in\fbasei\fcyl(\la)$ and $\eta\in\bbasei\bcyl(\la)$, define
\begin{align*}
  \fleaf{\fcyl}{\zeta}{\la}&:\basec^*\to\torus\st\pd\fcyl(\xi,\zeta;\la)=(\xi,\fleaf\fcyl\zeta\la(\xi));\\
  \bleaf{\bcyl}{\eta}{\la}&:\basec^*\to\torus\st\qd\bcyl(\eta,\xi;\la)=(\xi,\bleaf\bcyl\eta\la(\xi));
\end{align*}
\[
\dleaf{\Bcyl}{\eta}{\zeta}{\la} = \bleaf{\bcyl}{\eta}{\la} - \fleaf{\fcyl}{\zeta}{\la} \mod 1.\\
\]
Moreover we define the associated \emph{central leaf functions} as follows; first introduce:
\begin{align}\label{e_centralPoints}
  \czeta{\Bcyl}=\begin{cases}
    \xi_+&\textrm{if }\fbasei\fcyl = \pstr_+\\
    \xi_-&\textrm{if }\fbasei\fcyl = \pstr_-
  \end{cases}
  &&
  \ceta{\Bcyl}=\begin{cases}
    \xi_+&\textrm{if }\bbasei\bcyl = \pstr_+\\
    \xi_-&\textrm{if }\bbasei\bcyl = \pstr_-
  \end{cases}
\end{align}
then :
\begin{align*}
  \bcleaf{\bcyl}{\la}{}&=\bleaf{\bcyl}{\czeta{\Bcyl}}{\la} & \fcleaf{\fcyl}{\la}{} &=\fleaf{\fcyl}{\ceta{\Bcyl}}{\la} &\dcleaf{\Bcyl}{\la}&=\dleaf{\Bcyl}{\ceta{\Bcyl}}{\czeta{\Bcyl}}{\la}
\end{align*}
We also define the functions:
\begin{align}
  \ddleaf\Bcyl\eta\zeta\la(\xi)&=\slope{r+1}(\qd{\bcyl}(\eta,\xi;\la))-\slope{-r-1}(\pd{\fcyl}(\xi,\zeta;\la))\\
  \cddleaf\Bcyl\la(\xi)&=\ddleaf\Bcyl{\ceta{\Bcyl}}{\czeta{\Bcyl}}{\la}(\xi);
\end{align}
observe that, by definition:
\begin{subequations}\label{e_baseEstimatesOnLeaves}
  \begin{align}
    \partial_\xi\dleaf{\Bcyl}{\eta}{\zeta}{\la}(\xi)&=\la\ddleaf\Bcyl\eta\zeta\la(\xi)\\
    \partial_\eta\dleaf\Bcyl\eta\zeta\la(\xi)&=\stepex{0r+1}(\pd\fcyl(\xi,\eta;\la))^{-1}\\
    \partial_\zeta\dleaf\Bcyl\eta\zeta\la(\xi)&=\stepex{0r+1}(\pd\bcyl(\xi,\zeta;\la))^{-1}
    \intertext{and combining \eqref{eq_differentialEquationForY0} and \eqref{eq_differentialEquationForYN} we have}
    \partial_\la\dleaf{\Bcyl}{\eta}{\zeta}{\la}(\xi) &= \dot\phi(\xi)+\bigo{\la^{-1/2}}\label{e_parLaLeaf}
  \end{align}
\end{subequations}
\begin{prp}\label{p_observationsOnLeaves}
  For any $\xi\in\basec^*$, $\la\in\J_\Bcyl^\admissible$, $\zeta\in\fbasei\fcyl(\la)$ and $\eta\in\bbasei\bcyl(\la)$  the following properties are satisfied:
  \begin{subequations}
    \begin{align}
      \partial_\la\dleaf{\Bcyl}{\eta}{\zeta}{\la}(\xi) &= 1+\bigo{\la^{-1/2}}\label{e_estimateLaLeaf}\\
      \partial_\xi\ddleaf\Bcyl\eta\zeta{\la}(\xi) &= -4\pi^2\onepbigo{\la^{-1/2}}\label{e_estimateXiLeaf}
    \end{align}
  \end{subequations}
  Moreover $\dleaf{\Bcyl}{\eta}{\zeta}\la$ has a unique critical point in $\basec^*$.
\end{prp}
\begin{proof}
  Since $\xi\in\basec^*$, we can apply \eqref{e_estimateOnDotPhi} to \eqref{e_parLaLeaf} which yields \eqref{e_estimateLaLeaf}. On the other hand, by proposition~\ref{l_partialHxi}, estimates \eqref{eq_estimatedH}, \eqref{e_oneDefinitions} and \eqref{e_estimateOnDddotPhi}, we immediately obtain \eqref{e_estimateXiLeaf}.
  Finally, by estimates \eqref{e_obviousEstimatesHn} and by definition of $\basec^*$ we can check that $\ddleaf{\Bcyl}\eta\zeta\la$ has opposite signs on the two boundary points of $\basec^*$, therefore $\ddleaf{\Bcyl}\eta\zeta\la$ has at least one zero, which in fact needs to be unique by \eqref{e_estimateXiLeaf}.
\end{proof}
The following proposition shows the usefulness of leaf functions; they provide a simple method to check if the forward and backward cylinder domains associated to a given bicylinder intersect above a given interval $\basec'\subset\basec^*$.
\begin{prp}\label{p_cylinderIntersection}
  Let $\basec'\subset\basec^*$; then:
  \begin{align*}
    \dcleaf\Bcyl\la\basec'\cap\left[-\frac{1}{4}\ncLL\Bcyl^{-1},\frac{1}{4}\ncLL\Bcyl^{-1}\right]        =\emptyset &\Rightarrow \cdomain_\fcyl(\la)|_{\basec'}\cap\cdomain_\bcyl(\la)|_{\basec'} = \emptyset \\
    \dcleaf\Bcyl\la\basec'\cap\left[-\frac{1}{4}\scLL\Bcyl^{-1},\frac{1}{4}\scLL\Bcyl^{-1}\right]\not=\emptyset &\Rightarrow \cdomain_\fcyl(\la)|_{\basec'}\cap\cdomain_\bcyl(\la)|_{\basec'} \not = \emptyset
  \end{align*}
\end{prp}
\begin{proof}
  By definition we have that $\cdomain_\fcyl(\la)|_{\basec'}\cap\cdomain_\bcyl(\la)|_{\basec'} \not= \emptyset$ if and only if there exist $\xi\in\basec'$, $\eta\in\bbasei{\bcyl}(\la)$ and $\zeta\in\fbasei{\fcyl}(\la)$. such that $\dleaf\Bcyl\eta\zeta\la(\xi)=0$.
  By equations \eqref{e_baseEstimatesOnLeaves} we have:
  \begin{align*}
    \scLLi{\fcyl}&<|\partial_\eta\dleaf\Bcyl\eta\zeta\la(\xi)|<\ncLLi{\fcyl};\\
    \scLLi{\bcyl}&<|\partial_\zeta\dleaf\Bcyl\eta\zeta\la(\xi)|<\ncLLi{\bcyl}.
  \end{align*}
  Hence, since by construction $|\eta-\ceta{\Bcyl}|<1/4$ and $|\zeta-\czeta{\Bcyl}|<1/4$ we can easily conclude.
\end{proof}
\begin{prp}\label{p_estimateOnPhi}
  Let $\la_1,\la_2\in\J_{\Bcyl}^\admissible$, $\eta_1\in\bbasei{\bcyl}(\la_1),\,\eta_2\in\bbasei{\bcyl}(\la_2)$, $\zeta_1\in\fbasei{\fcyl}(\la_1),\,\zeta_2\in\fbasei{\fcyl}(\la_2)$ and $\xi\in\basec^*$; then we have:
  \begin{align*}
    \la_2|\varphi_\Bcyl(\xi,\eta_2,\zeta_2;\la_2)| &<\la_1|\varphi_\Bcyl(\xi,\eta_1,\zeta_1;\la_1)|\onepbigo{\laa^{-1}}+\\&\quad+\frac{1}{2}(\ncLL{\fcyl}^{-1}|\zeta_2-\zeta_1| + \ncLL{\bcyl}^{-1}|\eta_2-\eta_1|)+\\&\quad+|\la_2-\la_1|\bigo{\laa^{-1/2}}
  \end{align*}
\end{prp}
\begin{proof}
  First of all notice that:
  \[
  \la_1|\varphi_\Bcyl(\xi,\eta_1,\zeta_1;\la_1)| = \la_2|\varphi_\Bcyl(\xi,\eta_1,\zeta_1;\la_1)| + \frac{\la_1-\la_2}{\la_1}\la_1|\varphi_\Bcyl(\xi,\eta_1,\zeta_1;\la_1)|;
  \]
  thus we only need to estimate the difference $|\varphi_\Bcyl(\xi,\eta_1,\zeta_1;\la_1) - \varphi_\Bcyl(\xi,\eta_2,\zeta_2;\la_2)|$; if $\la_1 \geq \la_2$ the following expression is well defined:
  \begin{align*}
    |\varphi_\Bcyl(\xi,\eta_1,\zeta_1;\la_1) - \varphi_\Bcyl(\xi,\eta_2,\zeta_2;\la_2)| &<
    |\varphi_\Bcyl(\xi,\eta_1,\zeta_1;\la_1) - \varphi_\Bcyl(\xi,\eta_2,\zeta_2;\la_1)| +\\&+
    |\varphi_\Bcyl(\xi,\eta_2,\zeta_2;\la_1) - \varphi_\Bcyl(\xi,\eta_2,\zeta_2;\la_2)|;
  \end{align*}
  conversely, if $\la_2 > \la_1$, the following is well defined:
  \begin{align*}
    |\varphi_\Bcyl(\xi,\eta_1,\zeta_1;\la_1) - \varphi_\Bcyl(\xi,\eta_2,\zeta_2;\la_2)| &<
    |\varphi_\Bcyl(\xi,\eta_1,\zeta_1;\la_1) - \varphi_\Bcyl(\xi,\eta_1,\zeta_1;\la_2)| +\\&+
    |\varphi_\Bcyl(\xi,\eta_1,\zeta_1;\la_2) - \varphi_\Bcyl(\xi,\eta_2,\zeta_2;\la_2)|.
  \end{align*}
  To fix ideas we assume to be in the first case, the other case being virtually identical. Using propositions~\ref{l_estimateReferenceSequence} and~\ref{l_coneInvariance} we obtain the bound
  \[
  |\varphi_\Bcyl(\xi,\eta_1,\zeta_1;\la_1) - \varphi_\Bcyl(\xi,\eta_2,\zeta_2;\la_1)|<\frac{1}{4}(\ncLL{\fcyl}^{-1}|\zeta_2-\zeta_1| + \ncLL{\bcyl}^{-1}|\eta_2-\eta_1|)\cdot\,\la_1^{-1};
  \]
  on the other hand, by proposition~\ref{l_partialHLa} we have that
  \[
  |\varphi_\Bcyl(\xi,\eta_2,\zeta_2;\la_1) - \varphi_\Bcyl(\xi,\eta_2,\zeta_2;\la_2)|<|\la_2-\la_1\bigo{\la^{-3/2}};
  \]
  which concludes the proof.
\end{proof}
In the following proposition we give the precise statement and the proof of the argument we described in words at the beginning of this stage.
\begin{prp}\label{p_decreasingBcylEstimate}
  Let $\Bcyl=\fcyl\cap\bcyl$ be a $\lab$-admissible bicylinder of rank $r$; there exist sets:
  \begin{align*}
    \basec_\Bcyl&\subset\basec^* &\J_\Bcyl&\subset\J_\Bcyl^\admissible,
  \end{align*}
  that we call respectively \emph{adapted base} and \emph{adapted parameter set}, satisfying the following  properties:
  \begin{enumerate}[(a)]
  \item $\basec_\Bcyl^\elliptic\subset\basei{\Bcyl}$ and $\J_\Bcyl^\elliptic\subset\J_\Bcyl$;
  \item $\diam\basei{\Bcyl}<\frac{3}{8\pi^2}\laa^{-1}\ncLL{\Bcyl}^{-1}\onepbigo{\laa^{-1}}$ and $\diam\J_{\Bcyl}<\frac{1}{2}\ncLL{\Bcyl}^{-1}\onepbigo{\laa^{-1}}$;
  \item if $\Bcyl'\subset\Bcyl$ then $\basei{\Bcyl'}\subset\basei{\Bcyl}$ and $\J_{\Bcyl'}\subset\J_\Bcyl$.
  \end{enumerate}
\end{prp}
\begin{proof}
  For each $\Bcyl$ we construct two decreasing sequences of sets:
  \begin{align*}
    \basec^*=\basec_{\Bcyl}^0&\supset    \basec_{\Bcyl}^1\supset\cdots    \basec_{\Bcyl}^j\supset\cdots\\
    \J_{\Bcyl}^\admissible=\J_{\Bcyl}^0&\supset    \J_{\Bcyl}^1\supset\cdots    \J_{\Bcyl}^j\supset\cdots
  \end{align*}
  satisfying the following inductive versions of the above properties:
  \begin{enumerate}[(a')]
  \item $\basec_\Bcyl^\elliptic\subset\basei{\Bcyl}^j$ and $\J_\Bcyl^\elliptic\subset\J_\Bcyl^j$ for $j\geq 0$;
  \item the following bounds hold for $j\geq 1$:
    \begin{subequations}
      \begin{align}
        \diam\basec_{\Bcyl}^j&<\frac{1}{4\pi^2}\laa^{-1}(\ncLL{\Bcyl}^{-1}+\diam\J^{j-1}_{\Bcyl})\onepbigo{\laa^{-1/2}}\label{e_innerInductionBaseEstimate}\\
        \diam\J_{\Bcyl}^j&=
        \left(\frac{1}{2}\ncLL{\Bcyl}^{-1} + 2\laa^{-1}(\diam\J_{\Bcyl}^{j-1})^2\right)\onepbigo{\laa^{-1/2}}\label{e_innerInductionDiameterEstimate}
      \end{align}
    \end{subequations}
  \item If $\Bcyl'\subset\Bcyl$, then for any $j\geq 0$ we have $\basec_{\Bcyl'}^{j+1}\subset\basec_{\Bcyl}^{j}$ and $\J_{\Bcyl'}^{j+1}\subset\J_{\Bcyl}^{j}$.
  \end{enumerate}
  Suppose now we have constructed the sequences as prescribed; we then define:
  \begin{align*}
    \basec_{\Bcyl} &=\bigcap_j\basec_{\Bcyl}^j &  \J_{\Bcyl} &=\bigcap_j\J_{\Bcyl}^j
  \end{align*}
  In fact item (a) immediately follows by (a') and item (c) by (c'); let now $d_j = \diam\J^j_{\Bcyl}<1$;
  then by item (b') we have:
  \begin{align}\label{e_absoluteDiameterEstimate}
    d_{k} &<\frac{1}{2}\ncLL{\Bcyl}^{-1}+2\laa^{-1}d_{k-1} < \frac{1}{2}\ncLL{\Bcyl}^{-1}\cdot \sum_{j=0}^{k-1} (2\laa^{-1})^j+(2\laa^{-1})^k\notag\\
    &<\frac{1}{2}\ncLL{\Bcyl}^{-1}\onepbigo{\laa^{-1}}+ (2\laa^{-1})^k
  \end{align}
  which implies (b) for $k$ large enough. We prove (a'), (b') and (c') by induction on $j$.

  Assume we defined $\basei{\Bcyl}^j$ and $\J_\Bcyl^j$ satisfying (a') and (b'); the case $j=0$ of (a') is guaranteed by stage 1.
  If $\J_{\Bcyl}^j$ is empty, we let $\basec_{\Bcyl}^{j+1}=\emptyset$ and $\J_\Bcyl^{j+1}=\emptyset$; notice that if this is the case, then necessarily $\basei{\Bcyl}^\elliptic=\emptyset$ and $\J_{\Bcyl}^\elliptic=\emptyset$, so both statements (a') and (b') are trivially satisfied. Assume otherwise that $\J_{\Bcyl}^j\not=\emptyset$, then fix any $\la'\in\J_{\Bcyl}^j$ and define the following sets:
  \begin{align*}
    \basec_{\Bcyl}^{j+1}&=\{\xi\in\basec_{\Bcyl}^j\st\la'|\bar\varphi_\Bcyl(\xi;\la')|<\ncLL{\Bcyl}^{-1}+\diam\J_{\Bcyl}^j\};\\
    \J_{\Bcyl}^{j+1}&=\{\la\in\J_{\Bcyl}^{j}\st\cdomain_{\fcyl}(\la)|_{\basec_{\Bcyl}^{j+1}}\cap\cdomain_{\bcyl}(\la)|_{\basec_{\Bcyl}^{j+1}}\not = \emptyset\}
  \end{align*}
  By \eqref{e_estimateXiLeaf} we immediately obtain that $\basec_{\Bcyl}^{j+1}$ is an interval and that \eqref{e_innerInductionBaseEstimate} holds.
  \begin{prp}\label{p_pcritComparison}
    We have $\basec_{\Bcyl}^\elliptic\subset\basec_{\Bcyl}^{j+1}$ and $\J_{\Bcyl}^\elliptic\subset\J_{\Bcyl}^{j+1}$. Moreover if $\xi\in\basec_{\Bcyl}^{j+1}$; then for any $\la\in\J_{\Bcyl}^j$, $\eta\in\bbasei{\bcyl}(\la)$ and $\zeta\in\fbasei{\fcyl}(\la)$:
    \begin{equation}\label{e_estimateGenericXiEta}
      \la|\varphi_\Bcyl(\xi,\eta,\zeta;\la)|<2\ncLL{\Bcyl}^{-1}+\onepbigo{\la^{-1/2}}\diam\J_{\Bcyl}^j
    \end{equation}
  \end{prp}

  \begin{proof}
    Apply proposition~\ref{p_estimateOnPhi} to $\la_1=\la'$, $\eta_1=\ceta{\Bcyl}$, $\zeta_1=\czeta{\Bcyl}$ and $\la_2=\la$, $\eta_2=\eta$, $\zeta_2=\zeta$; then for large enough $\laa$ we obtain \eqref{e_estimateGenericXiEta}.
    Assume now that $\xi\in\basec_{\Bcyl}^\elliptic$; then there exists $N>r+1$, $\itin\in\cwords{N-1}\cap\Bcyl$, $\la\in\J_{\Bcyl}^\elliptic\subset\J_{\Bcyl}^j$, such that $p=\pd{\itin}(\xi,\xi;\la)=\qd{\itin}(\xi,\xi;\la)$ and
    \[
    \la|\slope{N}(\qd{\itin}(\xi,\xi;\la))-\slope{-N}(\pd{\itin}(\xi,\xi;\la))|<2\ndLL\itin^{-1}=\bigo{\la^{-1/2}}\ncLL{\Bcyl}^{-1};
    \]
    let $\zeta=\pix_{r+1} p$ and $\eta=\pix_{-r-1} p$; then, by definition:
    \[
    p=\pd{\fcyl}(\xi,\zeta;\la)=\qd{\bcyl}(\eta,\xi;\la).
    \]
    Using proposition~\ref{l_estimateReferenceSequence} we can then conclude that $\varphi_\Bcyl(\xi,\eta,\zeta;\la) = o(\ncLL{\Bcyl})$; using once more proposition~\ref{p_estimateOnPhi} we obtain that $\xi\in\basec_{\Bcyl}^{j+1}$, but since $\xi$ was arbitrary, this implies that $\basec_{\Bcyl}^\elliptic\subset\basec_{\Bcyl}^{j+1}$, which in turn yields $\J_{\Bcyl}^\elliptic\subset\J_{\Bcyl}^{j+1}$.
  \end{proof}
  \begin{prp}
    The following estimate holds:
    \[
    \diam \J_{\Bcyl}^{j+1}<\left(\frac{1}{2}\ncLL{\Bcyl}^{-1} + 2\laa^{-1}(\diam\J_{\Bcyl}^j)^2\right)\onepbigo{\la^{-1/2}}
    \]
  \end{prp}
  \begin{proof}
    If $\J_{\Bcyl}^{j+1}=\emptyset$ the estimate is trivially satisfied; otherwise
    take any $\la^*\in\J_{\Bcyl}^{j+1}$; then by proposition~\ref{p_cylinderIntersection} there exists $\xi^*\in\basec_{\Bcyl}^{j+1}$ such that:
    \[
    \dcleaf{\Bcyl}{\la^*}(\xi^*)\in\left[-\frac{1}{4}\ncLL\Bcyl^{-1},\frac{1}{4}\ncLL\Bcyl^{-1}\right].
    \]
    introduce the quantity $d = \sup_{\la\in\J_\Bcyl^{j+1}}\sup_{\xi\in\basec_\Bcyl^{j+1}}
    \dist(\dcleaf{\Bcyl}{\la}(\xi),\dcleaf{\Bcyl}{\la}(\xi^*))$;
    by definition we have:
    \[
    d
    \leq\diam\basec_\Bcyl^{j+1}
    \sup_{\la\in\J_\Bcyl^{j+1}}
    \sup_{\xi\in\basec_\Bcyl^{j+1}}
    |\la\bar\varphi_\Bcyl(\xi;\la)|;
    \]
    which, using \eqref{e_innerInductionBaseEstimate} and proposition~\ref{p_pcritComparison}, yields:
    \[
    d
    < \delta_{j+1}=
    \laa^{-1}(\ncLL{\Bcyl}^{-1}+\diam\J_k^j)
    \cdot(2\ncLL{\Bcyl}^{-1}+\diam\J_k^j)\onepbigo{\la^{-1/2}},
    \]

    which implies that, for any $\la\in\J_\Bcyl^{j+1}$:
    \[
    \dcleaf{\Bcyl}{\la}\basei{\Bcyl}^{j+1}\subset\ball{\delta_{j+1}}{\dcleaf{\Bcyl}{\la}(\xi^*)}.
    \]
    Hence, if ${\dcleaf{\Bcyl}{\la}(\xi^*)}\not\in\left[-\frac{1}{4}\ncLL\Bcyl^{-1}-\delta_{j+1},\frac{1}{4}\ncLL\Bcyl^{-1}+\delta_{j+1}\right]$ then by proposition~\ref{p_cylinderIntersection} we must have $\la\not\in\J_{\Bcyl}^{j+1}$, whence we can conclude using \eqref{e_estimateLaLeaf}.
  \end{proof}
  It remains to prove item (c'); the case $j=0$ is trivial by definition; assume that we proved that $\basei{\Bcyl'}^j\subset\basei{\Bcyl}^{j-1}$ and $\J_{\Bcyl'}^j\subset\J_{\Bcyl}^{j-1}$, then we  prove that that $\basei{\Bcyl'}^{j+1}\subset\basei{\Bcyl}^{j}$, which in turn implies $\J_{\Bcyl'}^{j+1}\subset\J_{\Bcyl}^{j}$ by construction.

  Let $\xi\in\basei{\Bcyl'}^{j+1}$ and $\la\in\J_{\Bcyl'}^{j}$; then by proposition~\ref{p_pcritComparison} we have for any $\eta'\in\bbasei{{\bcyl}'}$ and $\zeta'\in\fbasei{{\fcyl}'}$:
  \[
  \la|\varphi_{\Bcyl'}(\xi,\eta',\zeta';\la)|<2\ncLL{\Bcyl'}^{-1}+\onepbigo{\la^{-1/2}}\diam\J_{\Bcyl'}^j.
  \]
  but by proposition~\ref{l_estimateReferenceSequence} this implies that for some $\eta\in\bbasei{{\bcyl}}$ and $\zeta\in\fbasei{{\fcyl}}$
  \[
  \la|\varphi_{\Bcyl}(\xi,\eta,\zeta;\la)|<o(\ncLL{\Bcyl}^{-1})+\onepbigo{\la^{-1/2}}\diam\J_{\Bcyl'}^j.
  \]
  However, \eqref{e_absoluteDiameterEstimate} implies that $\diam\J_{\Bcyl'}^j=\diam\J_{\Bcyl}^{j-1}\bigo{\la^{-1/2}}$; we then apply proposition~\ref{p_estimateOnPhi} and, since by induction hypothesis we have that $\la\in\J_{\Bcyl'}^j\subset\J_{\Bcyl}^{j-1}$, we obtain that  $\xi\in\basei{\Bcyl}^j$, which concludes the proof.
\end{proof}


\paragraph{Stage 3} For any $r>0$ we show that the collection $\{\J_{\Bcyl}\}$, where $\Bcyl$ ranges on all $\J$-admissible bicylinders of rank $r$, covers $\J$ except for a $\bigo{\la^{-1/2}}$-neighborhood of its boundary. Throughout this stage we fix a small $\eps>0$ and assume $\la$ to be large enough to ensure every term that we previously bounded with $\bigo{\la^{-1/2}}$ to be smaller than $\eps$. In this constructive part of the proof it is very convenient to restrict our considerations to a special class of bicylinders, as long as the covering property still holds. The first property that we need to ensure is good control of their adapted parameter set and it is defined in the following
\begin{prp}
  Fix $\Bcyl$; for $\la\in\J_\Bcyl^\admissible$ let $\xi_\critical(\la)$ be defined as the unique critical point of the leaf function $\dcleaf{\Bcyl}{\la}$ i.e. $\cddleaf{\Bcyl}{\la}(\xi_\critical(\la))=0$; then if there exists $\la\in\J_\Bcyl^\admissible$ such that
  \begin{equation}\label{d_tangencyParameter}
    \dcleaf{\Bcyl}{\la}(\xi_\critical(\la)) = 0\mod 1
  \end{equation}
  we say that $\la$ is a \emph{tangency parameter for $\Bcyl$}. For each $\Bcyl$, there exists at most one tangency parameter in $\J$ that will be denoted by $\la_\Bcyl$ and we have
  \begin{equation}\label{e_upperInclusionJ}
    \J_\Bcyl\subset\ball{\frac{1}{4}\ncLL{\Bcyl}^{-1}(1+\eps)}{\la_\Bcyl}.
  \end{equation}
  If we additionally assume that 
  \begin{equation}\label{e_intactCondition}
    \ball{\frac{1}{4}\ncLL{\Bcyl}^{-1}(1+\eps)}{\la_\Bcyl}\subset\J_\Bcyl^\admissible
  \end{equation}
  then we also obtain the following lower bound
  \begin{equation}\label{e_lowerInclusionJ}
    \ball{\frac{1}{4}\scLL{\Bcyl}^{-1}(1-\eps)}{\la_\Bcyl}\subset\J_\Bcyl.
  \end{equation}
  A bicylinder for which there exists a tangency parameter $\la_\Bcyl$ satisfying \eqref{e_intactCondition} is said to be \emph{intact in $\J$}.
\end{prp}
\begin{proof}
  To prove uniqueness of the tangency parameter in $\J$, notice that,  by definition:
  \[
  \de{}{\la}\dcleaf{\Bcyl}{\la}(\xi_\critical(\la))=\partial_\la\dcleaf{\Bcyl}{\la}(\xi_\critical(\la)),
  \]
  and we can conclude by \eqref{e_estimateLaLeaf}: in fact since $|\J|<7/8$ we know $\dcleaf{\Bcyl}{\la}(\xi_\critical(\la))$ cannot wrap around the torus, hence it can intersect $0$ only once in $\J$.
  We now proceed to prove the inclusion statements; following the proof of proposition~\ref{p_decreasingBcylEstimate} and using
  \[
  \la|\cddleaf{\Bcyl}{\la}(\xi_\critical(\la))|<|\la-\la_\Bcyl|\bigo{\la^{-1/2}},
  \]
  which follows by applying proposition~\ref{p_estimateOnPhi}, it is immediate to check that $\xi_\critical\in\basei{\Bcyl}^j$ for all $j\geq 0$, which implies that $\xi_\critical\in\basei{\Bcyl}$, whence $\la_\Bcyl\in\J_\Bcyl$.
  Consequently, if $\la\in\ball{\frac{1}{4}\scLL{\Bcyl}^{-1}(1-\eps)}{\la_\Bcyl}$ and if \eqref{e_intactCondition} holds, by \eqref{e_estimateLaLeaf} we obtain:
  \[
  \dist(\dcleaf{\Bcyl}{\la}(\xi_\critical),\dcleaf{\Bcyl}{\la_\Bcyl}(\xi_\critical)) \leq |\la-\la_\Bcyl|\onepbigo{\la^{-1/2}},
  \]
  which implies \eqref{e_lowerInclusionJ} by proposition~\ref{p_cylinderIntersection}.

  Conversely, assume that $\la\in\J_\Bcyl$; then by proposition~\ref{p_cylinderIntersection} there exists $\xi\in\basei{\Bcyl}$ such that:
  \[
  \dcleaf{\Bcyl}{\la}(\xi)\in\left[-\frac{1}{4}\ncLL\Bcyl^{-1},\frac{1}{4}\ncLL\Bcyl^{-1}\right]
  \]
  Using once more \eqref{e_estimateLaLeaf} we obtain
  \[
  \dist(\dcleaf{\Bcyl}{\la_\Bcyl}(\xi),\dcleaf{\Bcyl}{\la}(\xi))=  |\la-\la_\Bcyl|\onepbigo{\la^{-1/2}};
  \]
  and moreover, by the diameter estimates and by definition of $\basei{\Bcyl}$ we have:
  \[
  \dist(\dcleaf{\Bcyl}{\la_\Bcyl}(\xi),\dcleaf{\Bcyl}{\la_\Bcyl}(\xi_\critical)) \leq 3\laa^{-1}\ncLL{\Bcyl}^{-2}.
  \]
  Hence:
  \[
  |\la-\la_\Bcyl|\onepbigo{\la^{-1/2}} = \dist(\dcleaf{\Bcyl}{\la}(\xi),\dcleaf{\Bcyl}{\la_\Bcyl}(\xi_\critical))\leq\frac{1}{4}\ncLL\Bcyl^{-1},
  \]
  from which we conclude.
\end{proof}
We are now going to inductively construct a set of intact with good density properties, which we call \emph{balanced bicylinders}; at the $n$-th inductive step we construct the $n$-th \emph{generation} of balanced bicylinders, that is a collection of intact bulk bicylinders of rank $n$. Since the actual definition is rather cumbersome, we first try to expose the general idea in words before providing all the details. First of all notice that if $\Bcyl=\fcyl\cap\bcyl$ is a bicylinder of rank $r$, then in general $\ncLL{\fcyl}$ and $\ncLL{\bcyl}$ are quite different (in fact their ratio can oscillate between $\la^{-r/2}$ and $\la^{r/2}$); on top of this, even if we had $\ncLL{\fcyl}\sim\ncLL{\bcyl}$, the forward and backward expansion rates at step $k<r$ could still be very different for points belonging to $\fcyl$ and $\bcyl$ respectively. This leads to complications which we want to avoid: in order to choose forward and backward cylinders that have similar expansion rates at all steps we appropriately construct left-closed and right-closed words and take the associated forward and backward cylinder. In the construction process it is also natural to index the resulting bicylinders in such a way that if two bicylinders are indexed consecutively, then their adapted parameter set will be close to each other: in fact what we will show is that they indeed overlap.

\begin{mydef}[First generation]\label{p_childBicylinders}
  Fix $s_1$ and $\s_{-1}$ depending on $\J$ in a way to be described later; for $j\in\integers$ define the symbols $\altLetter^\forw_{(j,\pm)}\in\alphabetHead(\lab)$ and $\altLetter^\backw_{(j,\pm)}\in\alphabetTail(\lab)$ as follows:
  \begin{align*}
    \altLetter^\forw_{(j,\pm)}  &=\hsignDeg{-}{\s_1}{j}{\pm} &
    \altLetter^\backw_{(j,\pm)} &=\tsignDeg{\pm}{\s_{-1}}{j}{-};
  \end{align*}
  further, define $\zeta_{(j,\pm)}=\xi_{{\altLetter^\forw_{(j,\pm)}}}$ and $\eta_{(j,\pm)}=\xi_{{\altLetter^\backw_{(j,\pm)}}}$. We then define cylinders of rank $1$ associated to the following words:
  \begin{align*}
    \tfcyl_{(j,\pm)}&=-\altLetter^\forw_{(j,\pm)}&\tbcyl_{(j,\pm)}&=\altLetter^\backw_{(j,\pm)}-,
  \end{align*}
  and the bicylinders:
  \begin{align*}
    \tBcyl_{(4j)}&={\tfcyl_{(j,+)}}\cap{\tbcyl_{(j,-)}}\\
    \tBcyl_{(4j+1)}&={\tfcyl_{(j,+)}}\cap{\tbcyl_{(j,+)}}\\
    \tBcyl_{(4j+2)}&={\tfcyl_{(j,-)}}\cap{\tbcyl_{(j,+)}}\\
    \tBcyl_{(4j+3)}&={\tfcyl_{(j,-)}}\cap{\tbcyl_{(j+1,-)}};
  \end{align*}
  and likewise:
  \begin{align*}
    \zeta_{(4j)}&=\zeta_{(j,+)}&  \eta_{(4j)}&=\eta_{(j,-)}\\
    \zeta_{(4j+1)}&=\zeta_{(j,+)}&  \eta_{(4j+1)}&=\eta_{(j,+)}\\
    \zeta_{(4j+2)}&=\zeta_{(j,-)}&  \eta_{(4j+2)}&=\eta_{(j,+)}\\
    \zeta_{(4j+3)}&=\zeta_{(j,-)}&  \eta_{(4j+3)}&=\eta_{(j+1,-)}
  \end{align*}
  The \emph{first generation of balanced bicylinders} is the collection of all $\tBcyl_{(l)}$ that are intact bulk bicylinders.

  Assume we defined the $r$-th generation of balanced bicylinders, we proceed to define the $r+1$-st generation of balanced bicylinders.
  Fix $\Bcyl=\fcyl\cap\bcyl$ an arbitrary $r$-th generation balanced bicylinder; to fix ideas let $\fcyl=-\letter_1\cdots\letter_r$ and $\bcyl=\letter_{-r}\cdots\letter_{-1}-$; let $S^\forw = \s_1\cdots\s_{r-1}$ and $\S^\backw=\s_{-1}\cdots\s_{-r+1}$,  where $\s_j$ is the sign $\s$ associated to the symbol $\letter_j$;  define the symbols $\altLetter^\forw_{(\Bcyl:j,\pm)},\,\altLetter^\backw_{(\Bcyl:j,\pm)}\in\alphabetReg(\lab)$ as follows:
  \begin{align*}
    \altLetter^\forw_{(\Bcyl:j,\pm)} &=\signDeg{\s_r}{\si_r}{S^\forw j}{\pm} &
    \altLetter^\backw_{(\Bcyl:j,\pm)}&=\signDeg{\pm}{\sii_{-r}}{S^\backw j}{\s_{-r}};
  \end{align*}
  further, define $\zeta_{(\Bcyl:j,\pm)}=\xi_{{\altLetter^\forw_{(\Bcyl:j,\pm)}}}$ and $\eta_{(\Bcyl:j,\pm)}=\xi_{{\altLetter^\backw_{(\Bcyl:j,\pm)}}}$. Define the following cylinders of rank $r+1$:
  \begin{align*}
    \tfcyl_{(\Bcyl:j,\pm)}&=-\letter_1\cdots\letter_r\altLetter^\forw_{(\Bcyl:j,\pm)} &
    \tbcyl_{(\Bcyl:j,\pm)}&=\altLetter^\backw_{(\Bcyl:j,\pm)}\letter_{-r}\cdots\letter_{-1}-,
  \end{align*}
  and the corresponding bicylinders:
  \begin{align*}
    \tBcyl_{(\Bcyl:4j)}&={\tfcyl_{(\Bcyl:j,S^\forw)}}\cap{\tbcyl_{(\Bcyl:j,-S^\backw)}}\\
    \tBcyl_{(\Bcyl:4j+1)}&={\tfcyl_{(\Bcyl:j,S^\forw)}}\cap{\tbcyl_{(\Bcyl:j,S^\backw)}}\\
    \tBcyl_{(\Bcyl:4j+2)}&={\tfcyl_{(\Bcyl:j,-S^\forw)}}\cap{\tbcyl_{(\Bcyl:j,S^\backw)}}\\
    \tBcyl_{(\Bcyl:4j+3)}&={\tfcyl_{(\Bcyl:j,-S^\forw)}}\cap{\tbcyl_{(\Bcyl:j+1,-S^\backw)}};
  \end{align*}
  likewise:
  \begin{align*}
    \zeta_{(\Bcyl:4j)}&=\zeta_{(\Bcyl:j,S^\forw)}&  \eta_{(\Bcyl:4j)}&=\eta_{(\Bcyl:j,-S^\backw)}\\
    \zeta_{(\Bcyl:4j+1)}&=\zeta_{(\Bcyl:j,S^\forw)}&  \eta_{(\Bcyl:4j+1)}&=\eta_{(\Bcyl:j,S^\backw)}\\
    \zeta_{(\Bcyl:4j+2)}&=\zeta_{(\Bcyl:j,-S^\forw)}&  \eta_{(\Bcyl:4j+2)}&=\eta_{(\Bcyl:j,S^\backw)}\\
    \zeta_{(\Bcyl:4j+3)}&=\zeta_{(\Bcyl:j,-S^\forw)}&  \eta_{(\Bcyl:4j+3)}&=\eta_{(\Bcyl:j+1,-S^\backw)}
  \end{align*}
  Every $\tBcyl_{(\Bcyl:l)}$ that is an intact bulk bicylinder is said to be a \emph{child} of $\Bcyl$ and the collection of all children of all $r$-th generation balanced bicylinders composes the \emph{$r+1$-st generation of balanced bicylinders}.
\end{mydef}
By proposition~\ref{p_symmetricLetters} and \eqref{e_bulkApproximation} it is immediate to check that:
\begin{align}\label{e_estimateBalancedBcyl}
  \ncLLi{\fcyl}&=\frac{1}{2}\ncLLi{\Bcyl}\onepbigo{\la^{-1}}&
  \ncLLi{\bcyl}&=\frac{1}{2}\ncLLi{\Bcyl}\onepbigo{\la^{-1}}.
\end{align}
We are now ready to properly state the covering lemma:
\begin{mydef}
   Let $\Bcyl$ be an intact bicylinder; we define the \emph{core parameter set of $\Bcyl$} as:
  \[
  \hat\J_\Bcyl=\ball{\frac{3}{8}\frac{\ncLL{\Bcyl}^{-1}}{2}}{\la_\Bcyl}
  \]
\end{mydef}
\begin{lem}\label{l_covering}
  The collection of the core parameter sets of the first generation balanced bicylinder covers the main core parameter set $\hat\J$.
  Moreover for any $r$-th generation balanced bicylinder $\Bcyl$; then the core parameter set $\hat\J_\Bcyl$ is covered by the collection of the core parameter sets of its children:
  \[
  \hat\J_{\Bcyl}\subset\bigcup_{l}\hat\J_{\tBcyl_{(\Bcyl:l)}}
  \]
\end{lem}
\begin{proof}
  First of all we need to control the tangency parameters of balanced bicylinders; this can be done by means of the following
  \begin{prp}\label{p_approximateTangency}
    For any bicylinder $\Bcyl$, $\la\in\J_\Bcyl^\admissible$, let
    \[
    \dcleaf{\Bcyl}{\la}(\xi_\critical(\la))=\tilde\delta\mod 1
    \]
    assume $\tilde\delta$ is such that there exists $\delta=\tilde\delta - m$ with $m\in\{0,1\}$ satisfying \[
    \sball=\ball{C\delta\la^{-1/2}}{\la-\delta}\subset\J_\Bcyl^\admissible,
    \] where $C$ is some  constant independent of $\Bcyl$ and on the initial point $\la$;
    then we can conclude that $\Bcyl$ admits a tangency parameter $\la_\Bcyl\in\sball$.
  \end{prp}
  \begin{proof}
    Let $\la_0=\la$, $\xi_0=\xi_\critical(\la)$ and $\delta_0=\delta$; for $j>0$ we define sequences $\la_j\in\J$, $\xi_j\in\basec^*$ and $\delta_j\in\reals$ as follows: let $\la_j=\la_{j-1}-\delta_{j-1}$, $\xi_{j}$ be such that $\bar\varphi_{\Bcyl}(\xi_j,\la_j)=0$ and let $\delta_j=\dcleaf{\Bcyl}{\la_j}(\xi_j)-m_j$ where $m_j$ will be chosen later.
    We then claim that $\la_j$ converges to $\la_\Bcyl$; in fact using \eqref{e_estimateLaLeaf} we obtain:
    \[
    \dist(\dcleaf{\Bcyl}{\la_{j+1}}(\xi_j),0)\leq C' |\delta_j|\la^{-1/2}
    \]
    for some constant $C'$; further, by proposition~\ref{p_estimateOnPhi} we know that \[
    |\cddleaf{\Bcyl}{\la_{j+1}}(\xi_j)-\cddleaf{\Bcyl}{\la_j}(\xi_j)|=|\delta_j|\bigo{\la^{-3/2}};\]
    since $\cddleaf{\Bcyl}{\la}$ is monotonic we can then choose $m_{j+1}$ such that:
    \[
    |\delta_{j+1}| = |\dcleaf{\Bcyl}{\la_{j+1}}(\xi_{j+1})+m_{j+1}|\leq C'' |\delta_j|\la^{-1/2}.
    \]
    If we define
    \[
    C = 2\frac{C''}{1-C''\la^{-1/2}}
    \]
    we have that $\la_{j}\in\sball$ for any $j$ and the induction procedure is well defined. By construction the limit point of the sequence $\la_j$ belongs to $\sball$, and by continuity we can conclude that it is indeed given by $\la_\Bcyl$.
  \end{proof}
  Proposition~\ref{p_approximateTangency} is extremely useful, since it allows us to check if a bicylinder admits a tangency parameter and if this is the case, it allows to obtain a estimate on its value by simply evaluating a function that is defined for all $\la\in\J_\Bcyl^{\admissible}$. We now use the proposition to establish the following result about the distribution of tangency parameters for balanced bicylinders
  \begin{prp}\label{p_firstGenerationDistribution}
    Let $\la_0=(\laa+\lab)/2$; then we can choose $\s_1$ and $\s_{-1}$ in definition~\ref{p_childBicylinders} in such a way that
    \begin{equation}\label{e_zeroBicylinder}
      |\la_{\tBcyl_{(0)}}-\la_0|<6\laa^{-1}
    \end{equation}
    moreover let $\tBcyl_{(l)}$ and $\tBcyl_{(l+1)}$ both be first generation bicylinders; then:
    \begin{equation}\label{e_coveringBicylinder}
      \la_{\tBcyl_{(l+1)}}-\la_{\tBcyl_{(l)}}=-\frac{1}{4}\ncLL{\tBcyl_{(l)}}^{-1}\onepbigo{\la^{-1/2}}.
    \end{equation}
  \end{prp}
  \begin{proof}
    By definition $\la_0 = n/2$ for some $n\in\naturals$; we let $\s_1=\s_{-1}=+$ if $n$ is odd, and $\s_1=+$, $\s_{-1}=-$ if $n$ is even.
    By definition of central leaf function it is trivial to check that
    \begin{align*}
      \fcleaf{\tfcyl_{(l)}}{\la}(\gxi\cs-) &= \zeta_{(l)}&
      \bcleaf{\tbcyl_{(l)}}{\la}(\gxi\cs-) &= -\eta_{(l)} + \la
    \end{align*}
    whence
    \begin{equation}\label{e_dcleafEtaZeta}
      \dcleaf{\tBcyl_{(l)}}{\la}(\gxi\cs-)=-\eta_{(l)}+\la - \zeta_{(l)}.
    \end{equation}
    We claim that \eqref{e_zeroBicylinder} holds; in fact by \eqref{e_absoluteRelationsXiC} we know that
    \begin{align*}
      |\zeta_{(0)}-\gxi\ds{\s_1}|=|\xiC{\cs}{-}{{\s_1}}{\ds}{{+}}{0} - \gxi\ds{\s_1}|&<\la^{-1}\\
      |\eta_{(0)}-\gxi\ds{\s_{-1}}|=|\xiC{\ds}{{-}}{{\s_{-1}}}{\cs}-{0} - \gxi\ds{\s_{-1}}|&<\la^{-1};
    \end{align*}
    on the other hand, by our choice of $\s_1$ and $\s_{-1}$ we know that
    \[
    \gxi\ds{\s_1}+\gxi\ds{\s_{-1}}-\la_0=0 \mod 1,
    \]
    hence we conclude that $\dist(\dcleaf{\tBcyl_{(l)}}{\la_0}(\gxi\cs-),0)<2\la_0^{-1}$. Since we have $\xi_\critical(\la_0)\in\basec^*$, $|\basec^*|<\laa^{-1}$ and $|\partial_\xi\dcleaf{\tBcyl_{(0)}}{\la_0}|<2$ on $\basec^*$ we also obtain that
    \[
    \dist(\dcleaf{\tBcyl_{(l)}}{\la_0}(\xi_\critical(\la_0)),0)<4\la_0^{-1}
    \]
    whence, using proposition~\ref{p_approximateTangency}, we obtain \eqref{e_zeroBicylinder}.

    In order to prove \eqref{e_coveringBicylinder} we perform a similar computation; first assume that we proved the following estimate:
    \begin{equation}\label{e_leafAtWrongCritical}
      \dcleaf{\tBcyl_{(l+1)}}{\la}(\gxi\cs-) - \dcleaf{\tBcyl_{(l)}}{\la}(\gxi\cs-) = -\frac{1}{4}\ncLLi{\Bcyl_{(l)}}\onepbigo{\la^{-1}}
    \end{equation}
    We now claim that if \eqref{e_leafAtWrongCritical} holds we obtain \eqref{e_coveringBicylinder}. In fact by a Gronwall argument we can extend \eqref{e_leafAtWrongCritical} to
    \[
    \dcleaf{\tBcyl_{(l+1)}}{\la}(\xi) - \dcleaf{\tBcyl_{(l)}}{\la}(\xi) = -\frac{1}{4}\ncLLi{\Bcyl_{(l)}}\onepbigo{\la^{-1}};
    \]
    for $\xi\in\basec^*$, from which we can finally obtain that
    \[
    \dist(\dcleaf{\tBcyl_{(l+1)}}{\la_{\tBcyl_{(l)}}}(\xi_\critical^{(l+1)}(\la_{\tBcyl_{(l)}})),0)=
    -\frac{1}{4}\ncLLi{\Bcyl_{(l)}}\onepbigo{\la^{-1}};
    \]
    where $\xi_\critical^{(l)}(\la)$ is the unique critical point of $\dcleaf{\tBcyl_{(l)}}\la$ in $\basec^*$. We can then apply proposition~\ref{p_approximateTangency} and conclude that \eqref{e_coveringBicylinder} holds.

    We now need to prove \eqref{e_leafAtWrongCritical}; by \eqref{e_dcleafEtaZeta} it suffices to show that
    \begin{equation}\label{e_differenceEtaZeta}
      \zeta_{(l+1)}+\eta_{(l+1)}-\zeta_{(l)}-\eta_{(l)} = \frac{1}{4}\ncLLi{\Bcyl_{(l)}}\onepbigo{\la^{-1}}.
    \end{equation}
    We need to prove \eqref{e_differenceEtaZeta} separately in each case $l=4k, 4k+1, 4k+2, 4k+3$; we will explicitly prove just the first possibility; all other follow from similar reasoning; we assume then $l=4k$.  By our indexing strategy for balanced bicylinders we have that $\zeta_{(l)}=\zeta_{(l+1)}$, therefore we only need to estimate:
    \[
    \eta_{(l+1)}-\eta_{(l)}=\eta_{(k,+)}-\eta_{(k,-)}=\xiC\ds+{{\s_{-1}}}\cs-{k} - \xiC\ds-{{\s_{-1}}}\cs-{k},
    \]
    which by estimates \eqref{e_deltaRelationsXiC} imply:
    \[
    \eta_{(l+1)}-\eta_{(l)}=\frac{1}{2}\ncLLi{\tbcyl_{(l)}}{}\onepbigo{\la^{-1}}=\frac{1}{4}\ncLLi{\tBcyl_{(l)}}{}\onepbigo{\la^{-1}}.
    \]
    The other three cases can be treated in the same way, exploiting  relations \eqref{e_orderRelationsXiC} and \eqref{e_deltaRelationsXiC}.
  \end{proof}
  Notice that \eqref{e_coveringBicylinder} implies that the core parameter set of each pair of consecutive bicylinders of first generation always overlap, since each bicylinder $\tBcyl_{{j}}$ has a neighboring bicylinder whose tangency parameter is at a distance $(1/4)\ncLLi{\tBcyl_{{j}}}$ from its tangency parameter and, by definition, the core parameter set of each bicylinder is a ball of radius $(3/16)\ncLLi{\tBcyl_{{j}}}$ around the respective tangency parameter.

  We can then start from $\tBcyl_{(0)}$ and cover larger and larger balls in the set $\J$ taking at each step the $j$-th and $-j$th bicylinder; in fact the procedure has to stop $l
  \bigo{\la^{-1/2}}$-close to the boundary of $\J$, but since we assume $\la$ large enough, we will nonetheless cover $\hat\J$. We now state the corresponding proposition for bicylinder of higher generation.
  \begin{prp}\label{p_childBicylinder}
    Let $\Bcyl$ be balanced  bicylinder of rank $r$; then
    \begin{equation}\label{e_childBicylinder}
      |\la_{\tBcyl_{(\Bcyl:0)}}-\la_\Bcyl|<3\laa^{-1}\ncLL{\Bcyl}^{-1}.
    \end{equation}
    moreover let $\tBcyl_{(\Bcyl:l+1)}$ and $\tBcyl_{(\Bcyl:l)}$ be balanced bicylinders of rank $r+1$ that are both children of $\Bcyl$:
    \begin{equation}\label{e_siblingBicylinder}
      \la_{\tBcyl_{(\Bcyl:l+1)}}-  \la_{\tBcyl_{(\Bcyl:l)}}=-\frac{1}{4}\ncLL{\tBcyl_{(\Bcyl:l)}}^{-1}\onepbigo{\la^{-1/2}}
    \end{equation}
  \end{prp}
  \begin{proof}
    The proof is similar to the proof of proposition~\ref{p_firstGenerationDistribution}; we first proceed to find the expression corresponding to \eqref{e_dcleafEtaZeta}. Let $\Bcyl=\fcyl\cap\bcyl$; consider $\xi_\critical(\la_\Bcyl)$ the critical point of $\dcleaf{\Bcyl}{\la_\Bcyl}$ and consider the vertical line $\vl_\critical=\pix^{-1}(\xi_\critical(\la_\Bcyl)$. Notice that, by definition, the point $(\xi_\critical,\fcleaf{\fcyl}{\la_\Bcyl}(\xi_\critical)$ is mapped by $\smaux{\la_\Bcyl}^{r+1}$ on the appropriate vertical line $\gvl\ds{\s_{r+1}}$; moreover, by estimates \eqref{e_bulkEstimates} and \eqref{e_bulkApproximation} we have the following expansion estimate along the vertical line $\vl_\critical$ in the domain $\cdomain_\fcyl$:
    \[
    \de{x_r}{y}=S^\forw\ncLL{\fcyl}\onepbigo{\la^{-1}},
    \]
    where $S^\forw=\s_1\cdots \s_{r-1}$.
    This implies that
    \[
    \fcleaf{\tfcyl_{(\Bcyl:l)}}{\la_\Bcyl}(\xi_\critical) = \fcleaf{\fcyl}{\la_\Bcyl}(\xi_\critical)+S^\forw(\zeta_{(l)} - \gxi\ds{\s_{r+1}})\ncLLi{\fcyl}\onepbigo{\la^{-1}}\mod 1.
    \]
    With a similar argument we find that:
    \[
    \bcleaf{\tbcyl_{(\Bcyl:l)}}{\la_\Bcyl}(\xi_\critical) = \bcleaf{\bcyl}{\la_\Bcyl}(\xi_\critical)-S^\backw(\eta_{(l)} - \gxi\ds{\s_{-r-1}}) \ncLLi{\bcyl}\onepbigo{\la^{-1}}\mod 1,
    \]
    where $S^\backw=\s_{-1} \cdots \s_{-r+1}$; hence:
    \begin{equation}\label{e_dcleafEtaZetar}
      \dcleaf{\tBcyl_{(\Bcyl:l)}}{\la_\Bcyl}(\xi_\critical) = (-S^\forw(\zeta_{(l)} - \gxi\ds{\s_{r+1}}) - S^\backw(\eta_{(l)} - \gxi\ds{\s_{-r-1}}))\frac{1}{2}\ncLLi{\Bcyl}\onepbigo{\la^{-1}}.
    \end{equation}
    Following an argument similar to the one given previously it is easy to show that \eqref{e_dcleafEtaZetar} implies \eqref{e_childBicylinder}. Similarly, the proof of \eqref{e_siblingBicylinder} amounts to check that the indexing strategy that we have chosen in the definition of balanced bicylinders is such that:
    \begin{align*}
      (S^\forw(\zeta_{(l+1)} - \zeta_{(l)})  &+ S^\backw(\eta_{(l+1)} - \eta_{(l)})) \frac{1}{2}\ncLLi{\Bcyl}\onepbigo{\la^{-1}}\\&=  \frac{1}{4}\ncLLi{\tBcyl_{(\Bcyl:l)}}\onepbigo{\la^{-1}}
    \end{align*}
    Again the proof should be done separately in each case $l=4k, 4k+1, 4k+2, 4k+3$; we will once more explicitly prove just the first possibility. As before we have that $\zeta_{(l)}=\zeta_{(l+1)}$, hence we only need to estimate:
    \begin{align*}
      \eta_{(\Bcyl:l+1)}-\eta_{(\Bcyl:l)} &=  \eta_{(\Bcyl:l+1,S^\backw)} - \eta_{(\Bcyl:l+1,-S^\backw)} \\&= \xiC\ds{{S^\backw}}{{\sii_{-r}}}\ds{{\s_{-r}}}{k} - \xiC\ds{{-S^\backw}}{{\sii_{-r}}}\ds{{\s_{-r}}}{k}.
    \end{align*}
    By estimate \eqref{e_deltaRelationsXiC} we then obtain:
    \[
    \eta_{(\Bcyl:l+1)}-\eta_{(\Bcyl:l)} = S^\backw\frac{1}{2}|\la\slope{1}(p_{(\altLetter_{k,S^\backw})}|^{-1}
    \]
    whence we can conclude, since $\ncLL{\tBcyl_{(\Bcyl,l)}}=\ncLL{\Bcyl}\cdot\la|\slope{1}(p_{\altLetter_{k,S^\backw}})|$.
  \end{proof}
  Using a similar reasoning as before we prove that given a balanced bicylinder of generation $n$ we can cover its core parameter set with the core parameter sets of its children, which concludes the proof of lemma~\ref{l_covering}.
\end{proof}
Lemma~\ref{l_covering} immediately implies the following
\begin{cor}\label{c_decreasingSequence}
  For $N$ sufficiently large, define $\eps_N$ as follows:
  \[
  \eps_N=\left(\frac{1}{10}\laa\right)^{-(N-2\lfloor\log N\rfloor-10)}
  \]
  Then if $\mathcal{B}\subset\hat\J$ has diameter larger than $\eps_N$, there exists a balanced bicylinder $\tBcyl$ of rank $(N-2\lfloor\log N\rfloor-10)+1$ such that $\J_{\tBcyl}\subset\mathcal{B}$; moreover we can choose $\tBcyl$ to be $\tBcyl_{(\Bcyl:0)}$.
\end{cor}
The proof of the corollary is straightforward; $\eps_N$ is defined in such a way that it larger than the parameter set of any bulk cylinder of rank $(N-2\lfloor\log N\rfloor-10)$; hence, by the covering property, we can always find a balanced cylinder of rank $(N-2\lfloor\log N\rfloor-10)+1$ whose parameter set is contained in any ball larger than $\eps_N$.
\paragraph{Stage 4} For a given closed word $\itin\in\Bcyl$ we further reduce the set $\basec_\Bcyl$ to a smaller set $\basei{\itin}$ such that we have $\basec_\itin^\elliptic\subset\basei{\itin}$. Moreover we prove that if a word belongs to an intact bicylinder, then we always have an elliptic island for some specific parameter set of diameter appropriately bounded from below.
Given $\itin$ we introduce the \emph{domain leaf functions} $\leafPrivate_{\itin}$; for $\la\in\J_\itin^\admissible$, $\zeta,\eta\in\basec^*$, define
\begin{align*}
  \fleaf{\itin}{\zeta}{\la}&:\basec^*\to\torus &\pd\itin(\xi,\zeta;\la)&=(\xi,\fleaf\itin\zeta\la(\xi));\\
  \bleaf{\itin}{\eta}{\la}&:\basec^*\to\torus &\qd\itin(\eta,\xi;\la)&=(\xi,\bleaf\itin\eta\la(\xi));
\end{align*}
\[
\dleaf{\itin}{\eta}{\zeta}{\la}=\bleaf{\itin}{\eta}{\la}-\fleaf{\itin}{\zeta}{\la}\mod 1
\]
We also define the function
\begin{equation*}
  \ddleaf\itin\eta\zeta\la(\xi)=\slope{N}(\qd{\itin}(\eta,\xi;\la))-\slope{-N}(\pd{\itin}(\xi,\zeta;\la));
\end{equation*}
As in equations \eqref{e_baseEstimatesOnLeaves} observe that:
\begin{subequations}\label{e_baseEstimatesOnDomainLeaves}
  \begin{align}
    \partial_\xi\dleaf{\itin}{\eta}{\zeta}{\la}(\xi)&=\la\ddleaf\itin\eta\zeta\la(\xi)\\
    \partial_\eta\dleaf\itin\eta\zeta\la(\xi)&=\stepex{0r+1}(\pd\itin(\eta,\xi;\la))^{-1}\\
    \partial_\zeta\dleaf\itin\eta\zeta\la(\xi)&=\stepex{0r+1}(\pd\itin(\xi,\zeta;\la))^{-1}
    \intertext{and combining \eqref{eq_differentialEquationForY0} and \eqref{eq_differentialEquationForYN} we have}
    \partial_\la\dleaf{\itin}{\eta}{\zeta}{\la}(\xi) &= \dot\phi(\xi)+\bigo{\la^{-1/2}}\label{e_parLaLeaf2}
  \end{align}
\end{subequations}
\begin{prp}\label{p_observationsOnLeaves2}
  For any $\xi\in\basec^*$, $\la\in\J_\Bcyl^\admissible$, $\zeta,\eta\in\basec^*$, the following properties are satisfied:
  \begin{subequations}
    \begin{align}
      \partial_\la\dleaf{\itin}{\eta}{\zeta}{\la}(\xi) &= 1+\bigo{\la^{-1/2}}\label{e_estimateLaLeaf2}\\
      \partial_\xi\ddleaf\itin\eta\zeta{\la}(\xi) &= -4\pi^2\onepbigo{\la^{-1/2}}\label{e_estimateXiLeaf2}
    \end{align}
  \end{subequations}
  Moreover $\dleaf{\itin}{\eta}{\zeta}\la$ has a unique critical point in $\basec^*$.
\end{prp}
The proof is identical to the proof of proposition~\ref{p_observationsOnLeaves}.
\begin{prp}\label{p_estimateOnPhi2}
  Let $\la_1,\la_2\in\J_{\itin}^\admissible$, $\eta_1,\,\eta_2,\,\zeta_1,\,\zeta_2, \xi\in\basec^*$; then we have:
  \begin{align*}
    \la_2|\ddleaf{\itin}{\eta_2}{\zeta_2}{\la_2}(\xi)| &<\la_1|\ddleaf{\itin}{\eta_1}{\zeta_1}{\la_1}(\xi)|\onepbigo{\laa^{-1}}+\\&\quad+\frac{1}{2}\ndLL{\itin}^{-1}+|\la_2-\la_1|\bigo{\laa^{-1/2}}
  \end{align*}
\end{prp}
\begin{proof}
  The proof is similar to the proof of proposition~\ref{p_estimateOnPhi}; first of all notice that:
  \[
  \la_1|\ddleaf{\itin}{\eta_1}{\zeta_1}{\la_1}(\xi)| = \la_2|\ddleaf\itin{\eta_1}{\zeta_1}{\la_1}(\xi)| + \frac{\la_1-\la_2}{\la_1}\la_1|\ddleaf\itin{\eta_1}{\zeta_1}{\la_1}(\xi)|;
  \]
  thus we only need to estimate the difference $|\ddleaf\itin{\eta_1}{\zeta_1}{\la_1}(\xi) - \ddleaf\itin{\eta_2}{\zeta_2}{\la_2}(\xi)|$:
  \begin{align*}
    |\ddleaf\itin{\eta_1}{\zeta_1}{\la_1}(\xi) - \ddleaf\itin{\eta_2}{\zeta_2}{\la_2}(\xi)| &<
    |\ddleaf\itin{\eta_1}{\zeta_1}{\la_1}(\xi) - \ddleaf\itin{\eta_2}{\zeta_2}{\la_1}(\xi)| +\\&+
    |\ddleaf\itin{\eta_2}{\zeta_2}{\la_1}(\xi) - \ddleaf\itin{\eta_2}{\zeta_2}{\la_2}(\xi)|;
  \end{align*}
  Using propositions~\ref{l_estimateReferenceSequence} and~\ref{l_coneInvariance} we obtain the bound
  \[
  |\ddleaf\itin{\eta_1}{\zeta_1}{\la_1}(\xi) -
  \ddleaf\itin{\eta_2}{\zeta_2}{\la_1}(\xi)|<\frac{1}{4}\ndLL{\itin}^{-1}\cdot\,\la_1^{-1};
  \]
  on the other hand, by proposition~\ref{l_partialHLa} we have that
  \[
  |\ddleaf\itin{\eta_2}{\zeta_2}{\la_1}(\xi) - \ddleaf\itin{\eta_2}{\zeta_2}{\la_2}(\xi)|<|\la_2-\la_1|\bigo{\la^{-3/2}};
  \]
  which concludes the proof.
\end{proof}

\begin{prp}\label{p_estimateItinDiameter}
  Fix $N>0$ and $\itin\in\cwords {N-1}$; then there exist intervals $\basec_\itin\subset\basec^*$ and $\J_\itin\subset\J_\itin^\admissible$ such that
  \begin{subequations}\label{e_estimateItinDiameter}
    \begin{align}
      \basec_\itin^\elliptic&\subset\basec_\itin &\diam\basec_\itin&<\const\,\laa^{-1}\ndLL\itin^{-1}\\
      \J_\itin^\elliptic&\subset\J_\itin &\diam\J_\itin&<\const\,\laa^{-1}\ndLL\itin^{-2};\label{e_estimateItinParameterDiameter}
    \end{align}
  \end{subequations}
\end{prp}
\begin{proof}
  The proof follows the same type of argument of the proof of proposition~\ref{p_decreasingBcylEstimate}. First of all notice that for all $\Bcyl\ni\itin$ we have by definition:
  \begin{align*}
    \basec^\elliptic_\itin&\subset\basec_\Bcyl &
    \J^\elliptic_\itin&\subset\J_\Bcyl;
  \end{align*}
  let then $\Bcyl^*$ be the bicylinder of rank $\max(N-2,0)$ containing $\itin$; we construct a decreasing sequence of sets:
  \begin{align*}
    \basec_{\Bcyl^*}&=\basec^0_\itin\supset\basec^1_\itin\supset\cdots\supset\basec^j_\itin\supset\cdots\\
    \J_{\Bcyl^*}\cap\J_\itin^\admissible&=\J_\itin^0\supset\J_\itin^1\supset\cdots\supset\J_\itin^j\supset\cdots
  \end{align*}
  satisfying the following properties
  \begin{enumerate}[(a')]
  \item $\basec_\itin^\elliptic\subset\basec_\itin^j$ and $\J_\itin^\elliptic\subset\J_\itin^j$ for all $j\geq 0$;
  \item the following estimates hold for $j>0$:
    \begin{subequations}
      \begin{align}
        \diam\basec_\itin^j&<\frac{1}{4\pi^2}\laa^{-1}(3\ndLL\itin^{-1}+\diam\J_\itin^{j-1})\onepbigo{\la^{-1/2}}\label{e_innerInductionBaseEstimate2} \\
        \diam\J_\itin^j&<\frac{3}{\pi^2}\laa^{-1}(4\ndLL\itin^{-1}+\diam\J_\itin^{j})^{2}
      \end{align}
    \end{subequations}
  \end{enumerate}
  Assume we constructed the sequences as prescribed; then we can conclude by taking:
  \begin{align*}
    \basec_\itin&=\bigcap_j\basec^j_\itin&
    \J_\itin&=\bigcap_j\J^j_\itin;
  \end{align*}
  in fact, as we did before, we can iterate the inductive estimate for the diameter of $\J$ and obtain for large enough $\la$ and $j$ that $\diam\J_\itin^j<40\laa^{-1}\ndLL\itin^{-2}$, from which both estimates \eqref{e_estimateItinDiameter} follow. We now describe the inductive construction, which is indeed very similar to the one given in stage 2.
  If $\J_\itin^j=\emptyset$, then as before we can set $\basei\itin^{j+1}=\emptyset$ and $\J_\itin^{j+1}=\emptyset$ and these will trivially satisfy items (a') and (b') since both $\basei\itin^\elliptic$ and $\J_\itin^\elliptic$ are empty. Assume then that $\J_\itin^j$ is not empty and fix $\la'\in\J_\itin^\admissible\cap\J_\itin^{j}$, $\eta',\zeta'\in\basec_\itin^j$ and define:
  \begin{align*}
    \basec_\itin^{j+1}&=\{\xi\in\basec_\itin^k\st\la'|\varphi_\itin(\xi,\eta',\zeta';\la')|<3\ndLL{\itin}^{-1}+\diam\J_\itin^j\}.\\
    \J_\itin^{j+1}&=\{\la\in\J_\itin^j\st\domain_\itin|_{\basei\itin^{j+1}}\cap\sm^{-N}\domain_\itin|_{\basei\itin^{j+1}}\not =\emptyset\}
  \end{align*}
  Using \eqref{e_estimateLaLeaf2} we can prove that $\basec_{\itin}^{j+1}$ is connected and satisfies \eqref{e_innerInductionBaseEstimate2}.

  \begin{prp}\label{p_pcritComparisonWords}
    We have $\basec_\itin^\elliptic\subset\basec_\itin^{j+1}$ and $\J_\itin^\elliptic\subset\J_\itin^{j+1}$; moreover if $\xi\in\basec_{\itin}^{j+1}$ then for all $\la\in\J_\itin^\admissible\cap\J_\itin^{j}$, $\eta,\zeta\in\basec_\itin^j(\la)$:
    \begin{equation}\label{e_estimateGenericXiEtaWords}
      \la|\varphi(\xi,\eta,\zeta;\la)|<4\ndLL\itin^{-1}+\onepbigo{\la^{-1/2}}\diam\J_\itin^j.
    \end{equation}
  \end{prp}
  \begin{proof}
    Proposition~\ref{p_estimateOnPhi2} implies \eqref{e_estimateGenericXiEtaWords} for large enough $\la$ by taking $\la_1=\la'$ and $\la_2=\la$ and using the definition of $\basec_{\itin}^{j+1}$.

    Now, if $\xi\in\basec_{\itin}^\elliptic$, there exists $\la\in\J_{\itin}^\elliptic\subset\J_{\itin}^j$, such that $p=\pd{\itin}(\xi,\xi;\la)=\qd{\itin}(\xi,\xi;\la)$ and
    \[
    \la|\varphi(\xi,\xi,\xi;\la)|<2\ndLL\itin^{-1}.
    \]
    Using once more proposition~\ref{p_estimateOnPhi2} we obtain that $\xi\in\basec_{\itin}^{j+1}$, but since $\xi$ was arbitrary this implies that $\basec_{\itin}^\elliptic\subset\basec_{\itin}^{j+1}$, which in turn yields $\J_{\itin}^\elliptic\subset\J_{\itin}^{j+1}$.
  \end{proof}
  \begin{prp}
    The following estimate holds:
    \[
    \diam\J_\itin^{j+1}<
    \frac{3}{\pi^2}\laa^{-1}(4\ndLL\itin^{-1}+\diam\J_\itin^{j})^{2}
    \]
  \end{prp}
  \begin{proof}
    If $\J_\itin^{j+1} = \emptyset$ the estimate is trivial; assume then there exists $\la^*\in\J_\itin^{j+1}$ and  $\xi^*,\eta^*,\zeta^*\in\basec_\itin^{j+1}$ such that
    \[
    \dleaf{\itin}{\eta^*}{\zeta^*}{\la^*}(\xi^*)=0;
    \]
    Introduce the quantity:
    \[
    d=\sup_{\la\in\J_\itin^{j+1}}\sup_{\eta,\zeta\in\basei\itin^{j+1}}\sup_{\xi\in\basei\itin^{{j+1}}}\dist(\dleaf{\itin}{\eta}{\zeta}{\la}{\xi},\dleaf{\itin}{\eta}{\zeta}{\la}{\xi^*});
    \]
    by definition we have:
    \[
    d\leq\diam\basei\itin^{j+1}\sup_{\la\in\J_\itin^{j+1}}\sup_{\eta,\zeta\in\basei\itin^{j+1}}\sup_{\xi\in\basei\itin^{{j+1}}}|\la\ddleaf\itin\zeta\eta\la(\xi)|;
    \]
    Proposition~\ref{p_pcritComparisonWords} and \eqref{e_innerInductionBaseEstimate2} then imply that:
    \[
    d\leq  \frac{1}{4\pi^2}\laa^{-1}(3\ndLL\itin^{-1}+\diam\J_\itin^{j})(4\ncLL{\itin}^{-1}+\diam\J_k^j)\onepbigo{\la^{-1/2}}
    \]
    Furthermore, estimates (\ref{e_baseEstimatesOnDomainLeaves}b-c) and \eqref{e_innerInductionBaseEstimate2} imply:
    \begin{align*}
      \sup_{\smash{
          \eta,\zeta\in\basec_\itin^{j+1}}}&
      \dist(\dleaf{\itin}\zeta\eta\la(\xi),\dleaf{\itin}{\zeta^*}{\eta^*}\la(\xi))
      \\&\quad< \frac{1}{2\pi^2}\laa^{-1}(3\ndLL\itin^{-1}+\diam\J_\itin^{j})\ndLL{\itin}^{-1}\onepbigo{\la^{-1/2}}
    \end{align*}
    Therefore
    \begin{align}
      \sup_{\eta,\zeta\in\basei\itin^{j+1}}\sup_{\xi\in\basei\itin^{{j+1}}}&\dist(\dleaf{\itin}{\eta}{\zeta}{\la}{\xi},\dleaf{\itin}{\eta^*}{\zeta^*}{\la}{\xi^*})\leq\notag\\&\leq
      \frac{3}{4\pi^2}\laa^{-1}(4\ndLL\itin^{-1}+\diam\J_\itin^{j})^{2}\label{e_estimateImageLeaf}
    \end{align}
    Using \eqref{e_estimateLaLeaf2} we hence obtain that if
    \[
    7/8 > |\la-\la^*| > \frac{3}{2\pi^2}\laa^{-1}(4\ndLL\itin^{-1}+\diam\J_\itin^{j})^{2}
    \]
    we have necessarily for any $\eta,\zeta,\xi\in\basei\itin^{j+1}$:
    \[
    \dleaf{\itin}{\zeta}\eta\la(\xi)\not = 0,
    \]
    which implies $\la\not\in\J_\itin^{j+1}$ and concludes the proof.
  \end{proof}
  Hence items (a') and (b') are proved and the proof of proposition~\ref{p_estimateItinDiameter} is also complete.
\end{proof}
Notice that \eqref{e_estimateItinParameterDiameter} implies item (a) of
lemma~\ref{l_estimateParameterSpace}. We now describe a construction that will be useful
in the proof of item (b) and in the proof of lemma~\ref{l_fatDensityComplete}. The main
idea is to notice that the appropriate iterate of $\sm$ behaves as the conservative
H\'enon map at the scale given by $\basei\itin$. This fact is indeed true for a generic
unfolding of quadratic homoclinic tangencies in a conservative settings, as proved in
\cite{MoraRomero} (see also \cite{Duarte99}); we will not, however, use any of these
previous results, as our geometrical analysis is sufficient to obtain all needed
information. Indeed, the geometry of the conservative H\'enon map will only provide us
with inspiration to finalize the study of the dynamics of the standard map at these
scales.

Let $\itin$ be a closed word and define the \emph{diagonal leaf function} for
$\xi\in\basec^*$ and $\la\in\J_\itin^\admissible$ as follows:
\[
\diago{\itin}{\la}(\xi)=\dleaf\itin\xi\xi\la(\xi)
\]
Using the definition and equations \eqref{e_baseEstimatesOnDomainLeaves} it is easy to compute the derivative:
\begin{align*}
  \partial_\xi\diago\itin\la(\xi)&=\la\ddleaf\itin\xi\xi\la(\xi)+2\stepex{0N}^{-1}(\pd{\itin}(\xi,\xi;\la))\\
  \partial_\la\diago\itin\la(\xi)&=\dot\phi(\xi)+\bigo{\la^{-1/2}}\label{e_parLaLeafDiago}.
\end{align*}
Assume now that $p=p_{\itin}(\xi,\xi;\la)$ is an $N$-periodic point; then, by definition,
\[
\diago\itin\la(\xi)=0;
\]
hence we can bound the number of elliptic $N$-periodic points contained in $\domain_\itin$ with the number of zeros (mod $1$) of $\diago\itin\la$ in $\basec_\itin$. If $\basec_\itin=\emptyset$ there are no parameters in $\J$ allowing for an elliptic periodic point, so item (b) of lemma~\ref{l_estimateParameterSpace} trivially holds. Assume then $\basei\itin\not=\emptyset$; then %
by definition and by the previous estimate we have that $\gamma(\xi)$ can have at most one quadratic critical point in $\basec_\itin$; moreover, by \eqref{e_estimateImageLeaf} we also have that the image of $\diago\itin\la$ has small diameter and hence cannot wrap around the torus $\torus$. This immediately implies that we can have at most two elliptic periodic points in $\domain_\itin$: one for each sign of the derivative $\partial_\xi\diago\itin\la$.
However, by definition of $\diago\itin\la$ we have that $\pd\itin(\xi,\xi;\la)$ is elliptic if and only if the trace condition \eqref{eq_traceCondition} holds, i.e.:
\begin{equation}\label{e_ellipticDiagoCondition}
  -4\stepex{0N}^{-1}(\pd{\itin}(\xi,\xi;\la)) \leq \partial_\xi\diago\itin\la(\xi) \leq 0;
\end{equation}
On the other hand if $\diago\itin\la$ has two zeros in $\basec_\itin$, then necessarily one of them, that we denote by $\xi^*$, is such that $\partial_\xi\diago\itin\la(\xi^*)>0$; hence it cannot be elliptic, which finally implies that there can be only one elliptic cyclicity one $N$-periodic point in a given domain $\domain_\itin$ and proves item (b) of lemma~\ref{l_estimateParameterSpace}.

The diagonal leaf function is indeed a very powerful object as the following proposition shows:
\begin{prp}\label{p_lowerEstimate}
  Assume there exist $\la_1<\la_2\in\J_\itin^{\admissible}$ such that for all $\xi\in\basei\itin$:
  \[
  \diago\itin{\la_1}(\xi) < 0 <\diago\itin{\la_2}(\xi)
  \]
  then there exists an interval $\hat\J_\itin^\elliptic\subset(\la_1,\la_2)$ of diameter at least $(1/4\pi^2)\laa^{-1}\scLL\itin^{-2}$ such that for $\la\in\hat\J_\itin^\elliptic$ there exists a cyclicity 1 elliptic periodic point in $\domain_\itin$ that is the center of an elliptic island.
\end{prp}
\begin{proof}
  By the intermediate value theorem we know there exists a $\la'\in(\la_1,\la_2)$ such that for some $\xi_\zero(\la')\in\basei\itin$ the function $\diago\itin\la(\xi_\zero(\la'))=0$; let us write a differential equation for the function $\xi_\zero(\la)$ such that $\diago\itin\la(\xi_\zero\la))=0$:
  \begin{align}
    \deh\diago\itin\la(\xi) &=\partial_\xi\diago\itin\la(\xi)\deh\xi + \partial_\la\diago\itin\la(\xi)\deh\la=\notag\\
    &=(\la\ddleaf\itin{\xi_\zero}{\xi_\zero}\la(\xi_\zero)+2\stepex{0N}^{-1}(\pd{\itin}(\xi_\zero,\xi_\zero;\la)))\deh\xi_\zero+\\&+\onepbigo{\la^{-1/2}}\deh\la = 0\label{e_odeForXi}
  \end{align}
  If we want the periodic point to be elliptic we need to satisfy condition \eqref{e_ellipticDiagoCondition}; moreover, to prove that around the periodic point we have an elliptic island we need to avoid resonances of order $4$ and show that the Birkhoff normal form of $\sm$ around $p$ is non-degenerate.
  The non-resonance condition can be ensured by simply restricting the allowed values for the multiplier $\exp(i\vartheta)$; we take $0<\vartheta<\pi/3$; hence it is sufficient to modify \eqref{e_ellipticDiagoCondition} to the following condition:
  \begin{equation}\label{e_ellipticDiagoConditionNonResonant}
    -\stepex{0N}^{-1}(\pd{\itin}(\xi,\xi;\la)) < \partial_\xi\diago\itin\la(\xi) < 0.
  \end{equation}
  We will check non-degeneracy of the Birkhoff normal form in appendix~\ref{appendixBNF}.
  We will now study the evolution of the multiplier with $\la$:
  \begin{align*}
    \de{}{\la}\partial_\xi\diago\itin\la(\xi_\zero(\la))&=
    \de{\xi_\zero}{\la}\partial_{\xi}(\la\ddleaf\itin{\xi}{\xi}\la(\xi)+2\stepex{0N}^{-1}(\pd{\itin}(\xi,\xi;\la)))+\\
    &+\partial_{\la}(\la\ddleaf\itin{\xi}{\xi}\la(\xi)+2\stepex{0N}^{-1}(\pd{\itin}(\xi,\xi;\la)))
  \end{align*}
  By equations (\ref{e_baseEstimatesOnDomainLeaves}bc), proposition~\ref{l_partialHxi}  and \eqref{e_estimateXiLeaf2} we have that:
  \[
  \partial_{\xi}(\la\ddleaf\itin{\xi}{\xi}\la(\xi)+2\stepex{0N}^{-1}(\pd{\itin}(\xi,\xi;\la)))= -4\pi^2\la\onepbigo{\la^{-1/2}}
  \]
  and by \eqref{e_partialStepexLa} and proposition~\ref{l_partialHLa} we have:
  \[
  \partial_{\la}(\la\ddleaf\itin{\xi}{\xi}\la(\xi)+2\stepex{0N}^{-1}(\pd{\itin}(\xi,\xi;\la)))=\ddleaf\itin{\xi}{\xi}\la(\xi)+\bigo{\la^{-1/2}}
  \]
  which is small since $\ddleaf\itin{\xi}{\xi}\la(\xi)$ is small if $\xi\in\basei\itin$. Hence we have:
  \[
  \de{}{\la}\partial_\xi\diago\itin\la(\xi_\zero(\la))=-4\pi^2\la\onepbigo{\la^{-1/2}}\de{\xi_\zero}{\la}.
  \]
  Using \eqref{e_odeForXi} and letting $q=\partial_\xi\diago\itin\la(\xi_\zero)$ we thus obtain:
  \[
  \left|\int_{-\scLL\itin^{-1}}^{0} q\deh q\right|\leq     \left|4\pi^2\la\int_{\hat\J_\itin^\elliptic}\deh\la\right|
  \]
  which implies that the diameter of $\hat\J$ is bounded below by $(1/4\pi^2)\la^{-1}\scLL\itin^{-2}$ and concludes the proof.
\end{proof}
We are now ready to give the
\begin{proof}[Proof of lemma~\ref{l_fatDensityComplete}]
  We can always assume that $\mathcal{B}$ has diameter less than $1/4$; otherwise we can simply consider a subset of $\mathcal{B}$ satisfying this property. Then by construction there exists a $\J(n)$ such that $\mathcal{B}\subset\hat\J(n)$; fix $\J=\J(n)$. Then by corollary~\ref{c_decreasingSequence} there exists a decreasing sequence $\eps_N$ such that if $\mathcal{B}$ has diameter larger than $\eps_N$ it will contain the core parameter set of some $r_1$-th generation balanced bicylinder $\tBcyl=\tfcyl\cap\tbcyl$; where $r_1=(N-\lfloor\log N\rfloor-5) +1$. Let
  \begin{align*}
    \tfcyl&=-\letter_1\cdots\letter_{r_1} &
    \tbcyl&=\letter_{-r_1}\cdots\letter_{-1}-
  \end{align*}
  and assume
  \[
  \itin=\altLetter_{-r}\cdots\altLetter_{r}
  \]
  where by hypothesis $r\leq\lfloor \log N\rfloor$. Then we claim that we can define the following $\mathcal{B}$-admissible closed word of rank $N-1$:
  \[
  \itin^*=-\letter_1\cdots\letter_{r_1} ****\, \altLetter_{-r}\cdots\altLetter_{r} ****\, \letter_{-r_1}\cdots\letter_{-1}-
  \]
  where by $*$ we denote some arbitrary admissible symbols that make $\itin^*$ admissible.
  In fact such padding symbols exist since, by the compatibility condition \eqref{e_compatibilityCondition} we can connect two arbitrary symbols using at most 5 other symbols; since $N-(r+r_1)\geq 5$ we can conclude. Then if we can prove that $\itin^*$ satisfies the hypotheses of proposition~\ref{p_lowerEstimate} we would have that there exists an interval $\hat\J_{\itin^*}^\elliptic\subset\hat\J_{\tBcyl}\subset\mathcal{B}$ of parameter values such that $\sm$ has an elliptic periodic point in $\hypset_\itin$ that is the center of an elliptic island. Moreover the diameter of $\hat\J_{\itin^*}^\elliptic$ can be estimated as follows by proposition~\ref{p_lowerEstimate}:
  \[
  \diam\hat\J_{\itin^*}^\elliptic>\const\,\sdLL{\itin^*}^{-2}>\const\,\la^{-2N}
  \]
  where the constant can contain terms such as $\la^{\bigo{1}}$, but does not depend on $N$.
  Hence we just need to prove that $\itin^*$ satisfies the hypotheses of proposition~\ref{p_lowerEstimate}; this in turn is clear, since by construction we have that the function $\diago{\itin^*}{\la}$ is $\ncLL{\tBcyl}^{-1}$-close to $\dcleaf{\tBcyl}{\la}$. On the other hand, by our construction, we know that we can take $\tBcyl=\tBcyl_{(\Bcyl:0)}$ for some other balanced bicylinder $\Bcyl$. Therefore we can take $\la_1=\la_\Bcyl - 1/2\ncLLi{\Bcyl}$ and $\la_2=\la_\Bcyl + 1/2\ncLLi{\Bcyl}$, in such a way that
  \begin{align*}
    \dcleaf{\tBcyl}{\la_1}&=-\frac{1}{2}\ncLLi{\Bcyl}\onepbigo{\la^{-1/2}}&
    \dcleaf{\tBcyl}{\la_2}&=\frac{1}{2}\ncLLi{\Bcyl}\onepbigo{\la^{-1/2}};
  \end{align*}
  hence we conclude that $\la_1$ and $\la_2$ satisfy the hypotheses of lemma~\ref{p_lowerEstimate}, which conclude the proof.
\end{proof}

%% file: appendix.tex
\appendix
\input{riemannlemma.p}
\newcommand{\I}{\mathcal{I}}

\section{Technical appendix}
\subsection{Proof of lemma~\ref{l_technicalSummation}}\label{appendixRL}
\begin{proof}
The proof relies on considering the sum as a Riemann Sum and approximating it with a Lebesgue-type argument; 
fix a large integer $N>0$ and consider any decreasing sequence $1=\cursor_0>\cursor_1>\cdots>\cursor_N=0$, so that:
\[
\emptyset=\indexset_{\cursor_0}\subset\indexset_{\cursor_1}\subset\cdots\subset\indexset_{\cursor_N}=\indexset
\]
Define $\tilde{\indexset}_j=\indexset_{\cursor_j}\setminus\indexset_{\cursor_{j-1}}$ for $j=1,\cdots,N$; thus we obtain 
\begin{align*}
\rslc_1 \spread^{1-\cursor_j} &< \RiemFunc_i \leq \rslc_1 \spread^{1-\cursor_{j-1}}\ \textrm{for}\ i\in\tilde{\indexset}_j,& \sum_{i\in\tilde{\indexset}_j}\delta_i&<\rslc_2 \spread^{\expo\cursor_j},
\end{align*}
moreover:
\[
\sum_{i\in\indexset}\RiemFunc_i\delta_i=\sum_{j=1}^N\sum_{i\in\tilde{\indexset}_j}\RiemFunc_i\delta_i<\sum_{j=1}^N \rslc_1\rslc_2\spread^{1-\cursor_{j-1}+\expo\cursor_j}.
\]
If $\expo<1$ choose $\cursor_j$ satisfying the following relations:
\[
\cursor_j=\frac{\expo^{j-N}-1}{\expo(\expo^{j-N-1}-1)}\cursor_{j-1},
\]
in such a way that $\cursor_{j-1}-\expo\cursor_j=\cursor_j-\expo\cursor_{j+1}$; we therefore obtain:
\[
1-\cursor_0+\expo\cursor_1=\frac{\expo^{1-N}-1}{\expo^{-N}-1}=\expo-\eps,
\]
with $\eps\sim 1/N$ which can therefore be taken arbitrarily small.\\
The case $\expo=1$ can be obtained as a limit for $\expo\to 1$; in this setting we choose $\cursor_j$ as follows:
\[
\cursor_j=1-\frac{j}{N}
\]the fast decreasing
hence, again we obtain $\cursor_{j-1}-\expo\cursor_j=\cursor_{j}-\expo\cursor_{j+1}=1/N$, which implies
\[
1-\cursor_0+\expo\cursor_1 = 1 - 1/N = \expo-\eps.
\]
\end{proof}
\subsection{Proof of lemma~\ref{l_technicalHD}}\label{appendixHD}
For any fixed $m$, let $\I_m\subset\hat\J$ be an interval of diameter $\eps'_m$; we now describe an inductive procedure to construct a decreasing sequence of sets:
  \[
  \I_m=\I_m^0\supset\I_m^1\supset\cdots\supset\I_m^n\supset\cdots;
  \]
  such that each $\I_m^n$ is the disjoint union of $L_m^n$ intervals of same length $\delta_m^{n}=\eps'_{n+m}$, i.e
  \begin{align*}
    \I_m^n&=\bigsqcup_{j=1}^{L_m^{n}}\I_{m}^{n,j}, &
    |\I_m^{n,j}|&=\delta_m^{n}.
  \end{align*}
Let $l_m^n=\lfloor\delta_{m}^{n}/\eps_{m}^{n+1}\rfloor$; partition each interval $\I_{m}^{n,j}$ in $l_m^n$ sub-intervals of equal length; by construction their diameter is not smaller than $\eps_{m}^{n+1}$, hence, for each of them, there exists smaller sub-intervals of diameter $\delta^{n+1}_{m}$ for which property $P_{m+n+1}$ holds. We define $\I_{m}^{n+1}$ be the union of these $L_{m}^{n+1}=L_m^{n}\cdot l_m^{n}$ smaller sub-intervals.
  Let
  \[
\bar\I_m=\bigcap_n\I_m^n,
  \]
  Then, by construction, any $\la\in\bar\I_m$ satisfies properties $P_k$ for all $k>m$. We
  thus construct $\fatset$ as follows: for each $m$, partition of $\hat\J$ in a number of
  intervals of diameter $\eps'_{m}$, which we call $\{\I_{m,j}\}$, and, possibly, a
  leftover interval $\I_{m,*}$ of diameter smaller than $\eps'_m$. Then by the above
  construction, for each $\I_{m,j}$ we obtain a set $\bar\I_{m,j}$ such that any
  $\la\in\bar\I_{m,j}$ satisfies properties $P_k$ for all $k>m$.  Let
  $\fatset_m=\bigcup_j\bar\I_{m,j}$ and $\fatset=\bigcup_m\fatset_m$; then $\fatset$ is,
  by construction, a residual set, and to conclude the proof, we only need to prove that
  the Hausdorff dimension of $\fatset$ satisfies the lower bound \eqref{e_lowerBoundHD}.
  We will in fact prove that, for any fixed $m$, $\bar\I_m$ satisfies the same bound;
  since $m$ is now fixed, let us drop it from the notations; also, we will slightly abuse
  the notation by redefining $\eps_n$ and $\eps'_n$ to be $\eps_{m+n}$ and $\eps'_{m+n}$,
  respectively. First of all let
  \[
  s=\liminf_{n\to\infty}\frac{\log\eps'_n}{\log\eps_n};
  \]
  if $s=0$, then \eqref{e_lowerBoundHD} is trivial, hence we can assume $s>0$. We will use the following standard result\footnote{The proof can be found e.g.\ in \cite{Falconer}}
  \begin{lem}
    Assume there exists a Borel probability $\mu$ on a metric space $X$ and positive real numbers $C$, $s$ such that for all sufficiently small balls $B$ we have:
    \begin{equation}\label{e_conditionHD}
      \mu(B)<C\diam B^{s},
    \end{equation}
    then $\dim_H X\geq s$.
  \end{lem}
Define a sequence of positive functionals $\Phi_n$ acting on a continuous function $\varphi$ on $\bar\I$ by means of the following expression:
  \[
  \Phi_n(\varphi)= L_n^{-1} \sum_{\bar\I^{n,j}}\varphi(x_{\bar\I^{n,j}}),
  \]
  where $x_{\bar\I^{n,j}}\in\bar\I^{n,j}$ can be chosen arbitrarily. Since $\bar\I$ is compact, $\varphi$ is uniformly continuous and it is easy to prove that for any $\varphi$ and $m>n$ we have $|\Phi_n(\varphi)-\Phi_m(\varphi)|\to 0$ as $n\to\infty$. In fact for fixed $\varphi$ and $\eps$ consider the $\delta$ given by uniform continuity and let $n$ be large enough so that $\delta_n<\delta$. Therefore
  \begin{align*}
    \Phi_m(\varphi)&=L_m^{-1}\sum_{\bar\I^{n,j}}\sum_{\bar\I^{m,k}\subset\bar\I^{n,j}}\varphi(x_{\bar\I^{m,k}})=\\
    &=L_n^{-1}\sum_{\bar\I^{n,j}}\frac{L_n}{L_m}\sum_{\bar\I^{m,k}\subset\bar\I^{n,j}}(\varphi(x_{\bar\I^{n,j}})-\varphi(x_{\bar\I^{n,j}})+\varphi(x_{\bar\I^{m,k}}))\\
    &=L_n^{-1}\sum_{\bar\I^{n,j}}\varphi(x_{\bar\I^{n,j}})+\bigo{\eps}=\Phi_n(\varphi)+\bigo{\eps}.
  \end{align*}
  Hence $\Phi_n$ converges to some probability measure (since $\Phi_n(1)=1$ for all $n$) that we call $\mu$; we now claim that $\mu$ satisfies \eqref{e_conditionHD}. First of all notice that, for any $n\in\naturals$ the following bound holds:
  \begin{align*}
    \mu(B) &= L_n^{-1}\sum_{\bar\I^{n,j}}\lim_{m\to\infty}\frac{L_n}{L_m}\sum_{\bar\I^{m,k}\subset\bar\I^{n,j}}1_B(x_{\bar\I^{m,k}})\leq \\
    &\leq L_n^{-1}I_n(B),
  \end{align*}
  where $I_n=\sum_{\bar\I^{n,j}\cap B\not = \emptyset} 1$ is the number of intervals in $\bar\I^n$ intersecting $B$.
  First observe that:
  \begin{align*}
    -\log L_k
    &= -\sum_{j=0}^{k-1} \log l_j =\sum_{j=0}^{n-1} (\log\eps_{j+1}-\log\eps'_j+o(1))=\\
    &= -\log\eps'_0 - \sum(\log\eps'_j(1-s_j))+s_k\log\eps'_k+o(k)\\
    &\defeq s_k'\log\eps'_k
  \end{align*}
  where  $s_k\log\eps'_k=\log\eps_k$ and $s_k-C/k\leq s'_k \leq s_k$ for some fixed $C$.
  Define $S_n=\inf_{m\geq n}s_k$, let $\rho=\diam B$ and choose $n$ such that $\eps'_n\leq\rho<\eps'_{n-1}$; we assume $\rho$ to be small enough so that for the corresponding $n$ the bound $S_n>C/n$ is satisfied. We then need to estimate $I_n(B)$: recall that by construction, each interval of $\bar\I^n$ is obtained by subdividing each interval of $\bar\I^{n-1}$ in subintervals of diameter larger than $\eps_n$; any ball $B$ of diameter $\rho$ can then intersect at most $\rho/\eps_n+3$ such subintervals; on the other hand each subinterval can contain only one component of $\bar\I^n$, therefore we conclude:
  \[
  I_n(B)\leq\frac{\rho}{(\eps'_n)^{s_n}}+3.
  \]
  Thus we obtain the estimate:
  \[
  \mu(B)\leq L_n^{-1}I_n(B)\leq (\eps'_n)^{s'_n-s_n}\rho+3(\eps'_n)^{s'_n}\leq \rho^{1-\sigma_n}+3\rho^{s_n-\sigma_n}
  \]
  where $\sigma_n=s_n-s'_n<C/n$; by the previous assumptions on $n$ we then have:
  \[
  \mu(B)<\const\,\rho^{S_n-C/n}.
  \]
  From which we can conclude since we can take $n$ arbitrarily large provided that we restrict to sufficiently small balls $B$.
\qed
\subsection{Birkhoff Normal Form for cyclicity one elliptic periodic points}\label{appendixBNF}
Fix $\itin$ of rank $N-1$, let $p_0$ be the elliptic periodic point of period $N$ and consider the coordinates given by $\fhomeo_r(p)=(\xi,\eta)$ in a neighborhood of $p_0$; with a slight abuse of notation we will assume that $p_0=\pd\itin(\xi=0,\eta=0)$. It is immediate to write the differential of $\sm^N$ in these coordinates:
\[
\deh\sm^N=\matrixtt{0}{1}{-1}{2c(\xi,\eta)}
\]
where $2c(\xi,\eta)=\la\stepex{0N}(\pd\itin(\xi,\eta)(\slope{N}(\qd\itin(\xi,\eta))-\slope{-N}(\pd\itin(\xi,\eta)))$; for ease of notation we let $c_0=c(0,0)$. Note that in these coordinates the differential is, apart from a scaling factor, essentially the same as the one for the conservative Hénon mapping. The proof of the twist condition is also virtually the same, modulo some small terms.

Denote the multiplier at the periodic point by $\lambda=\exp(i\vartheta)$, i.e.: $\lambda=c_0+i(1-c_0^2)^{1/2}$; we can thus diagonalize in $p_0$ taking the following complex coordinates:
\begin{align*}
z&=\xi+\lambda\eta&\bar z&=\xi+\bar\lambda\eta
\end{align*}
so that we obtain:
\begin{equation}\label{e_complexMap}
z\mapsto \lambda z - 2\tilde c(z,\bar z)\frac{\lambda z + \bar\lambda \bar z}{\bar\lambda-\lambda}
\end{equation}
where $\tilde c = c - c_0$. We should now prove the twist condition for the map \eqref{e_complexMap}; in particular, following \cite{Lazutkin} we shall write:
\[
z\mapsto \lambda z + A_3z^2 + A_4z\bar z + A_5\bar z^2 + A_6z^3 + A_7z^2\bar z + A_8z\bar z ^2 + A_9\bar z^3+\bigo{|z|^4}
\]
and prove that the following quantity is nonzero:
\begin{equation}\label{e_twistCondition}
\Upsilon=\textrm{Im}(\bar\lambda A_7) + 3|A_3|^2\textrm{ctg}(\vartheta/2)+|A_5|^2\textrm{ctg}(3\vartheta/2)\not =0
\end{equation}
We first of all compute the derivatives $\partial_\xi c$ and $\partial_\eta c$; using the definition of $c$, propositions~\ref{l_partialHxi},~\ref{l_estimateReferenceSequence} and equations \eqref{e_oneDefinitions} we immediately obtain:
\begin{align*}
  \partial_\xi c &= \la\stepex{0N}(\pd\itin(\xi,\eta)\bigo{\la^{-1/2}}& \partial_\eta c &= \la\stepex{0N}\left(\partial_{\slope{1}}\slope{1}(\qd\itin(\xi,\eta))+\bigo{\la^{-1/2}}\right)
\end{align*}
and by our choice of $\phi$, since $\ddddot\phi$ is small in $\pcrit$, we have
\[
\partial^2_{(\xi,\eta)} c = \la\stepex{0N}(\pd\itin(\xi,\eta)\bigo{\la^{-1/2}}.
\]
Hence, we have, by the chain rule:
\begin{align*}
  \partial_z c&=\frac{1}{\bar\lambda-\lambda}(\bar\lambda\partial_\xi c - \lambda\partial_\eta c)=  \frac{-\lambda}{\bar\lambda-\lambda}\la\stepex{0N}(\pd\itin(\xi,\eta))\onepbigo{\la^{-1/2}}\\
  \partial_{\bar z} c&=\frac{1}{\bar\lambda-\lambda}( -\partial_\xi c + \partial_\eta c)=  \frac{1}{\bar\lambda-\lambda}\la\stepex{0N}(\pd\itin(\xi,\eta))\onepbigo{\la^{-1/2}}
\end{align*}
which imply:
\begin{align*}
  A_3&=\frac{\lambda^2}{(\bar\lambda-\lambda)^2}\la\stepex{0N}(\pd\itin(\xi,\eta))\onepbigo{\la^{-1/2}}\\
  A_5&=\frac{-\bar\lambda}{(\bar\lambda-\lambda)^2}\la\stepex{0N}(\pd\itin(\xi,\eta))\onepbigo{\la^{-1/2}}\\
\end{align*}
and
\[
A_7=\la\stepex{0N}(\pd\itin(\xi,\eta))\bigo{\la^{-1/2}},
\]
which can therefore be neglected in the computation of $\Upsilon$.
Since we have chosen $0<\vartheta<\pi/3$ we have that the cotangent functions in \eqref{e_twistCondition} are both positive, hence $\Upsilon$ is positive and the twist condition holds.
 Notice, that for generic $\phi$ the second derivatives of $c$ could in principle be of the same order as the first derivatives; however, in expression \eqref{e_twistCondition} the second order terms $A_3$ and $A_5$ appear squared, hence the contribution of the term containing $A_7$ could still be neglected and we again obtain the twist condition.
